\documentclass[12
pt,twoside]{amsart}
\usepackage{amssymb}
\usepackage{a4wide}
\usepackage[utf8]{inputenc}
\usepackage[T1]{fontenc}
\usepackage{times}
\usepackage{ae}
\usepackage{hyperref}

\def\dhash{{d^\#}}
\def\ddhash{{dd^\#}}
\def\w{\wedge }

\def\hpq0{h^{p,q}_{\leq 0}}
\def\Hpq0{\H_{\leq 0}^{p,q}}

\def\dbar{\bar\partial}

\def\R{{\mathbb R}}

\def\C{{\mathbb C}}

\def\A{{\mathcal A}}

\def\K{{\mathcal K}}

\def\F{{\mathcal F}}

\def\H{{\mathcal H}}
\def\E{{\mathcal E}}

\def\L{{\mathcal L}}
\def\Q{{\mathcal Q}}

\def\L{{\mathcal L}}

\def\be{\begin{equation}}
\def\ee{\end{equation}}

\newtheorem{thm}{Theorem}[section]
\newtheorem{lma}[thm]{Lemma}
\newtheorem{cor}[thm]{Corollary}
\newtheorem{prop}[thm]{Proposition}

\theoremstyle{definition}

\newtheorem{df}[thm]{Definition}

\newtheorem{remark}[thm]{Remark}

\numberwithin{equation}{section}

\begin{document}

\title[]
{ Superforms, supercurrents and convex geometry. }

\author[]{ Bo Berndtsson}

\bigskip

\maketitle

\begin{abstract} We develop the calculus of superforms as a tool for convex geometry. The formalism is applied  to valuations on convex bodies,  the Alexandrov-Fenchel inequalities and Monge- Amp\`ere equations on the boundary of convex bodies.
\end{abstract}

\tableofcontents

\section{Introduction}

The calculus of superforms and supercurrents was introduced by Lagerberg in \cite{Lagerberg}, after a brief suggestion in \cite{Berndtsson1}. The main motivation in \cite{Lagerberg} was applications to tropical geometry and Lagerberg proved in particular that tropical varieties in $\R^n$ can be described as supercurrents of a certain type. After pioneering work in \cite{Chambert-Loir}, this point of view has  also been used  to develop a potential theory on Berkovich spaces; see  \cite{Gubleretal} for recent work and further references. In another direction, in \cite{Berndtsson2} we applied the same formalism to (extrinsic) Riemannian geometry, and used it, e.g. to give volume estimates for minimal surfaces  and a proof of Weyl's tube formula. In this paper we shall explore yet another setting  and apply the formalism to convex geometry. This has already been done in \cite{Larsson}, which gives an account of Alexandrov's second proof of the Alexandrov-Fenchel theorem in terms of superforms, and part of our work here builds on, and develops,  this. Our main objective is not to prove new results -- although there are some, and we will list them at the end of this introduction -- but to provide a convenient framework in which many known results can be formulated in a seemingly natural way.  

The basic idea of superforms and -currents  is to use a formalism inspired by K\"ahler geometry for problems in real geometry. We will give a very brief outline of the theory in the next section and refer to \cite{Lagerberg} and \cite{Berndtsson2} for a more precise account.(It should be said, however, that the setting in \cite{Lagerberg} is somewhat different, but it is nevertheless close enough to give an almost isomorphic theory (see \cite{Rubinstein})). A crucial concept in K\"ahler geometry is positivity of forms and currents and we  dedicate section \ref{positivity} to developing the analogous theory in the real setting. This has already been done in \cite{Lagerberg}, and been much further developed in \cite{Gubleretal} in the setting of toric varieties, but we shall go through this from scratch, highlighting   some concepts that are fundamental for the rest of the paper. The main difference between the complex and the real case is that in the complex case, the concept of weak positivity is central and serves as a hypothesis for a number of basic results. This is a bit more complicated  in the real setting (see \cite{Gubleretal}) since forms can be weakly positive and weakly negative at the same time; we call such forms weakly null. In \cite{Lagerberg} and \cite{Gubleretal} the authors avoid this difficulty by working with a stronger notion of positivity, but for our applications (especially to valuations) we need to consider weak positivity.  Nevertheless, the problem can be overcome, by restricting to a subclass of weakly positive forms, `orthogonal' to the weakly null forms. We call such forms `strong', since they are precisely the forms that can be written as the difference between two strongly positive forms. Fortunately, and somewhat miraculously, the forms that appear in our applications to valuations satisfy this orthogonality condition (see Theorem \ref{theoremstrongform}). 

Once this is done we discuss, developing further \cite{Lagerberg}, the analog of the theory of Bedford and Taylor of wedge products of closed positive currents. The conclusion is the same as in the complex case; a weakly positive closed current $\Omega$ can be multiplied by $\ddhash\phi$ for any convex function $\phi$, {\it provided that} $\Omega$ is `strong'. ( The condition that $\phi$ be locally bounded from Bedford-Taylor theory is automatically satisfied when $\phi$ is convex.) In the next section we use this to define the  Monge- Amp\`ere measure of a convex function, and indicate why our definition gives the same measure as Alexandrov's. This section ends by a fundamental estimate for the  Monge- Amp\`ere mass (Theorem \ref{MAestimate})  of functions of linear growth. After that we introduce the notion of mixed volumes in the next section.

In Section \ref{sectionhomform} we introduce the notion of strongly homogeneous forms and currents. We show that a current of bidegree $(1,1)$ is strongly homogeneous if and only if it can be written as $\ddhash\phi$ where $\phi$ is homogeneous of order 1. The corresponding notion in bidegree $(p,p)$ will play a crucial role in the paper, and is further discussed in the following four sections. It turns out that, if we identify $\R^n$ with the affine hyperplane $H_0=\{(x_0,x)\in \R^{n+1}; x_0=1\}$, then the  `strong' and closed $(p,p)$-currents  on $\R^n=H_0$ are precisely the restrictions of strongly homogeneous currents on $\R^{n+1}_+:=\{(x_0,x)\in \R^{n+1}; x_0>0\}$. This fact, which generalizes the fact that any convex function on $\R^n$ can be extended to a 1-homogeneous convex function on $\R^{n+1}_+$, is one of the reasons for the importance of  strongly homogeneous currents.

In section \ref{Asdiffoperator} we start our discussion of the Alexandrov-Fenchel theorem by introducing Alexandrov's differential operator on the sphere using our formalism (this follows and extends the work of Larsson, \cite{Larsson}), and in the following section we use it to give Alexandrov's proof of the Alexandrov-Fenchel theorem.   Here the superformalism allows a minor simplification of Alexandrov's proof. 

After that, we give another proof of the Alexandrov-Fenchel theorem, based on an idea of Gromov, \cite{Gromov},  to prove the Khovanskii-Teissier theorem from algebraic geometry. Gromov also sketches how to get the Alexandrov-Fenchel 
theorem from the Khovanskii-Teissier theorem, by using the theory of toric varieties. Our main point here is to avoid the last part and substitute for toric varieties a compactification of $\R^n$ as a real manifold, by adding a sphere at infinity. While the proof in the previous section uses Alexandrov's differential operator on the sphere, this alternative proof is based on a study of the Dirichlet problem for the same operator on a half-sphere.

 In this connection one should also mention the work of Wang, \cite{Wang}, that also eliminates the use of toric compactifications by working with complete K\"ahler metrics on the complex torus instead. 

The work in section \ref{positivity}, was carried out partly with a view towards applications to translation invariant valuations on the space of convex bodies, $\K$. A valuation on $\K$ is function defined on $\K$ such that
$$
\theta(K\cup L)+\theta(K\cap L)=\theta(K)+\theta(L),
$$
if $K,L\in \K$, {\it provided that  their union is also convex}. Our basic observation here is that the space  $\K$, modulo translation, is in 1-1 correspondence with the class
of positive and strongly homogeneous $(1,1)$-currents,  via the map
$$
K\to \ddhash h_K:=\omega_K,
$$
where $h_K$ is the support function of $K$. Therefore, any translation invariant valuation (or, for that matter, any translation invariant function on $\K$ whatsoever) must be a function  of $\omega_K$.

In section \ref{valuations} we give a representation formula for  two classes of  valuations in terms of supercurrents. We first show that there is a one-to-one correspondence between `smooth' valuations ( introduced by Alesker in   
 \cite{2Alesker}) and  smooth strongly homogeneous forms, $\Omega$, on $\R^n$, via the formula (Theorem \ref{somevaluations})
 $$
 \theta(K)=\F(\Omega):=\int e^{\omega_K}\w\Omega.
 $$
 Here $e^{\omega_K}=\sum_0^n \omega_K^k/k!$ is defined by power series expansion.
 The map $K\to e^{\omega_K}$ can be viewed as a current-valued character on $\K$ equipped with Minkowski addition (it satisfies $e^{\omega_{K+L}}=e^{\omega_K}\w e^{\omega_L}$) and we argue that $\F$ is analogous to Fourier transformation.

  This part should be new, but we stress that it is based on Alesker's work, as amplified by Knoerr, \cite{Knoerr} and van Handel (unpublished). It is also closely related to a description of  smooth valuations as integrals of a differential form over the `normal cycle' of a convex body (cf. \cite{Alesker-Fu}), and we give a translation between the two approaches in Section \ref{normalcycle} -- as far as smooth valuations are concerned they are equivalent. In our view, one advantage with the superformalism is that it shows a close analogy between the theory of valuations and intersection products in K\"ahler geometry; another that it adapts in a natural way to non-smooth valuations. After this, we characterize monotonically increasing valuations; they correspond precisely to  weakly positive strongly homogeneous currents. For smooth valuations, related  results in terms of the normal cycle  have been obtained previously by Bernig and Fu, \cite{Bernig-Fu2}, and the case of general valuations that are homogeneous of degree 1 and $(n-1)$ respectively is due to McMullen, \cite{2McMullen}, and Firey, \cite{Firey}. 

Finally, in section \ref{Minkowski} we discuss Minkowski's surface area measure, and relate it to the Monge- Amp\`ere equation on the boundary of convex bodies. The case of the standard sphere is Minkowski's original setting.

Since many of the results in this paper are not new, and in fact often classical, we end the introduction by pointing out what we think {\it is} new. On the technical side, this is first the general set-up in terms of superforms, then the use and characterization of `strong' forms, Theorem \ref{theoremstrongform}. The introduction of `strongly homogeneous forms' and their basic properties are also new. On the more conceptual side, the rewrite (\ref{incarnation1}) of Alexandrov's differential operator as the current-valued operator $\A$,  has some claims to originality. This also leads to a minor simplification of the central step in Alexandrov's proof of the Alexandrov-Fenchel theorem, see the proof of Theorem \ref{Alexandrov}. The study of the Dirichlet problem on the half-sphere for Alexandrov's differential operator is also new as far as we know, and Section \ref{Gromov} is new, apart from the obvious inspiration from \cite{Gromov} (explained in Remark\ref{comparison}). Finally, we consider the representation formulas for valuations, Theorems \ref{somevaluations} and \ref{monchar} new; the latter is, to our knowledge, the first description of general translation invariant and monotone valuations of all orders. 

I would like to thank Rolf Andreasson, Jakob Hultgren, Mattias Jonsson, Simon Larsson and Xu Wang for many discussions on the topic of this paper. Special thanks go to Thomas Wannerer, Ramon van Handel and Semyon Alesker for generous comments on a preliminary version.

\newpage
\section{Brief recap of superforms and some conventions.}\label{recap}
A superform on $\R^n$ is a differential form on $\C^n=\{x+i\xi; x,\xi\in \R^n\}$,
$$
\alpha=\sum_{I K} \alpha_{I K}(x) dx_I\w d\xi_K,
$$
whose coefficients do not depend on $\xi$. We use standard multiindex notation and if $|I|=p$ and $|K|=q$ for all non-zero  terms in the sum we say that $\alpha$ is of bidegree $(p,q)$. Notice that this bigrading differs from the standard one for differential forms on $\C^n$, where one would use instead the basis $dz_I\w d\bar z_K$, $z=x+i\xi$. The exterior derivative, $d$ operates on superforms. 

We define the complex structure, $J$,  on differential forms on $\C^n$, by
$$
J(d\xi_j)=dx_j, \quad J(dx_j)=-d\xi_j.
$$
This convention (unfortunately) differs from the one in \cite{Berndtsson2} by a  sign, but is here chosen in the standard way so that $J(dx_j+id\xi_j)=i(dx_j+id\xi_j)$. To avoid unpleasant minuses  we shall sometimes use the notation $\alpha^\#=-J(\alpha)$ when $\alpha=\sum \alpha_j dx_j$ is a $(1,0)$-form, so that
$$
\alpha^\#=\sum \alpha_j d\xi_j.
$$

We say that a form of bidegree $(p,p)$,
$$
\Omega=\sum \Omega_{I K} dx_I\w d\xi_K
$$
is {\it symmetric} if $J(\Omega)=\Omega$. Since
$$
J(\Omega)= \sum \Omega_{I K} (-1)^p d\xi_I\w dx_K=\sum \Omega_{I K } dx_K\w d\xi_I,
$$
this means that $\Omega_{I K}=\Omega_{K I}$ if we sum only over increasing multiindices, so that the representation is unique.

If $\psi$ is a function we define 
$$
\dhash\psi= \sum\psi_j d\xi_j= J^{-1} d\psi,
$$
and for general superforms
$$
\dhash\alpha =\sum (\dhash\alpha_{I K})\w dx_I\w d\xi_K= J^{-1}d J\alpha.
$$
This coincides with the operator $d^c$ from K\"ahler geometry, but we write $\dhash$ to emphasize that it acts on superforms. 

If $\alpha= fdx_1\w...dx_n\w d\xi_1\w ... d\xi_n=: f dx\w d\xi$ is a superform of bidegree $(n,n)$ (with, say, compact support) we define its superintegral as
$$
\int \alpha=\int_{\R^n} fdx\int d\xi =\int_{\R^n} fdx (-1)^{n(n-1)/2},
$$
i.e. we {\it define} the integral of $d\xi$ to be $(-1)^{n(n-1)/2}$. This is sometimes called the Berezin integral, except for the possible minus-sign that we have inserted in order to make the integral of 
$$
fdx_1\w d\xi_1\w ...dx_n\w d\xi_n
$$
positive if $f\geq 0$ (in accordance with the conventions for integrals over $\C^n$).

Instead of using the Berezin integral, we could have considered our superforms as living on $\R^n\times T^n$, where $T^n=\R^n_\xi/ \mathbb Z^n$, and used integration over the torus. The advantage in using Berezin integration instead, is that we may take restrictions of superforms to linear subspaces and integrate over them. This would be problematic with the other approach, since the quotient of a subspace with a lattice is in general not a subtorus. Another possibility is to consider superforms as forms on $\R^n$, whose coefficients are of the type
$$
A_I:=\sum_K a_{I K}(x) d\xi_K,
$$
and view $d\xi_i$ as extra, anticommuting, variables, in addition to the commuting variables $x_i$. (This is the motivation for the terminology `super-', in analogy with  analysis on supermanifolds.)  Berezin integration would then define a pairing, and norms, on the anticommuting part of the variables. The advantage with our formulation, as compared to this second possibility, is that it allows us to use the formalism of Kähler manifolds, which turns out to be very useful.

 If $M$ is an (oriented) submanifold of $\R^n$ of dimension $m$ we can integrate superforms
$$
\alpha=\sum_{|I|=m} \alpha_I dx_I\w d\xi 
$$
of bidegree $(m,n)$ in the same way
\be\label{boundaryintegral}
\int_M \alpha= (-1)^{n(n-1)/2}\int_M \sum_{|I|=m} \alpha_I dx_I.
\ee
It is clear that  Stokes' formula holds, e.g. 
\be\label{Stokes}
\int_{\partial D}\alpha=\int_D d\alpha,
\ee
if $D$ is a bounded domain in $\R^n$ and $\alpha$ is of bidegree $(n-1,n)$. Similarily, if $\alpha$ is an $(n-1,n)$-form of compact support, the integral of $d\alpha=0$. Hence
$$
\int du\w v=(-1)^{\deg(u)+1}\int u\w dv
$$
if either $u$ or $v$ has compact support (and $\deg(u)$ is the total degree of $u$). Since
$$
\int \alpha=\int J(\alpha)
$$
and $\dhash= J^{-1}dJ$, the same  holds for  $\dhash$. This leads to the useful formula
$$
\int (\ddhash u)\w v=\int u\w\ddhash v.
$$
Notice, however, that there is no counterpart to Stokes' formula (\ref{Stokes}) for $\dhash$; in fact 
$$
\int_{\partial D} \alpha
$$
is not even defined if $\alpha$ is of bidegree $(n,n-1)$. 

So far we have tacitly assumed that we have fixed a coordinate system on $\C^n$, $z_j=x_j+i\xi_j$. If we change coordinates to 
$$
w= Az, \quad w=y+i\eta
$$
where $A\in GL(n,\R)$ ( i.e. $A$ is a  matrix with {\it real} entries), most of what we have said does not change, with the exception that 
$$
\int d\xi \neq \int d\eta
$$
in general. We therefore endow  $\R^n$ with a scalar product, which induces a scalar product on forms, and consider only orthonormal changes of coordinates.  We define the volume form
$$
dV= dx\w d\xi (-1)^{n(n-1)/2}= dx_1\w d\xi_1\w ...dx_n\w d\xi_n
$$
for any such system of coordinates. Notice that the introduction of a scalar product also allows us to define Berezin integration on the complexification of linear subspaces of $\R^n$.

Finally, a supercurrent of bidegree $(p,q)$ is an element in the dual of the space of smooth superforms of bidegree $(n-p,n-q)$ with compact support, endowed with the usual topology from the theory of distributions. There are two main examples (and many more will appear later on in this paper): A superform $\gamma$ with, say, continuous coefficients defines a supercurrent by
\be\label{formascurrent}
\gamma.\alpha=\int\gamma\w\alpha.
\ee
Here the integral is the superintegral and components of $\gamma\w\alpha$ that are not of bidegree $(n,n)$ give zero contribution to the integral. The second example is integration over an $m$-dimensional submanifold of $\R^n$, by (\ref{boundaryintegral}). Since this acts on superforms of bidegree $(m,n)$, this supercurrent has bidegree $(n-m,0)$ or, as is sometimes said, of bidimension $(m,n)$. 

In analogy with (\ref{formascurrent}) any supercurrent, $T$ of bidegree $(p,q)$,  can be written
\be\label{current}
T=\sum T_{I K} dx_I\w d\xi_K.
\ee
Here $T_{I K}$ acts on smooth forms of full bidegree, so they are distributions, or generalized functions. We will almost only deal with currents with measure coefficients. Therefore, throughout the paper, when we speak of  currents or supercurrents, we mean currents with measure coefficients (i.e. currents of order zero).  Strictly speaking, a measure is an element in the dual of the space of functions, so it should be considered as a supercurrent of  bidegree $(n,n)$. Here, however, we identify measures with (super)currents of bidegree $(0,0)$ by dividing with the volume form $dV$ and think of measures as generalized functions. Hence, if $T$ in (\ref{current}) has measure coefficients, the $T_{I K}$ are considered to be measures, even though it would be more correct to say that $T_{I K} dV$ is a measure. 

Notice also that a supercurrent, $S$, (of order zero!) of bidegree $(n,n)$ acts on superforms of bidegree $(0,0)$, i.e. functions on $\R^n$. $S$ is therefore a measure on $\R^n$; not on $\C^n$. 

One final convention: This paper is about superforms and supercurrents but to alleviate the language a little we will mostly omit the `super' in the sequel. Thus we  will often talk about `forms' and `currents' and `integrals', hoping that it will be clear from the context when a `super' should be inserted (which is almost all of the time).


\newpage

\section{Positivity of superforms}\label{positivity}

We first look at symmetric superforms of bidegree $(p,p)$ at a point, i.e. forms
$$
\Omega=\sum \Omega_{I J} dx_I\w d\xi_J (-1)^{p(p-1)/2}
$$
where the coefficients $\Omega_{I J}$ are real constants, symmetric in the indices $I$ and $J$. We also sum only over increasing multiindices in order to have a unique representation. As in the complex case (see e.g. \cite{Harvey-Knapp}), we will (as in \cite{Lagerberg}) consider three types of positivity: (Cf. also \cite{Lelong}, where positivity of forms and currents was first introduced, and \cite{Demailly}, but notice that the terminology there is slightly different.)

First, we define an {\it elementary form} of bidegree $(p,p)$ to be a form that can be written 
$$
E=\alpha_1\w \alpha_1^\#\w ...\alpha_p\w \alpha_p^\#
$$
for $\alpha_j$ of bidegree $(1,0)$.
The  {\it strongly positive} forms  are the forms that can be written as a finite sum 
$$
\Omega=\sum c_s E_s,
$$
where $c_s>0$ and $E_s$ are  elementary forms. This is a closed cone. The fact that it is closed is not evident from first principles; a proof was recently given by Xia, \cite{Mingchen}. (The proof in \cite{Mingchen} is for strongly positive form in the complex setting, but carries over to our setting.)

A weaker notion is just plain  positivity. $\Omega$ is called {\it positive} if the matrix $(\Omega_{I J})$ is positive definite. This is equivalent to saying that the quadratic form 
$$
Q(\alpha,\alpha):= \Omega\w \alpha\w \alpha^\# (-1)^{(n-p)(n-p-1)/2},
$$
defined on the space of forms $\alpha$ of bidegree $(p,0)$, is positive definite.

Finally, $\Omega$ is said to be {\it weakly positive} if 
$$
\Omega\w E\geq 0,
$$
for all elementary forms $E$,  or equivalently $\Omega\w \sigma \geq 0$ for all strongly positive forms $\sigma$. One easily sees that elementary forms, and hence all strongly positive forms are positive. Likewise, a positive form wedged with a positive form of complementary bidegree is a positive form of full degree. Hence a positive form wedged with a strongly positive form is also positive, so positive forms are weakly positive. 

If $\omega$ is of bidegree $(1,1)$,
$$
\omega=\sum a_{j k}dx_j\w d\xi_k
$$
with $(a_{j k})$ symmetric, we can, by   the spectral theorem write  $\omega=\sum\lambda_j dy_j\w d\xi_j$ after an orthonormal change of coordinates.  Hence $\omega$ is strongly positive if and only if all the eigenvalues $\lambda_j\geq 0$, i. e. if $(a_{j k})$ is positively semidefinite. It is easily seen that this is equivalent to $\omega$ being weakly positive, so for forms of bidegree $(1,1)$ all notions of positivity coincide.
 
Thus we have three closed cones in $\Lambda^{p,p}(R^n_x\oplus \R^n_\xi)$, $SP_p\subseteq P_p\subseteq WP_p$, strongly positive, positive and weakly positive respectively. In bidegree $(1,1)$ and $(n-1,n-1)$ all the cones are equal; otherwise all inclusions are strict.  By definition the cone of strongly positive $(p,p)$-forms is dual to the cone of weakly positive $(n-p,n-p)$-forms  under the non degenerate pairing coming from the wedge product. The cone of positive $(p,p)$-forms  is dual to the cone of positive forms of bidegree $(n-p,n-p)$. Note also that if $\Omega$ is positive in any of these senses and $\omega$ is a positive $(1,1)$-form, then $\omega\w \Omega$ is again positive in the same sense (use the spectral theorem again).

So far the discussion is completely parallel to the complex case. There are, however, important differences, due to the fact that we have a different notion of bigrading; (number of $dx_i$, number of $d\xi_j$) versus (number of $dx_i+i d\xi_i$, number of $dx_j-i d\xi_j)$ in the complex setting. To distinguish the two bigradings we will sometimes write bidegree $(p,q)_{\R}$ for the real case and bidegree $(p,q)_{\C}$ for the complex case.

In the complex case the counterparts of our forms $\alpha\w \alpha^\#$, where $\alpha$ is $(1,0)_\R$, are forms $i a\w\bar a$, where $a$ is $(1,0)_\C$. An elementary form in the complex setting is a $p$-fold wedge product of such forms
$$
ia_1\w \bar a_1\w ...ia_p\w\bar a_p,
$$
and a form is strongly positive in the complex setting if it is a positive combination of elementary forms. 

Note first that if $\alpha$ is of bidegree $(1,0)_\R$, then $a:=\alpha -iJ(\alpha)$ satisfies
$ J(a)=ia$, so $a$ is $(1,0)_\C$. Since 
$$
(i/2) a\wedge\bar a=-\alpha\w J(\alpha)=\alpha\w\alpha^\#,
$$
we see that an elementary form in the real sense is elementary in the complex sense as well.

On the other hand a complex $(1,0)_\C$ form $a$ is uniquely determined by its real part, $\gamma$,
$$
a=\gamma-iJ(\gamma), \,\,\text{and}\,\, ia\w\bar a=-2\gamma\w J(\gamma)
$$
($J$ is the complex structure).  Looking at the real bigrading we can write $\gamma=\alpha+\beta$ where $\alpha$ is $(1,0)_\R$ and $\beta$ is $(0,1)_\R$. Therefore 
$$
\gamma\w J(\gamma)=\alpha\w J(\beta)+\beta\w J(\alpha)+ \alpha\w J(\alpha)+ \beta\w J(\beta),
$$
so $ia\w\bar a$ is of real bidegree $(1,0)$ if and only if 
$$
\alpha\w J(\beta)+\beta\w J(\alpha)=0.
$$
Since the two terms here are of different real bidegree, each of them must vanish, which means that $\alpha$ and $J(\beta)$ are parallel. Equivalently, if $i(a\w\bar a)$ is elementary in the real sense, after multiplication with a complex number, $a=\gamma-iJ(\gamma)$, where $\gamma $ is $(1,0)_\R$. Thus the set of elementary forms in the real sense is strictly included in the set of complex elementary forms. 

A fundamental fact in the complex case is that any real symmetric $(p,p)$-form
$$
\Omega=\sum \Omega_{I J} dz_I\wedge d\bar z_J (-1)^{p(p-1)/2}, \Omega_{I J}=\Omega_{J I}
$$
can be written as a finite real combination of complex elementary forms, i. e. as a difference of strongly positive forms (see \cite{Harvey-Knapp} and \cite{Demailly}). The analog of this fails in the real setting, as first discovered in  \cite{Gubleretal}; a general symmetric $(p,p)_\R$ form can not be written as (a limit of) a finite combination of real elementary forms. To see this we introduce the linear map on superforms
\be\label{T}
Tu:= \sum dx_i\w \delta_{\xi_i}(u),
\ee
where $\delta_{\xi_i}$ is contraction with the vector field $\partial / \partial \xi_i$. This operator will reappear and be more motivated later. For the moment we just notice that it has the following properties:

1. It is independent of choice of linear coordinates.

2. It is a derivation; $T(u\wedge v)=T(u)\w v+u\w T(v)$.

3. If $\alpha$ is $(1,0)_\R$, then $T(\alpha)=0$ and $T(\alpha^\#)=\alpha$.

By properties 3 and 2, $T(\alpha\wedge\alpha^\#)=0$, which, using  property 2 again,  means that any elementary form, and hence any linear combination of elementary forms, are annihilated by $T$. On the other hand, there are symmetric $(p,p)_\R$-forms that are not annihilated by $T$, like 
$$
dx_1\w dx_2\w d\xi_3\w d\xi_4+ dx_3\w dx_4\w d\xi_1\w d\xi_2
$$
in the complexification of $\R^4$, and such forms can not be written as combinations of elementary forms. For lack of better terminology we say that a (necessarily symmetric) $(p,p)_\R$ form is {\it strong} if it is a difference between two strongly positive forms, and we denote the space of all strong forms of bidegree $(p,p)$, $S_p$ - or sometimes just $S$ when the degree is understood or variable. 

Thus $S\subseteq Ker(T)$ and  $Ker(T)$ is a proper subspace of the space of all symmetric forms. 
The next theorem is   Theorem \ref{theoremstrongform} in Section \ref{charstrongform}. We postpone the proof till  there since it uses tools that will be developed later.
\begin{thm}\label{strongforms1}
A  $(p,p)$-form , $\Omega$, that satisfies $T(\Omega)=0$ is strong. In other words, $S=Ker(T)$.
\end{thm}

The next lemma will play an important role later when we discuss valuations on convex bodies.
\begin{lma}\label{nondeg} Let $\Omega$ be a strong $(p,p)_\R$-form and assume that
$$
\Omega\w\Omega'=0
$$
for all strong (or, equivalently,  elementary) forms $\Omega'$ of bidegree $(n-p,n-p)_\R$. Then $\Omega=0$.
\end{lma}
\begin{proof}

We will use the Hodge *-operator on $\R^n_x\oplus \R^n_\xi$, defined so that
$$
A\w *B=(A,B) dV,
$$
for $(p,q)$-forms $A$ and $B$. Here $(A,B)$ is the Euclidean scalar product and $dV$ the volume element. We first claim that if $E$ is an elementary $(p,p)$-form,
$$
E=\alpha_1\w \alpha_1^\#\w ...\alpha_p\w \alpha_p^\#
$$
then $*E$ is also elementary, of bidegree $(n-p,n-p)$. To see this, use first that up to a sign $E$ equals
 $$
 \alpha_1\w ...\alpha_p\w\alpha_1^\#...\alpha_p^\#.
 $$
 We can replace $\alpha_j$ by an orthonormal basis $e_j$ for the space spanned by $\alpha_1, ...\alpha_p$; this changes $E$ only by multiplication with some non-zero number. Now complete the $e_j$:s to an orthonormal basis for all $(1,0)$-forms, $e_1, ...e_n$. Then 
 $$
 *E= \pm e_{p+1}\wedge ... e_n\w e^\#_{p+1}\wedge ... e^\#_n,
 $$
 hence elementary. Therefore $*\Omega$ is a strong form if $\Omega$ is strong. From the hypothesis we get
 $$
 0=\Omega\wedge *\Omega=\|\Omega\|^2 dV.
 $$
Hence $\Omega=0$.
\end{proof}
The next lemma gives a criterion for the vanishing of a strong form which is more useful in practice.
\begin{lma}\label{Criterion}
Any elementary $(p,p)$-form $E$ can be written
$$
E=Q^p/p!,
$$
where 
$$
Q=\sum q_{j k}dx_j\w d\xi_k,
$$
and the matrix $(q_{j k})$ is positively semidefinite.  Hence, a strong $(p,p)$-form $\Omega$ must vanish if
$$
\Omega\w Q^{n-p}=0
$$
for any $Q$ of this type. Conversely, any power $Q^p$ of  a semipositive form $Q$ is strongly positive.
\end{lma}
\begin{proof}
As in the proof of lemma \ref{nondeg} we may assume that 
$$
E=e_1\w e_1^\#\w ...e_p\w e_p^\#,
$$
where $e_1, ...e_p$ are the first $p$ elements in some orthonormal basis for the space of $(1,0)$-forms. Putting 
$$
Q=\sum_1^p e_j\w e_j^\#
$$ 
we get the first part of the lemma. The last part follows from the spectral theorem: Any positively semidefinite $(1,1)$-form  can be written 
$$
\sum \lambda_j dy_j\w d\eta_j,
$$
where $\lambda_j\geq 0$,
after an orthonormal change of coordinates. 
\end{proof}

The fact that not all forms can be written as linear combinations of elementary forms, i.e.  as differences of strongly positive forms, has the inconvenient consequence that some non-zero forms are both weakly positive and weakly negative.  Recall that $WP$ denotes the cone of weakly positive forms, and $SP$ the cone of strongly positive forms. (We look at forms of a fixed bidegree but suppress the bidegree in the notation when the degree is understood.)  We put
$$
W_0:= WP\cap(-WP).
$$ 
Thus $W_0$ is the linear space of forms that are both weakly positive and weakly negative; we will call them {\it weakly null}. We claim that $W_0$ is the orthogonal complement of $S$ both with respect to the scalar product on $(p,p)$- forms and with respect to the pairing given by the wedge product. Indeed,
$$
W_0=\{\Omega; \Omega\w \Omega'=0, \Omega'\in SP_{n-p}\}=\{\Omega; \Omega\w \Omega'=0, \Omega'\in S_{n-p}\}=\{\Omega; (\Omega, \Omega')=0, \Omega'\in S_p\},
$$
since the Hodge * of a strong form is strong. 

We now move on to discuss positivity of smooth (or continuous) forms with variable coefficients and of currents. A form with continuous coefficients is defined to be positive in any of the senses discussed if it is positive at any fixed point. Similarily, it belongs to the spaces $S$ and $W_0$ if it does at any point. The corresponding notions for currents is defined by duality. We will denote the action of a current on a smooth form by integrals, which is reasonable since in the end we will only work with currents with measure coefficients.

Thus, a symmetric current, $\Omega$,  of bidegree $(p,p)$ is strongly positive if 
$$
\int \Omega\w\Omega'\geq 0
$$ 
for all weakly positive $\Omega'$ (of bidegree $(n-p,n-p)$) with compact support, and $\Omega$ is (weakly) positive if the same thing holds for $\Omega'$ (strongly) positive. In case $\Omega$ is actually a continuous form, this coincides we the notions for forms, by duality of the respective cones. In the same way we define the spaces $S$ and $W_0$ for currents. Just like for smooth forms, all these positivity notions are preserved if we wedge $\Omega$ with a smooth positive $(1,1)$-form, as can be seen by duality, and the spaces $S$ and $W_0$ are also preserved. 

Note that if a current is positive in any of these senses, its translates are positive in the same sense. This means that convolutions with a smooth positive function , $\Omega\star \chi$ (defined coefficient-wise) is also positive. As a consequence, we can regularize a positive current, and approximate it by smooth positive forms, and also approximate a current in $S$ by smooth forms in $S$. This means that for many statements about positive currents, it is enough to prove them for smooth positive forms. As a consequence we get from Lemma \ref{Criterion}
\begin{prop} \label{Current-Criterion} Let $\Omega$ be a current in $S$ of bidegree $(p,p)$ and assume that
\be\label{hypoth}
\int \Omega\w \chi Q^{n-p}=0
\ee
for all smooth  forms $Q=\sum q_{i j}dx_i\w d\xi_j$ with  symmetric coefficients such that $(q_{i j})$ is positively semidefinite, and $\chi$ a smooth function  with compact support. Then $\Omega=0$. The same conclusion follows if (\ref{hypoth}) is only supposed to hold when   $Q=\ddhash\psi$, where   $\psi$ is smooth.
\end{prop}
\begin{proof} The first part of the proposition follows from Lemma \ref{Criterion} by regularization since the hypothesis is invariant under translation. In the same way we see that it is enough to prove the last part under the assumption that $\Omega$ is smooth. But then \ref{hypoth} implies that $\Omega\w (\ddhash\psi)^{n-p}=0$ pointwise and we can just apply Lemma 
\ref{Criterion} again (to $\psi=\sum q_{i j}x_ix_j$).
\end{proof}
For emphasis we state the following stronger version of Proposition \ref{hypoth} separately. 

\begin{prop}\label{positive}
Let $\Omega$ be a current  of bidegree $(p,p)$ in an open set $U$ and assume that
\be\label{2hypoth}
\int \Omega\w \chi Q^{n-p}\geq 0
\ee
for all  $(1,1)$-forms $Q= \ddhash\psi$ where $\chi\geq 0$ is smooth of compact support in $U$ and $\psi$ is smooth and convex. Then $\Omega$ is weakly positive. Conversely, if $\Omega$ is weakly positive (\ref{2hypoth}) holds.
\end{prop}
\begin{proof}
Again, it is enough to prove that $\Omega$ is weakly positive assuming that  $\Omega$ is smooth. Then the hypothesis says that $\Omega\w (\ddhash\psi)^{n-p}\geq 0$ for all convex $\psi$.
The conclusion then follows from the first part of Lemma \ref{Criterion}. The second part follows from the second part of the same lemma.
\end{proof}

It will later be important to know that currents that are positive in various senses have measure coefficients and we will end this section with a result to that effect. If 
$$
\Omega=\sum \Omega_{I J} dx_I\w d\xi_J (-1)^{p(p-1)/2}
$$
we call
$$
\tau_\Omega:= \sum \Omega_{I I}
$$
the trace of $\Omega$. Notice that 
$$
\ddhash |x|^2/2=\sum dx_i\w d\xi_j=:\beta\,\, \text{and}\,\, \beta^q/q!=\sum_{|I|=q} dx_I\w d\xi_I (-1)^{q(q-1)/2}, 
$$
from which it follows that
$$
\tau_\Omega=\Omega\w\beta^{n-p}/(n-p)!.
$$
Hence the trace of $\Omega$ does not depend on the choice of orthonormal coordinates. 
It also follows, since $\beta^{n-p}$ is strongly positive,  that the trace is nonnegative if $\Omega$ is weakly positive. Notice also that, for any choice of orthonormal coordinates and  for any multiindex, that we can take to be $I=(1,2...p)$,
$$
\Omega_{I I} dV=\Omega\w e_{p+1}\w e_{p+1}^\#\w ...e_n\w e_n^\#\geq 0
$$
if $\Omega$ is weakly positive and  therefore $\Omega_{J J}\leq \tau_\Omega$ for any $J$.

The trace, however,  may be zero; indeed this happens if $\Omega\in W_0$. Thus there is no way to estimate the whole current in terms of the trace (like there is in the complex setting) if $\Omega$ is only weakly positive.
\begin{lma}\label{traceestimate} Let
$$
\Omega=\sum \Omega_{I J} dx_I\w d\xi_J (-1)^{p(p-1)/2}
$$
be a symmetric $(p,p)$-form which is either positive or weakly positive and strong. Then,  there is a uniform constant depending only on $p$ and $n$ such that for all indices $I, J$
$$
|\Omega_{I J}|\leq C\tau_\Omega.
$$
\end{lma}
\begin{proof}
When $\Omega$ is positive this is a direct consequence of Cauchy's inequality, 
$$
|\Omega_{I J}|\leq (\Omega_{I I}+\Omega_{J J})/2\leq \tau_\Omega/2.
$$
Now assume that $\Omega$ is weakly positive and strong. Since $S$ is a subspace of the space of all symmetric $(p,p)$-forms it is spanned by a finite number, $N$,  of elementary forms, $E_1, ...E_N$, with $N$ depending only on $n$ and $p$. Discarding some of them we may assume the $E_j$ form a basis, which we fix. As in the proof of Lemma \ref{nondeg}, we may assume that each $E_j$ has the form
$$
E_j= e_1\w...e_p\w e_1^\#\w ... e_p^\#
$$
where the $e_k$ are elements in an orthonormal basis for the space of $(1,0)$-forms (which basis depends on $j$). Writing $\Omega$ in this basis, we see that, with $I=(1,...p)$,
$$
(\Omega, E_j)=\Omega_{I I}\leq \tau_\Omega,
$$
 for all $j$ since the trace is independent of the choice of orthonormal basis. Write $\Omega=\sum c_i E_i$. Then
$$
0\leq \sum c_i (E_i, E_j)\leq \tau_\Omega
$$ 
for all $j$. Since the matrix $(E_i,E_j)$ is invertible it follows that $(\sum c_j^2)^{1/2}\leq C\tau_\Omega$ (where $C$ does not depend on $\Omega$). Since
$$
|\Omega_{I J}|\leq \sum |c_i||(E_i)_{I J}|,
$$
the lemma follows. 
\end{proof}
The next proposition is the main consequence of the lemma.
\begin{prop}\label{estimate}
Let
$$
\Omega=\sum \Omega_{I J} dx_I\w d\xi_J (-1)^{p(p-1)/2}
$$
be a symmetric $(p,p)$-current which is either positive or weakly positive and strong. Then all the coefficients of $\Omega$ are measures, absolutely continuous with respect to the trace measure $\tau_\Omega$, with a Radon-Nikodym derivative bounded by a constant independent of $\Omega$. 
\end{prop}
\begin{proof}
Assume first that $\Omega$ is smooth. Then, if $\chi$ is a smooth function with compact support we have by the previous lemma that
$$
|\int \chi \Omega_{I J}|\leq C\int|\chi| \tau_\Omega.
$$
By regularization, the same thing holds if $\Omega$ is just a current, positive or weakly positive and strong. Since $\tau_\Omega \geq 0$, it is a positive measure, so the claim follows.
\end{proof}

\newpage
\section{The theory of Bedford and Taylor.}\label{Bedford-Taylor}

In \cite{Bedford-Taylor1} and \cite{Bedford-Taylor2} Bedford and Taylor defined and studied currents of the form
$$
dd^c\phi_1\w ...dd^c\phi_p,
$$
where $\phi_i$ are locally bounded plurisubharmonic functions. (It is a priori not evident that  such currents are well defined since currents with measure coefficients, like $dd^c\phi$ where $\phi$ is plurisubharmonic, cannot in general be multiplied.)  Following \cite{Lagerberg} we will mimic their arguments in the setting of supercurrents defined by convex functions. Our aim is to prove that currents like
$$
\ddhash\psi_1\w ...\ddhash\psi_p\w \Omega
$$
where $\psi_i$ are convex and $\Omega$ is positive (in various senses) are well defined and  continuous under uniform convergence of the $\psi_i$ and weak convergence of $\Omega$. This is essentially included in Lagerberg's work but we repeat the arguments here, partly for completeness and partly since we need the results in a more general setting: Lagerberg worked with positive currents but for us it will be important to allow currents that are weakly positive and strong as well. Apart from this complication, the real setting is quite a bit easier than the complex setting since (finite valued) convex functions are automatically continuous. We start with a basic lemma.
\begin{lma}\label{basic} Let $\Omega$ be a closed current of bidegree $(p,p)$ which is either positive or weakly positive and strong in an open set $U$ in $\R^n$. Let $K$ be a compact subset of $U$ and let $\psi$ be smooth and convex in $U$. Put
$$
\Omega_\psi:=\ddhash\psi\w \Omega.
$$
Then $\Omega_\psi$ is also positive or weakly positive and strong, respectively, and 
$$
\int_K \tau_{\Omega_\psi}\leq C\sup_U|\psi| \int_U\tau_\Omega,
$$
where $C$ does not depend on $\Omega$ or $\psi$.
\end{lma}
\begin{proof} That $\Omega_\psi$ has the same positivity properties as $\Omega$  follows immediately from the definitions when $\Omega$ is smooth and in general by regularization (or duality). Recall that we defined 
$$
\beta=\sum dx_i\w d\xi_i,
$$
and we have 
$$
\tau_\Omega=\Omega\w \beta^{n-p}/(n-p)!\,\, \text{and}\,\, \tau_{\Omega_\psi}=\ddhash\psi\w\Omega\w\beta^{n-p-1}/(n-p-1)!.
$$
Let $\chi$ be a smooth cutoff function which is equal to 1 on $K$ and compactly supported in $U$. Then
$$
\int_K \tau_{\Omega_\psi}\leq \int_U\chi \ddhash\psi\w\Omega\w\beta^{n-p-1}/(n-p-1)!.
$$
By Stokes' theorem this equals
$$
\int_U\psi \ddhash\chi\w\Omega\w\beta^{n-p-1}/(n-p-1)!\leq C\sup_U|\psi| \int_U\tau_\Omega,
$$
since all coefficients of $\Omega$ can be estimated by its trace by Proposition \ref{estimate}.

\end{proof}

We are now ready to give the main result of this section.  We aim to define the map
$$
(\psi_1, ...\psi_k, \Omega) \to \ddhash\psi_1\w...\ddhash\psi_k\w \Omega
$$
for all convex functions and as large a class of $\Omega$ as possible. The map is clearly defined when $\psi_j$ are smooth and we will define its extension to general $\psi_j$ by continuity.

\begin{thm}\label{crucial} The map
$$
(\psi_1, ...\psi_k, \Omega) \to \ddhash\psi_1\w...\ddhash\psi_k\w \Omega =:M(\psi_1, ...\psi_k,\Omega)
$$
defined for $\psi_j$ smooth,
has an extension to all convex functions and $(p,p)$ currents $\Omega$ that are closed, weakly positive and strong, that is continuous in the following sense: If $\psi_j^i$ converge to $\psi_j$ uniformly and $\Omega^i$ converges to $\Omega$ weak* as $i\to\infty$, then $M(\psi_1^i, ...\psi_k^i,\Omega^i)$ converges weak* to $M(\psi_1, ...\psi_k,\Omega)$. 
$M(\psi_1, ...\psi_k,\Omega)$ is again closed, weakly positive and strong. 
\end{thm}

\begin{proof}
By induction it is enough to prove the theorem when $k=1$. We need to prove that if $\psi^i$  and $\Omega^i$ converge in the sense described, and $\psi^i$ are smooth, then 
$$
\ddhash\psi^i\w\Omega^i
$$
has a weak* limit. Note that the fact that $\Omega^i$ converges implies that the trace measures of $\Omega^i$ are uniformly bounded 
on each compact. Lemma \ref{basic} implies that this holds also for the trace measures of $\ddhash\psi^i\w\Omega^i$. Since these currents are  weakly positive and strong it follows by Proposition \ref{estimate} that all the coefficients are measures with uniformly bounded mass on compacts, so there is at least a subsequence that converges. We need to prove that all convergent subsequences have the same  limit. 

For this it is enough to prove that the limits of 
$$
I_i(\chi):= \int\chi\w\ddhash\psi^i\w\Omega^i
$$
exist when $\chi$ is a smooth and compactly supported form of bidegree $(n-p-1,n-p-1)$. We have
$$
I_i=\int\psi^i\ddhash\chi\w\Omega^i.
$$
Take a large $k$ and put
$$
I_{i , k}=\int\psi^k\ddhash\chi\w\Omega^i.
$$
Then 
$$
|I_{i,k}-I_i|\leq C\sup|\psi^i-\psi^k|,
$$
the supremum taken over the support of $\chi$. Thus $|I_{i,k}-I_i|<\epsilon$ if $i$ and $k$ are large enough. On the other hand the limit of $I_{i,k}$ as $i$ tends to infinity exists by the weak convergence of $\Omega^i$. Therefore the limits of any two subsequences of $I_i$ can differ at most by $2\epsilon$ for any $\epsilon>0$. Hence the limit of $I^i$ exists, which concludes the proof.

\end{proof}
\begin{remark}
Note that the same result, with `weakly positive and strong' replaced by `positive' also holds, with the same proof. 
\end{remark}

We list a few consequences of Theorem  \ref{crucial}   that will be useful. The first is that if $\psi$ is convex and $\Omega$ satisfies the hypothesis of the proposition, then
$$
d\psi\w \dhash\psi\w \Omega
$$
is well defined and has the same continuity properties as $\ddhash\psi\w \Omega$. In the proof of  this we may of course assume that $\psi>0$. Then $\psi^2$ is again convex with
$$
\ddhash\psi^2/2=\psi\ddhash\psi +d\psi\w \dhash\psi.
$$
Since our claim holds for $\ddhash\psi^2$ and $\psi\ddhash\psi$ it must hold for $d\psi\w \dhash\psi$ as well. 

The next thing we want to record is that our closed positive currents can also be wedged with the supercurrent of integration on  hypersurfaces. The simplest case is that of hyperplanes.
\begin{prop}\label{hyper}
Let $V$ be an affine hyperplane
$$
V=\{x\in\R^n; x_1=a\},
$$
and $V$ its current of integration. 
Then the map 
$$
\Omega\to \Omega\w [V]\w\dhash x_1,
$$
defined when $\Omega$ is a smooth, strong form of bidegree 
$(p,p)$, $p<n$, can be extended  to all weakly positive and strong
$(p,p)$-currents, so that it is continuous in the weak*-topology and in $a$.
\end{prop}
\begin{proof}
This follows immediately from Theorem \ref{crucial}, since
$$
[V]\w\dhash x_1=\ddhash \max(x_1, a).
$$
\end{proof}

We next turn to hypersurfaces that are level surfaces of convex functions. 
 More precisely, let $\mu$ be a {\it gauge} function. By this we mean a convex  function on $\R^n$, which is homogeneous of order 1 and strictly positive outside the origin. Then
$$
B_\mu:=\{x; \mu(x)\leq 1\}
$$
is a convex body, containing the origin in its interior. Conversely, any such convex body can be represented this way by a unique gauge. We define the supercurrent of integration on the boundary of $B_\mu$, $S_\mu$, by
$$
\int [S_\mu]_s \w \chi= \int_{S_\mu} \dhash\mu\w \chi
$$
for any smooth form $\chi$ of degree $2n-2$. This is well defined if $\mu$ (and hence $S_\mu$) is smooth. We will show that we can extend the definition by continuity to all gauge functions
and that the map 
$$
(\mu,\Omega) \to [S_\mu]_s\w \Omega
$$
is jointly continuous in $\mu$ and  $\Omega$ as in Theorem \ref{crucial}, provided we only look at currents $\Omega$ that are {\it homogeneous of order zero} (it does not hold in general). By this we mean that 
$$
F_t^*(\Omega) =\Omega,
$$
where $F_t(x,\xi) = (tx,\xi)$. If $\Omega$ is of bidegree $(p,p)$ this means that its coefficients are homogeneous of order $-p$. Homogeneous forms and currents will play an important role later and will be discussed in more detail in section \ref{sectionhomform}. For the moment we just note that if a function $\psi$ is homogeneous of order 1, then its second order derivatives are homogeneous of order $-1$, so currents like  $\ddhash\psi$ and their wedge products are homogeneous of order zero.

To define $[S_\mu]_s$ for general $\mu$ we  put $\mu^+=\max(1,\mu)$, and 
$$
T_\mu:= \ddhash\mu^+=\ddhash\max(0,\mu-1).
$$
We claim that, if $\mu$ is smooth, 
$$
T_\mu=[S_\mu]_s + \chi_{\{\mu>1\}}\ddhash\mu.
$$
For this we use Stokes' formula again: If $\chi$ is smooth with compact support
$$
\int \chi\w T_\mu=\int \max(0,\mu-1)\ddhash\chi=\int_{\mu>1} (\mu-1)\ddhash\chi=\int_{\mu>1}\chi\w\ddhash\mu +\int_{S_\mu} \dhash\mu\w\chi
$$
(the sign of the last term comes from integrating over the exterior of $B_\mu$). We therefore define, in general, 
$$
[S_\mu]_s=\chi_{\mu\leq 1}T_\mu.
$$
\begin{prop}\label{sumup}
The map
$$
(\mu, \Omega)\to [S_\mu]_s\w\Omega,
$$
defined for $\mu$ a smooth gauge function and $\Omega$ a smooth $(p,p)$ form, has an extension to all convex gauge functions and all $(p,p)$-currents $\Omega$ that are homogeneous of order zero,  closed, weakly positive and strong.   The map  is continuous in the following sense:  If $\mu^i \to \mu$ uniformly and $\Omega^i\to \Omega$ weak*, then 
$$
[S_{\mu^i}]_s\w\Omega^i \to [S_\mu]_s\w\Omega
$$ 
in the weak* topology.
\end{prop}
\begin{proof}
We have 
$$
[S_\mu]_s=\chi_{\mu\leq1} T_\mu.
$$
Therefore continuity  almost follows from the continuity properties of $T_\mu\w\Omega$ in Theorem \ref{crucial}, except for the problem that $\chi_{\mu\leq1}$ is not continuous.
To overcome this  we approximate $\chi_{\mu\leq 1}$ by $\chi_\epsilon(\mu)$ where $\chi_\epsilon(t)$ is a nonnegative smooth function that equals 1 when $t\leq 1$ and is supported where $t<1+\epsilon$. 
We will use the general fact that if currents $R^i$ with measure coefficients converge weak* to $R$ and continuous functions $\phi^i$ converge uniformly to $\phi$, then 
$\phi^iR^i$ converge weak* to $\phi R$.
Put 
$$
S^\epsilon_\mu=\chi_\epsilon(\mu) T_\mu.
$$
Since $\chi_\epsilon(\mu)$ is a continuous function and continuous in $\mu$ it follows from Theorem \ref{crucial}  and the remark above that $S^\epsilon_\mu\w\Omega$ is jointly continuous in $\mu$ and  $\Omega$.

Therefore it suffices to show that 
$$
\int (S^\epsilon_{\mu^i} -[S_{\mu^i}]_s)\w\Omega^i\w\Omega'=\int_{1<\mu^i<1+\epsilon} \chi_\epsilon(\mu^i) \ddhash\mu^i\w\Omega^i\w\Omega'
$$
is uniformly small when $\Omega^i$ tend weak* to $\Omega$, $\mu^i$ tend uniformly to $\mu$ and $\Omega'$ is continuous with compact support. But, by Lemma \ref{basic} the mass of the currents
$$
\ddhash\mu^i\w\Omega^i
$$
over any bounded domainin $\R^n\setminus\{0\}$ is uniformly bounded. By homogeneity this gives that the mass over a domain $\{ 1<\mu^i< 1+\epsilon\}$ is bounded by a constant times $\epsilon$, which concludes the proof.
\end{proof}
Although it is not quite indispensable, this proposition will be useful in computations later, since we may assume that e.g. $\mu$ (and hence $S_\mu$) is smooth. 

\newpage

\section{The Monge- Amp\`ere operator}

In particular, it follows from Theorem \ref{crucial} that  if $\psi$ is convex,
$$
(\ddhash\psi)^n/n!
$$
is a well defined measure and is continuous in $\psi$.  When $\psi$ is smooth
$$
(\ddhash\psi)^n/n!=\det(\psi_{j k}) dV
$$
and one therefore defines it to be the Monge- Amp\`ere measure, $MA(\psi)$,  of $\psi$ in general. 

The classical definition of Alexandrov (\cite{1Alexandrov})  of the Monge- Amp\`ere measure starts by defining the Monge- Amp\`ere measure of an open set $U$:
$$
MA(\psi)(U)=\lambda(\nabla\psi(U)),
$$
where $\nabla\psi$ is the (multivalued) subgradient of $\psi$ and $\lambda$ is Lebesgue measure. The two definitions coincide when $\psi$ is smooth ( see e.g. the next section) and are continuous in $\psi$ so they agree in general.

One remark is in order here (and will be important when we discuss Monge- Amp\`ere measures on the boundary of a convex body). Alexandrov's definition obviously depends on a choice of Lebesgue measure on $\R^n$; it is only determined up to a multiplicative constant. In our definition this normalization is hidden in a choice of coordinates such that
$$
\int d\xi_1\w...d\xi_n(-1)^{n(n-1)/2} =1,
$$
or rather in the choice of a scalar product on $\R^n$ such that this holds for any orthonormal coordinates. This is a difference compared to the complex case; the complex Monge- Amp\`ere measure is uniquely determined by the complex structure. (I thank Mattias Jonsson for this remark.)

We shall now give a useful estimate for Monge- Amp\`ere masses. Say that a convex function in $\R^n$ is of linear growth if 
$$
\psi(x)\leq A|x|+B
$$
for some constants $A$ and $B$. For such functions we can define their indicator functions
$$
\psi^\circ(x)=\lim_{t\to\infty} \frac{\psi(tx)}{t}.
$$
Such limits exist since 
$$
t\to \ \frac{\psi(tx)-\psi(0)}{t}
$$
is increasing and bounded from above. The next theorem is also essentially contained in \cite{Lagerberg}.
\begin{thm}\label{MAestimate} Let $\psi_1, ...\psi_n$ be convex functions of linear growth in $\R^n$. Then
$$
V(\psi_1, ...\psi_n):=\int \ddhash\psi_1\w ...\ddhash\psi_n/n!<\infty.
$$
If $\phi_1, ...\phi_n$ is another $n$-tuple of convex functions of linear growth with $\psi_i^\circ\leq \phi_i^\circ$ for $i=1, ...n$, then
$$
V(\psi_1, ...\psi_n)\leq V(\phi_1, ...\phi_n).
$$
In particular, $V(\psi_1, ...\psi_n)$ depends only on $\psi_i^\circ$.
\end{thm}
\begin{proof} To prove the first part we may of course assume that $\psi_i(x)\leq |x|/2$ for all $i$. It suffices to show that there is a constant $C$, such that for any   $R>0$,
$$
\int_{|x|<R} \ddhash\psi_1\w ...\ddhash\psi_n \leq C.
$$
Choose $A$ so large that $\psi_i(x)\geq |x|-A$, 
when $|x|\leq R$, and put 
$$
\phi_i(x)=\max(\psi_i, |x|-A).
$$
Then $\phi_i$ are convex functions and $\phi_i(x)=\psi_i(x)$ when $|x|\leq R$ and $\phi_i(x)=|x|-A$ when $|x|\geq R'/2$ for $R'$ sufficiently large. We get
$$
\int_{|x|<R} \ddhash\psi_1\w ...\ddhash\psi_n\leq \int_{|x|<R'} \ddhash\phi_1\w ...\ddhash\phi_n=\int_{|x|=R'} \dhash|x|\w (\ddhash |x|)^{n-1}=
$$
$$
\int_{|x|<R'} (\ddhash |x|)^{n}.
$$
Since $|x|$ is homogeneous of order 1, $(\ddhash |x|)^{n}=0$ outside the origin. This is a general fact that will be discussed further in the next section, but is of course easy to verify directly. The last integral is therefore independent of $R'$, which proves the claim. (The reader who so prefers can also replace $|x|$ here by $(1+|x|^2)^{1/2}$ and verify that one gets convergent integrals in this case.)

The proof of the second part is similar. Replacing $\phi_i$ by $\phi_i+\epsilon|x|$ (and letting $\epsilon$ tend to zero in the end) we may assume that $\phi_j^\circ\geq \psi_j^\circ+\epsilon |x|$. Clearly, we may also assume that $\psi_i(0)=\phi_i(0)=0$ for all $i$. Fix $R>0$. Then
$$
\phi_i(tx)/t \uparrow \phi^\circ(x).
$$
Since $\phi_i^\circ$ is continuous it follows from Dini's lemma that the convergence is uniform on $|x|=1$. Hence, if $|x|>R'/2$, ($R'$ sufficiently large)
$$
\phi_i(x)>\phi_i^\circ(x)-\epsilon|x|/2.
$$
We also have that $t\to \psi(tx)/t$ increases so
$$
\psi_i(x)\leq\psi_i^\circ(x)\leq\phi_i^\circ(x)-\epsilon|x|
$$
for all $x$. Hence
$$
\psi_i(x)<\phi_i(x)-\epsilon/2|x|
$$
if $|x|>R'/2$. Put
$$
\Phi_i(x)=\max(\psi_i(x), \phi_i(x)-A),
$$
where $A$ is so large that $\Phi_i(x)=\psi_i(x)$ for $|x|<R$. If $R'$ is large enough we also have that $\Phi_i(x)=\phi_i(x)-A$ when $|x|>R'/2$. Then
$$
\int_{|x|<R} \ddhash\psi_1\w ...\ddhash\psi_n\leq \int_{|x|<R'} \ddhash\Phi_1\w ...\ddhash\Phi_n=\int_{|x|=R'}\dhash\phi_1\w ...\ddhash\phi_n=
$$
$$
\int_{|x|<R'} \ddhash\phi_1\w ...\ddhash\phi_n\leq \int_{\R^n} \ddhash\phi_1\w ...\ddhash\phi_n.
$$
Since $R$  is arbitrary, this completes the proof.
\end{proof}


\newpage
\section{Volumes and mixed volumes of convex bodies.}\label{volumes}

If $K$ is a convex body its support function is defined as
$$
h_K(x)=\sup_{y\in K} x\cdot y.
$$
It is immediate that support functions are convex and homogeneous of order 1. Conversely, if $\psi$ is convex and homogeneous of order 1 we can consider its Legendre transform
$$
\psi^*(y)=\sup_x x\cdot y-\psi(x).
$$
Clearly $\psi^*\geq 0$. If $\psi^*(y)>0$ for a certain $y$ we see that actually $\psi^*(y)=\infty$, since $ x\cdot y-\psi(x)$ is homogeneous of order 1 in $x$. Let $K=\{\psi^*=0\}$. It is a closed set, since $\psi^*$ is lower semicontinuous, and bounded since $\psi$ is of linear growth, so $K$ is compact. By the involutivity of the Legendre transform
$$
\psi(x)=\sup_y x\cdot y-\psi^*(y)=\sup_{y\in K} x\cdot y
$$
is the support function of $K$. As we have seen $K$ is also uniquely determined by $\psi$ which leads to the important fact that  there is  a one-to-one correspondence between convex bodies and 1-homogeneous convex functions.

The next proposition says that the map  $K\to \omega_K:=\ddhash h_K$ gives a similar one-to-one correspondence between the space of convex bodies modulo translation and the class of  $(1,1)$-currents of the form $\ddhash\psi$ where $\psi$ is convex and 1-homogeneous. In the next section we shall see which positive $(1,1)$-currents  arise in this way (Proposition \ref{deltaomega}).
\begin{prop}\label{modtranslation} If $K$ and $L$ are two convex bodies, $\omega_K=\omega_L$ if and only if there is an $a\in\R^n$ such that $L=K+a$.
\end{prop}
\begin{proof}
It is clear from the definition that
$$
h_{K+a}=h_K+x\cdot a.
$$
Taking $\ddhash$ this implies that $\omega_{K+a}=\omega_K$. Conversely, if 
$\omega_K=\omega_L$, then $\ddhash (h_L-h_K)=0$. Thus
$$
h_L=h_K +c +x\cdot a,
$$
for some $c\in \R$ and $a\in \R^n$. By 1-homogeneity, $c=0$, so $h_L=h_K+x\cdot a= h_{K+a}$. Since the support function determines the convex body, the proposition follows.
\end{proof}
This means that any translation invariant functional on the space of convex bodies must be of the form $K\to F(\omega_K)$. 
The next proposition from \cite{Lagerberg}  makes this explicit for the (Lebesgue) volume. (The proof here is different from the one in \cite{Lagerberg}.)
\begin{thm}\label{volumeformula} Let $\psi$ be a convex function in $\R^n$ of linear growth, and $\psi^\circ$ its indicator function. Then $\psi^\circ$ is the support function $h_K$ for a uniquely determined convex body $K$. We have
$$
V(\psi, ...\psi)=\int (\ddhash\psi)^n/n!=\int MA(\psi)=|K|,
$$
the Lebesgue measure of $K$.
\end{thm}
\begin{proof} We first note that if $K$ has empty interior, then $K$ is included in a hyperplane, that we can take to be $\{x_1=0\}$, and so has measure zero.  Then $h_K(x)$ depends only on  $x_2, ... x_n$, so $(\ddhash h_K)^n=0$. So the proposition holds in this case and we assume from now on that the interior of $K$ is non-empty. By Theorem \ref{MAestimate} it is enough to prove the proposition for {\it some} $\psi$ with $\psi^\circ=h_K$. 

A convenient choice (cf. \cite{Gromov}) is
$$
\psi(x):=\log \int_K e^{x\cdot y} d\lambda(y).
$$
Then $\psi$ is convex and $\psi^\circ=h_K$, since $L^t$-norms tend to the $L^\infty$-norm when $t\to\infty$. Put
$$
g(x):=\partial\psi(x),
$$
the gradient map of $\psi$. We claim that $g$ is a diffeomorphism from $\R^n$ to the interior of $K$. To start with
$$
\partial\psi(x)=\frac{\int_K y e^{x\cdot y} d\lambda(y)}{\int_K e^{x\cdot y} d\lambda(y)}
$$
is the barycenter of a probability measure with full support on $K$, so it must lie in the interior of $K$. 

To prove that $g$ is bijective we take a point in the interior of $K$, that we may assume is the origin. We need to prove that there is exactly one $x$ in $\R^n$ with $g(x)=0$. Since $\psi$ is smooth and strictly convex, this amounts to saying that $\psi$ has a unique minimum. But, if $0\in K^\circ$, 
$$
\psi(x)\geq \log\int_{|x|<\epsilon}e^{x\cdot y} d\lambda(y),
$$
so $\psi$ is proper. Hence $\psi$ has a minimum, which is unique by strict convexity. The claim of the theorem now follows from a change of variables $y=g(x)$;
$$
 V(\psi, ...\psi)=\int_{\R^n} \det(\psi_{j k}) d\lambda= \int_{y\in K} d\lambda=|K|.
 $$

\end{proof}

As a corollary we get a theorem from \cite{Lagerberg}.
\begin{thm}\label{Dirac}
Let $K$ be a convex body and $h_K$ its support function. Then
$$
MA(h_K)=|K|\delta_0,
$$
where $\delta_0$ is the unit Dirac mass at the origin.
\end{thm}
\begin{proof} We know that $MA(h_K)$ is a positive measure. Since $h_K$ is 1-homogeneous the Monge- Amp\`ere measure must vanish outside the origin. This is evident when $h_K$ is smooth (since  $MA(\psi)=\det(\psi_{j k}) d\lambda$ for smooth functions), and again we refer to the next section for the general case. Hence
$$
MA(h_K)=c\delta_0
$$
for some $c$ and $c=|K|$ by Theorem   \ref{volumeformula}.
\end{proof}

Next we consider an $n$-tuple of convex bodies, $K_1, ...K_n$ and their weighted Minkowski sum
$$
K_t:= t_1K_1+ ...t_nK_n= \{ \sum t_ix_i; x_i\in K_i\}.
$$
It is immediate that 
$$
h_{K_t}=\sum t_i h_{K_i},
$$
so
$$
|K_t|=\int(\sum t_i\ddhash h_{K_i})^n/n!
$$
is a polynomial in $t$. The coefficient of $t_1...t_n$ here is 
$$
\int \ddhash h_{K_1}\w ...\ddhash h_{K_n}=V(h_{K_1}, ...h_{K_n}) n!.
$$
The {\it mixed volume } of $K_1, ...K_n$ is defined as 
$$
V(K_1, ...K_n):=V(h_{K_1}, ...h_{K_n})=\int \ddhash h_{K_1}\w ...\ddhash h_{K_n}/n!.
$$
Hence $V(K, ...K)=|K|$ and $V(K_1, ...K_n)$ is symmetric under permutations. Moreover it is Minkowski linear in each variable so can be seen as a polarization of the volume function. By Theorem \ref{MAestimate}
$$
V(K_1, ...K_n)=V( \psi_1, ...\psi_n)
$$
where $\psi_i$ are any convex functions of linear growth with indicators equal to $h_{K_i}$.

Because of the $n$-linearity of $V$ it is natural to extend its definition to functions, $f_i=\psi_i-\phi_i$,  that can be written as differences of convex functions of linear growth
$$
V(f_1, ...f_n)=\int \ddhash f_1\w ...\ddhash f_n/n!.
$$
Notice that such functions also have indicator functions $f_i^\circ=\psi_i^\circ-\phi_i^\circ$ and that $V(f_1, ...f_n)$ only depends on their indicators.

\newpage

\section{(Strongly) homogeneous forms.}\label{sectionhomform}

In the previous section, currents of the form
$$
\omega=\ddhash\psi
$$
where $\psi$ is convex and 1-homogeneous played an important role. In this section we shall first  give a characterization of  such currents, and then introduce the corresponding  notion for currents of higher bidegree, $(p,p)$.

 Let
$$
E:=\sum x_j\partial/\partial x_j
$$
be the Euler vector field on $\R^n$. Its flow, regarded as a vector field on $\C^n=\R^n_x\oplus \R^n_\xi$, is 
$$
F_t(x,\xi)=(e^tx, \xi).
$$ 
We also put
$$
E^\#:= \sum x_j\partial/\partial \xi_j,
$$
whose flow is 
$$
G_t(x, \xi)=(x, \xi+tx).
$$
We denote by $\delta$ contraction with the field $E$ and $\delta^\#$ contraction with $E^\#$. By Euler's formulas, $\delta(df)= E(f)=qf$ if $f$ is homogeneous of order $q$.

\begin{lma}\label{symmlemma} Let $\omega$ be a current of bidegree $(1,1)$, satisfying $d\omega=\delta^\#\omega=0$. Then $\omega$ is symmetric, so $\delta\omega=0$ too.
\end{lma}
\begin{proof}
Write $\omega=\sum \omega_{i j} dx_i\wedge d\xi_j$. The hypotheses imply that, first
$$
0=\delta^\# \sum \omega_{i j} dx_i\wedge d\xi_j=-\sum \omega_{i j} x_j dx_i,
$$
and then,
$$
0=d\delta^\# \sum \omega_{i j} dx_i\wedge d\xi_j=\sum \omega_{i j} dx_i\wedge dx_j
-\sum x_j(\partial\omega_{i j}/\partial x_k) dx_k\wedge dx_i.
$$
The last term on the right hand side vanishes since  since $d\omega=0$.  Hence 
$$
\sum \omega_{i j} dx_i\wedge dx_j=0,
$$
so $\omega$ is symmetric.
\end{proof}

\begin{prop}\label{deltaomega} 
A  $(1,1)$-current $\omega$ can be written $\omega=\ddhash\psi$ where $\psi$ is 1-homogeneous if and only if $d\omega=\delta^\#\omega=0$.
\end{prop}
\begin{proof}
If $\psi$ is 1-homogeneous, the partial derivatives of $\psi$,  $\psi_k=\partial\psi/\partial x_k$,  are homogeneous of order zero. Hence
$$
\delta^\#\omega=-\sum x_k\psi_{j k}dx_j=\sum E(\psi_j)dx_j=0
$$
by Euler's formulas.

For the converse we use that $\omega$ is symmetric by the lemma. Since $\omega$ is closed and symmetric, we get that for any $j$
$$
\sum_k \omega_{j k} dx_k
$$
is closed, so there are functions $\psi_j$ that solve 
$$
\partial\psi_j/\partial x_k =\omega_{j k}.
$$
Since $\omega$ is symmetric, $\alpha:=\sum_j\psi_j dx_j$ is also closed, so we can solve $d\psi=\alpha$. 
Therefore, we can always write $\omega=\ddhash\psi$ for some  function $\psi$. Then $\delta^\#\omega=0$ means 
$$
\sum_k x_k\psi_{j k}=0
$$
for all $j$. This means that all $\psi_j$ are homogeneous of order zero. Thus
$$
\phi(x):=\sum x_k\psi_{k }
$$
is 1-homogeneous and 
$$
\phi_j=\psi_j +\sum x_k\psi_{j k}=\psi_j.
$$
Hence $\dhash\phi=\dhash\psi$, so $\ddhash\phi=\ddhash\psi=\omega$.
\end{proof}
\noindent Since $\psi$ is 1-homogeneous it follows that the coefficients of $\omega$, $\psi_{j k}$ are homogeneous of order -1 under scaling in $x$. In other words
$$
F_t^*(\omega)=\omega,
$$
so, as a current, $\omega$ is homogeneous of order zero. This also follows from the fact that, since $\delta\omega=0$, the Lie derivative of $\omega$ along $E$ vanishes. However, not every closed symmetric current that is homogeneous of order zero can be written as $\ddhash$ of a 1-homogeneous function, as the example
$$
x_1^{-1} dx_1\w d\xi_1 =\ddhash x_1\log|x_1|
$$
shows. Hence, the condition $\delta^\#\omega=0$ is strictly stronger. In the  next definition we generalize this condition to forms of higher bidegree.

\begin{df} Let $\Omega$ be a  closed current with measure coefficients of bidegree $(p,p)$ on $\R^n\setminus \{0\}$. We say that $\Omega$ is strongly homogeneous if $\delta^\#\Omega=0$ and  write $\Omega\in SH_p$.
\end{df}

As we have seen, currents of the form $\omega=\ddhash\psi$ where $\psi$ is (convex and) 1-homogeneous are strongly homogeneous. Since contraction with a vector field is an antiderivation, it follows that any wedge product
$$
\Omega=\ddhash\psi_1\w ...\ddhash\psi_p
$$
is also strongly homogeneous. 

Strongly homogeneous currents turn out to have two important properties. The first is that they are automatically `strong', so the theory of Bedford and Taylor applies. The second property is that, just like in the case of bidegree $(1,1)$,  they are symmetric, so $\delta^\#\Omega=0$ implies that $\delta\Omega=0$ too, which in turn implies that $\Omega$ is invariant under scaling in the $x$-variable. 

To see this we will use the following important observation. Recall that the operator $T$ is defined by
$$
T=\sum dx_j\w \delta_{\xi_j},
$$
see (\ref{T}).
\begin{prop}\label{fundobservation} For any supercurrent $\Omega$,
$$
T\Omega=L_{E^\#}\Omega,
$$
the Lie derivative of $\Omega$ along the vector field $E^\#= \sum x_j\partial/\partial \xi_j$.
\end{prop}
\begin{proof} Since both sides are derivations and annihilate any form (or current) in $x$, $\sum_I
a_I(x) dx_I$, it is enough to check the claim for $\Omega= d\xi_j$. But
$T d\xi_j=dx_j$, and 
$$
\L_{E^\#}(d\xi_j)= \frac{d}{dt}|_{t=0}G_t^*(d\xi_j)= \frac{d}{dt}|_{t=0}d(\xi_j+tdx_j)=dx_j,
$$
which completes the proof.
\end{proof}

It follows  that a  strongly homogeneous current
 is invariant  under the flow of $E^\#$, 
$$
G_t(x, \xi)=(x, \xi+tx)
$$
since by Cartan's magic formula
$$
\L_{E^\#}\Omega=(d\delta^\#+\delta^\#d)\Omega=0.
$$
Thus, strongly homogeneous currents are annihilated by $T$. By Theorem \ref{theoremstrongform} this implies that they are strong. 

Using the invariance under the flow of $E^\#$ we can also generalize Lemma \ref{symmlemma} and prove that any strongly homogeneous current is symmetric, cf. \cite{Faifman-Wannerer}.
( I thank Thomas Wannerer for insisting that symmetry should be automatic and directing me to this reference. The proof we give here is different.)
\begin{prop}\label{symmetric} Let $\Omega$ be a superform of bidegree $(p,p)$ satisfying $d\Omega=0$ and $\delta^\#\Omega=0$. Then $\Omega$ is symmetric. 
\end{prop}
\begin{proof} As we have seen 
$$
\L_{E^\#}\Omega= (d\delta^\#+\delta^\# d)\Omega=0,
$$
so $\Omega$ is invariant under the flow $(x,\xi)\to (x,\xi +t x)$. Hence
$$
\Omega=\sum\Omega_{I J} dx_I\w d\xi_J=\sum\Omega_{I J} dx_I\w d(\xi+x)_J.
$$
Change variables $(y,\eta)=( x+\xi, \xi)$, $(x, \xi)=(y-\eta, \eta)$ . Then we have, in the new coordinates, 
$$
\sum\Omega_{I J} d(y-\eta)_I\w d\eta_J=\sum\Omega_{I J} d(y-\eta)_I\w d y_J.
$$
Identifying terms of bidegree $(p,p)$ in $(y,\eta)$ we get
$$
\sum\Omega_{I J} dy_I\w d\eta_J=\sum\Omega_{I J} d(-\eta)_I\w d y_J=\sum\Omega_{I J} dy_J\w d\eta_I,
$$
so $\Omega$  is symmetric. 
\end{proof}
From Proposition \ref{symmetric} we get that a strongly homogeneous current also satisfies $\delta\Omega=0$. Applying Cartan's formula again, we see that a strongly homogeneous current satisfies $L_E\Omega=0$, so it is invariant under the flow of $E$, $F_t(x,\xi)=(e^tx, \xi)$ too. Thus it is invariant under scaling in the $x$-variable, or equivalently, its coefficients are homogeneous of order $-p$, if $\Omega$ is of bidegree $(p,p)$.

Summing up the consequences of Propositions \ref{fundobservation} and \ref{symmetric} we then have.
\begin{prop} A strongly homogeneous current is strong and homogeneous of order zero, i.e. invariant under scaling in the $x$-variable.
\end{prop}


\newpage

\section{A pairing}

In this section we introduce a natural and useful pairing between strongly homogeneous currents and 1-homogeneous functions. Recall that we defined a gauge function as a convex 1-homogeneous function $\mu$ that is strictly positive outside the origin. Its associated unit sphere, $S_\mu$, is the set where $\mu=1$. We also defined the supercurrent of $S_\mu$ by
$$
[S_\mu]_s.\chi=\int_{S_\mu} \dhash\mu\w\chi
$$
if $\chi$ is a smooth form of bidegree $(n-1,n-1)$. We then showed that the wedge product
$$
[S_\mu]_s\w\Omega
$$ 
is well defined. By Proposition \ref{sumup} it is  continuous in $\Omega$ in the weak*-topology on the space of  $\Omega$:s  that are homogenuos of order zero,  weakly positive and `strong',  and also continuous under uniform convergence in $\mu$.  Hence, by the remarks just made, this applies in particular to forms that are strongly homogeneous.

\begin{df} Let $\mu$ be a gauge function and $f$ a continuous 1-homogeneous function, and let $\Omega$ be a  current in $SH_{n-1}$. Then 
$$
\langle f,\Omega\rangle_\mu=\int_{S_\mu} f\dhash\mu\w \Omega=[S_\mu]_s. f\Omega.
$$
\end{df}

\begin{prop}\label{independent} The pairing $\langle f,\Omega\rangle_\mu$ does not depend on $\mu$.
\end{prop}
The proof uses a simple but fundamental lemma that we will have use for in other contexts as well.
\begin{lma}\label{fundamentallemma} Let $\Omega$  be a current of bidegree $(n-1,n-1)$ satisfying $\delta^\#\Omega=0$, and let $f$ and $g$ be 1-homogeneous and smooth. Then
\be\label{fundeq}
f\dhash g\w\Omega=g\dhash f\w\Omega.
\ee
\end{lma}
\begin{proof} Applying $\delta^\#$ to both sides we get
$$
\delta^\#(f\dhash g\w\Omega)=fg\Omega= \delta^\#(g\dhash f\w\Omega).
$$
Since both sides in (\ref{fundeq}) are of full degree in $d\xi$ this implies that they are equal.
\end{proof}
In the same way one sees that a strongly homogeneous current, $\Omega$, of bidegree  $(n,n)$, like $\ddhash\psi_1\w ...\ddhash\psi_n$ where $\psi_i$ are 1-homogeneous,  vanishes outside the origin, since $\delta\Omega=0$ -- a fact that we have used repeatedly. Given the lemma we can now prove Proposition \ref{independent}.

\begin{proof}{(\it Of Proposition \ref{independent})} Let $\mu$ and $\mu'$ be two gauges, that we may assume smooth by the continuity properties just mentioned. Notice that then their unit spheres are smooth, by Lemma \ref{smoothgauge} in the appendix. Assume first that $S_{\mu'}\subset\{\mu<1\}$. We want to prove that
$$ 
\langle f,\Omega\rangle_\mu=\int_{S_\mu} f\dhash\mu\w \Omega= \int_{S_{\mu'}} f\dhash\mu'\w \Omega=\langle f,\Omega\rangle_{\mu'}
$$
for all continuous 1-homogeneous functions $f$ and we may of course assume that $f$ is smooth. Note first that
$$
\ddhash f\w \Omega=0
$$
outside the origin since it is strongly homogeneous of full degree. Therefore
$$
\langle f,\Omega\rangle_\mu= \int_{S_\mu} f\dhash\mu\w \Omega
= \int_{S_\mu} \mu\dhash f\w \Omega=\int_{S_\mu} \dhash f\w \Omega=\int_{S_\mu'} \mu'\dhash f\w \Omega=\langle f,\Omega\rangle_{\mu'}
$$
where we have used in turn Proposition \ref{fundeq}, that $\mu=1$ on $S_\mu$ and Stokes' theorem. If $S_{\mu'}$ is not a subset of the set where $\mu<1$ we can take a third gauge $\mu''$ such that both $\mu$ and $\mu'$ are larger than 1 on $S_{\mu''}$. Then the pairings defined by $\mu$ and $\mu'$ are both equal to the pairing defined by $\mu''$, hence equal.
 
\end{proof}
\begin{prop}\label{altdef} Let  $\Omega=\ddhash\psi_1\w ...\ddhash\psi_{n-1}$ and $f=\psi_n$, all $\psi_i$ being convex and 1-homogeneous. Then  the pairing becomes
$$
\langle \psi_n,\ddhash\psi_1\w ...\ddhash\psi_{n-1}\rangle_\mu=V(\psi_1, ...\psi_n)n!.
$$
\end{prop}
\begin{proof} We have
$$
\langle \psi_n,\ddhash\psi_1\w ...\ddhash\psi_{n-1}\rangle_\mu=\int_{\mu(x)=1} \psi_n \ddhash\psi_1\w ...\ddhash\psi_{n-1}\dhash\mu=
\int_{\mu(x)=1}  \ddhash\psi_1\w ...\ddhash\psi_{n-1}\w\dhash\psi_n=
$$
$$
\int  \ddhash\psi_1\w ...\ddhash\psi_{n},
$$
where we have used Lemma \ref{fundamentallemma}.
\end{proof}

We remark that Proposition \ref{independent} can be localized at no extra cost: If $\K$ is an open cone in $\R^n$ with vertex at the origin we have
$$
\int_{S_{\mu}\cap \K} f\Omega=\int_{S_{\mu'}\cap \K} f\Omega.
$$
This is simply because the characteristic function of the cone is homogeneous of order zero. Admittedly, it is not continuous, but this poses no problems since it can be approximated by an increasing  sequence of  smooth functions, homogeneous of order zero. Yet another rewrite of Proposition \ref{independent} is the following
\begin{prop}\label{pushforward}
Under the same assumptions  as in Proposition \ref{independent} define measures
$$
d\nu=\dhash\mu\w\Omega, \quad d\nu'=\dhash\mu'\w\Omega
$$
on $S_\mu$ and $S_{\mu'}$ respectively. Let $\pi$ be the radial projection from $S_{\mu'}$ to $S_\mu$. Then
$$
\pi_*(d\nu')=\mu' d\nu.
$$
\end{prop}
\begin{proof}
Note first that 
$$
\pi(x)=x/\mu(x)
$$
so the proposition  says that if $f$ is any continuous function on $S_{\mu}$, then
\be\label{blabla}
\int_{S_\mu} f \mu' d\nu=\int_{S_{\mu'}} f(x/\mu(x))d\nu'.
\ee
We know from Proposition \ref{independent} that if $g$ is 1-homogeneous, then
$$
\int_{S_\mu} g d\nu=\int_{S_{\mu'}} gd\nu'.
$$
Taking $g=f(x/\mu(x))\mu'$ we get (\ref{blabla}).
\end{proof}
Our next objective is to prove that the pairing, that we can now denote as $\langle f,\Omega\rangle$, is non degenerate. We deduce one part of that from a more general proposition.
\begin{prop}\label{firstpart} Let $\Omega$ be a closed positive strongly homogeneous  current of bidegree $(n-1,n-1)$. Assume that
$$
\langle f,\Omega\rangle=0
$$
for all smooth $f$.
Then $\Omega=0$ in $\R^n\setminus\{0\}$.
\end{prop}
\begin{proof}
By Proposition \ref{independent} it is enough to prove that $\Omega=0$ if $[S_\mu]\w\Omega =0$ for {\it some} $\mu$, that we may take to be smooth.  Then the hypothesis means that
$$
d\mu\w\dhash\mu\w\Omega=0
$$
when $\mu=1$. 
Applying $\delta$ and then $\delta^\#$ we get, since $\delta(d\mu)=\mu$ and $\delta^\#\dhash\mu=\mu$ that $\Omega=0$ when $\mu=1$. Since the coefficients of $\Omega$ are homogeneous of order $-p$, $\Omega$ vanishes outside the origin. 

\end{proof}
This proves one part of nondegeneracy. 
We now turn to the remaining part.
\begin{thm}\label{secondpart}
Let $\mu$ be a gauge and let $d\nu$ be a measure on $S_\mu$. Then there is a unique strongly homogeneous (hence closed) current in $\R^n\setminus\{0\}$, $\Omega$, such that $\Omega\w\dhash\mu=d\nu$, i.e.
$$
\int_{S_\mu} fd\nu=\int_{S_\mu} f\Omega\w\dhash\mu,
$$
for any continuous $f$. If $d\nu$ is positive, $\Omega$ is positive. Moreover, the trivial extension of  $\Omega$, $\tilde\Omega$,  to all of  $\R^n$ is closed if and only if $d\nu$ has barycenter zero. 
\end{thm}
Conversely, we already know that if $\Omega$ is weakly positive and strongly homogeneous of bidegree $(n-1,n-1)$ then its trace on $S_\mu$, is a well defined positive measure measure. Hence there is a one-to-one correspondence between positive measures on, say, the sphere, and weakly positive currents in $SH_{n-1}$. We also remark that any homogeneous current of bidegree $(p,p)$ ($p<n$) has a trivial extension across the origin. When $p<n-1$, the singularity at the origin is too weak to give a contribution to $d\tilde\Omega$, but when $p=n-1$ we may get Dirac masses at the origin as the theorem shows. 
\begin{proof} By Proposition \ref{pushforward} it is enough to prove this for $\mu(x)=|x|$. We define $\Omega$ directly and then show that it has the desired properties. Let
$$
\eta=\sum \eta_{i k} dx_i\w d\xi_k
$$
be a smooth form with compact support in $\R^n$. We define an $(n-1,n-1)$ current $\Omega$ by
$$
\Omega.\eta=\sum_{i k}\int_{|x|=1} x_ix_k d\nu(x)\int_0^\infty \eta_{i k}(tx)dt.
$$
Clearly, $\Omega$ is positive if $d\nu$ is positive. (Recall that all positivity notions coincide in bidegree $(n-1,n-1)$.)
Then $\Omega$ is a symmetric current with measure coefficients. 
In order to prove that $\delta\Omega=0$ outside the origin we also define an $(n,n-1)$ current $\rho$ outside the origin by
$$
\rho.\tau=\sum_k\int_{|x|=1} x_k d\nu(x)\int_0^\infty \tau_k(tx) dt/t,
$$
if $\tau=\sum \tau_kd\xi_k$ is a smooth $(0,1)$ form compactly supported in $\R^n\setminus\{0\}$. If $\eta=0$ near the origin we have, since $\delta\eta=\sum x_i\eta_{i k}d\xi_k$, that
$$
\rho.\delta\eta=\sum_{i k}\int_{|x|=1} x_ix_k d\nu(x)\int_0^\infty \eta_{i k}(tx)dt=\Omega.\eta.
$$
Hence $\delta\rho=\Omega$ so $\delta\Omega=0$. 

Let now $\eta=d\tau$, $\tau$ $(0,1)$ with compact support in $\R^n$. Then
$$
\Omega. d\tau=\sum_{i k}\int_{|x|=1} x_ix_k d\nu(x)\int_0^\infty \frac{\partial\tau_k}{\partial x_i}(tx)dt =\sum_k\int_{|x|=1} x_k d\nu(x)\int_0^\infty (d/dt)\tau_k(tx) dt=
$$
$$
=-\sum \tau_k(0)\int_{|x|=1} x_k d\nu(x).
$$
This means that
$$
d\Omega.\tau=\sum \tau_k(0)\int_{|x|=1} x_k d\nu(x)
$$
so $d\Omega$ is zero outside the origin and equal to zero across the origin  if and only if the barycenter of $d\nu$ is zero. It remains only to prove that the trace of $\Omega$ on the unit sphere is $d\nu$. This is equivalent to saying that for any smooth function $f$ supported outside the origin (with $\mu(x)=|x|$) 
\be\label{f}
\int_{|x|<1} df\w\Omega\w\dhash\mu=\int_{|x|=1} f d\nu-\int_{|x|<1}f\Omega\wedge \ddhash\mu.
\ee
But the LHS is 
$$
\sum_{i k}\int_{|x|=1} x_ix_k d\nu(x)\int_0^\infty f_i(tx)\mu_k(tx)dt =\int_{|x|=1}d\nu \int_0^1 (d/dt)f(tx) (d/dt)\mu(tx)dt.
$$
The formula (\ref{f}) therefore follows from integration by parts if we use that $(d/dt)\mu(tx)=\mu=1$ when $|x|=1$.

\end{proof}
\begin{remark} Notice that since any measure on the sphere can be written as a difference between two positive measures, the currents $\Omega$ that we get from the theorem are differences between two weakly positive currents. 
\end{remark}

We conclude this section with a few results  related to non-degeneracy that will be important in Section \ref{valuations}. 
\begin{prop}\label{symmetry}
Let $\Omega$ be a strongly homogeneous and weakly positive current of bidegree $(p,p)$. Put 
$$
\langle \psi_{n-p},\Omega\w\ddhash\psi_1\w...\ddhash\psi_{n-p-1}\rangle=V_\Omega(\psi_1,...\psi_{n-p}).
$$
Then $V_\Omega$ is symmetric in the $\psi_i$. 
\end{prop}
\begin{proof} The assumption on weak positivity implies that $\Omega\w\ddhash\psi_1\w...\ddhash\psi_{n-p-1}$ is a well defined current. It is clear that $V_{\Omega}$ is symmetric in $\psi_i$ for $i<n-p$. What we need to prove is that $\psi_{n-p}$ and e.g. $\psi_1$ can be interchanged. This follows from Stokes' formula
$$
\langle \psi_{n-p},\Omega\w\ddhash\psi_1\w...\ddhash\psi_{n-p-1}\rangle=
$$
$$
=\int_{\mu(x)=1}\psi_{n-p}\Omega\w\ddhash\psi_1\w...\ddhash\psi_{n-p-1}\w\dhash\mu=
\int_{\mu(x)=1}\Omega\w\ddhash\psi_1\w...\ddhash\psi_{n-p-1}\w\dhash\psi_{n-p}=
$$
$$
=\int_{\mu(x)=1}\Omega\w\dhash\psi_1\w...\ddhash\psi_{n-p-1}\w\ddhash\psi_{n-p}.
$$
\end{proof}

\begin{prop}\label{mainpoint}
Let $\Omega$ be a current of bidegree $(p,p)$ that can be written as a difference between two weakly positive strongly homogeneous currents. Assume that 
$$
\langle \psi,\Omega\w(\ddhash\psi)^{n-p-1}\rangle=0
$$
 for all  $\psi$  smooth, convex and 1-homogeneous. Then $\Omega=0$
\end{prop}
\begin{proof}
By polarization and the previous proposition it follows that
$$
\langle \phi,\Omega\w(\ddhash\psi)^{n-p-1}\rangle=0
$$
for any smooth convex and 1-homogeneous functions $\phi$ and $\psi$. Since any smooth function can be written as a difference of two convex functions, this actually holds even if $\phi$ is not convex. Then it follows from Proposition \ref{firstpart} that $\Omega\w(\ddhash\psi)^{n-p-1}=0$ for any 1-homogeneous convex $\psi$ . Restrict this to a hyperplane, that we take as $\{x_1=1\}$.
 Then we get  that  $\Omega\w (\ddhash\psi)^{n-p-1}=0$ when restricted to this hyperplane. By Proposition \ref{Current-Criterion} (in $\R^{n-1}$),   $\Omega$ vanishes on the hyperplane (intersected with $\K$), i. e. 
$$
\Omega\w dx_1\w d\xi_1=0
$$
when $x_1=1$ (note that if $\Omega$ lies in $S$, then its restriction to a hyperplane lies in $S$ there). Applying first $\delta$ and then $\delta^\#$ we get that $\Omega=0$ when $x_1=1$, and therefore in all of $\K$ by homogeneity.
\end{proof}
\begin{prop}\label{weaklypositive}
Let $\Omega$ be a  current of bidegree $(p,p)$ that can be written as a difference between two weakly positive strongly homogeneous currents. Assume that 
$$
\langle f,\Omega\w\Omega'\rangle\geq 0
$$
 for all 1-homogeneous smooth $f\geq 0$ and all $\Omega'=(\ddhash\psi)^{n-p-1}$ where $\psi$ is smooth, convex and 1-homogeneous. Then $\Omega$ is weakly positive.
\end{prop}
\begin{proof}
We can start the proof exactly as in the previous proposition and arrive at the conclusion that $\Omega$ is weakly positive  on the hyperplane $x_1=1$, i.e. that
$$
\Omega\w dx_1\w d\xi_1\w F'\geq 0
$$
for all elementary forms $F'$ of bidegree $(n-p-1,n-p-1)$, using Proposition \ref{positive} instead of Proposition \ref{Current-Criterion}. We need to prove that this implies that $\Omega$ is weakly positive. Let $F$ be an elementary form. We start from the identity
$$
x_1^2\Omega\w F=\Omega\w dx_1\w d\xi_1\w \delta^\#\delta F,
$$
which is easily proved using that $\delta$ and $\delta^\#$ are antiderivations. The proposition will therefore follow if we can prove that $\delta^\#\delta F$ is an elementary form when $F$ is elementary. 

Write $F=\alpha_1\w\alpha_1^\#\w...\alpha_p\w\alpha_p^\#$. Let $V:=[\alpha_1, ...\alpha_p]$ and let $N:=\{\alpha; \delta(\alpha)=0\}$. If $V$ is included in $N$ there is nothing to prove. Otherwise $N\cap V$ has codimension 1 in $V$ and we may choose a basis for $V$ such that $e_1, ...e_{p-1}$ lie in $N$. Then 
$$
F=c e_1\w e_1^\#\w...e_p\w e_p^\#, \, \, c>0,
$$
and
$$
\delta^\#\delta F=(\delta(e_p))^2 c e_1\w e_1^\#\w...e_{p-1}\w e_{p-1}^\#
$$
which is an elementary form.
\end{proof}


\newpage
\section{Charts.}\label{charts}

If $\Omega$ is a superform on $\R^n$ and $V$ is a hyperplane in $\R^n$, by the {\it restriction} of $\Omega$ to $V$ we mean the restriction of the form $\Omega=\sum \Omega_{I J} dx_I\wedge d\xi_J$ to the complexification of $V$, which is a complex hyperplane in $\C^n$. More concretely, if $V=H_1:=\{ x_1=1\}$, the restriction of $\Omega$ is
$$
r(\Omega):=\sum_{1\notin I\cup J}\Omega(1, x') dx_I\wedge d\xi_J.
$$
Observe that if $\Omega$ is strongly homogeneous, then $\Omega$ is uniquely determined on $\R^n_+:=\{x_1>0\}$ by its restriction to $H_1$. To see this, note first that
when $x\in H_1$,
$$
\Omega =\delta^\#\delta (dx_1\wedge d\xi_1\wedge\Omega)=
\delta^\#\delta (dx_1\wedge d\xi_1\wedge r(\Omega)),
$$
since $\Omega$ is annihilated by $\delta$ and $\delta^\#$. But, since the coefficients of $\Omega$ are homogeneous, this determines $\Omega$ on all of $\R^n_+$.

Choosing a suitable collection of hyperplanes (like $H_j:= \{x_j=1\}$), $\Omega$ is therefore determined everywhere on $\R^n$ by its restrictions to these hyperplanes. By homogeneity, $\Omega$ can be thought of as defined on $\R^n/\R_+$, and  we will consider such a family of hyperplanes as a collection of local charts on this space.  The corresponding restrictions then represent $\Omega$ in these local charts, and we are  led to ask: Which forms on $H_1$ appear as restrictions of strongly homogeneous forms?

Let us first change notation and consider instead homogeneous forms on $\R^{n+1}=\{(x_0, ...x_n)\}$ and look at restrictions to $H_0=\{x_0=1\}$ that we may identify with $\R^n$.  A function $\phi$ on $\R^n=H_0$ can always be extended to a (unique) homogeneous function
$$
\Phi(x_0,x)=x_0\phi(x/x_0)
$$
on $\R^{n+1}_+$. It is a well known fact that if $\phi$ is convex, $\Phi$ is also convex. (A simple proof that I learned from Mattias Jonsson starts by observing that if $\phi(x)=c+a\cdot x$ is affine, then $\Phi$ is linear. Writing $\phi$ as the supremum of affine functions one sees that $\Phi$ is the supremum of linear functions, hence convex.) Applying $\ddhash$ we see that any symmetric and closed $(1,1)$ -form, $\omega$, on $\R^n$ extends to a strongly homogeneous form on $\R^{n+1}_+$, which is positive if $\omega$ is positive. 

Our next theorem  shows that the situation for higher bidegrees is a little bit more complicated; for a form to have a strongly homogeneous closed extension it must lie in the kernel of the operator $T$ (remember that this is automatically satisfied for forms of bidegree $(1,1)$). 

\begin{thm}\label{extending} Let $\Omega'$ be a smooth $(p,q)$-form on $H_0=\R^n$. Then there is a unique $(p,q)$-form on $\R^{n+1}_+$ that restricts to $\Omega'$ on $H_0$ and satisfies
\be\label{homogeneousforms}
L_E(\Omega)=0, \quad \delta\Omega=0, \, \text{and}\,\,\delta^\#\Omega=0.
\ee

Moreover, $d\Omega=0$ if and only if 
$$
d\Omega'=0
$$
on $H_0$, and 
$$
 T'\Omega':=\sum_1^n dx_i\wedge \delta_{\partial_{\xi_i}}\Omega'=0.
$$
More generally, if $T'(\Omega')=0$, the extension operator $\Omega'\to \Omega:=\E(\Omega')$ commutes with the exterior derivative $d$, so that $d\E(\Omega')=\E(d\Omega')$.

If $p=q$ and $\Omega'$ is weakly (strongly) positive, then $\Omega$ is weakly (strongly) positive.
\end{thm}
\begin{proof}
That the Lie derivative $L_E(\Omega)$ vanishes  means that all coefficients of $\Omega$ are homogeneous of order $-p$. Assume first that $\Omega$ exists. Then, for $x$ on $H_0$, 
$$
\Omega=\delta^\#\delta(\frac{dx_0\wedge d\xi_0}{x_0^2}\wedge\Omega)=
\delta^\#\delta(\frac{dx_0\wedge d\xi_0}{x_0^2}\wedge\Omega').
$$
Since the coefficients of $\Omega$ are homogeneous, this means that $\Omega$ is uniquely determined everywhere. Here we have just followed the same argument as in the beginning of this section.

To prove existence, we first let $\Omega''$ be an arbitrary extension of $\Omega'$ with coefficients homogeneous of order $-p$ and define
$$
\delta^\#\delta(\frac{dx_0\wedge d\xi_0}{x_0^2}\wedge\Omega'').
$$
Then the coefficients of $\Omega$ have the same homogeneity, so $L_E(\Omega)$ vanishes. Finally,
$$
\Omega=\Omega''+ dx_0\wedge\alpha+d\xi_0\wedge\beta,
$$
for some $\alpha,\beta$, so $\Omega$ restricts to $\Omega'$. 

We next investigate when $\Omega$ is closed. Assume first that this holds. Then, clearly $d\Omega'=0$ and since, by Cartan's  formula, 
$$
L_{E^\#}=\delta^\# d+d\delta^\#=T (=\sum_0^n dx_i\wedge \delta_{\xi_i}),
$$
$T\Omega=0$. But it is easy to check that $T\Omega=T'\Omega'$ modulo terms that contain a factor $dx_0$, so restricting to $H_0$ this gives that $T'\Omega'=0$. 
Assume now that, conversely, $d\Omega'=0$ and $T'\Omega'=0$. 
Define a map  $F$ from $\R^{n+1}_+$ to $H_0$ by
$$
F(x_0,x):=(1,x/x_0),
$$
and this time make the particular choice of  $\Omega''$ as $\Omega''=F^*(\Omega')$. Then, as before, the coefficients of $\Omega''$ are homogeneous of order $-p$, so $L_E(\Omega'')=0$,  and, moreover, $d\Omega''=0$. Put
$$
A:= \frac{dx_0\wedge d\xi_0}{x_0^2}\wedge\Omega''.
$$
Then $dA=0$ and
$$
\Omega=\delta^\#\delta A=-\delta\delta^\# A.
$$
Hence, using Cartan's formula again, 
$$
-d\Omega=(d\delta)\delta^\#A= L_E(\delta^\#A) -\delta d\delta^\#(A).
$$
But $\delta^\#A$ is of bidegree $(p+1,q)$ and its coefficients are homogeneous of order $-(p+1)$, so the first term vanishes. As for the second term
$$
\delta d\delta^\#(A)=\delta T(A) - \delta\delta^\# dA.
$$
The second term vanishes since $A$ is closed, and the first term vanishes since when $x_0=1$
$$
TA=T(\frac{dx_0\wedge d\xi_0}{x_0^2}\wedge\Omega')=\frac{dx_0\wedge d\xi_0}{x_0^2}\wedge T'\Omega'=0.
$$
Hence $d\Omega=0$, as claimed. If $\Omega'$ is not closed (but still satisfies $T'\Omega'=0$), the same argument shows that 
$$
d\Omega=d\E(\Omega')= \delta\delta^\# dA=\E(d\Omega'),
$$
since
$$
dA= \frac{dx_0\wedge d\xi_0}{x_0^2}\wedge d\Omega''=\frac{dx_0\wedge d\xi_0}{x_0^2}\wedge G^*_{x_0}(d\omega')
$$
( where $G_{x_0}(x)= x/x_0$).
For the last part, assume that $\Omega'$ is weakly positive, and let $\alpha$ be an elementary form of bidegree $(n+1-p, n+1-p)$. We need to show that $\Omega\wedge\alpha$ is non-negative. But, when $x_0=1$, 
$$
\Omega\wedge\alpha=(\delta^\#\delta A)\wedge\alpha=A\wedge\delta^\#\delta\alpha=\Omega'\wedge dx_0\wedge d\xi_0
\wedge\delta^\#\delta\alpha.
$$
Since $\delta^\#\delta\alpha$ is an elementary form, the right hand side is non-negative. Hence $\Omega$ is weakly positive when $x\in H_0$ and therefore everywhere by homogeneity. The case of strong positivity follows since $\Omega$ is elementary if $\Omega'$ is elementary (see the proof of Proposition \ref{weaklypositive}). 
\end{proof}

We shall call forms satisfying conditions (\ref{homogeneousforms}) homogeneous in the sequel. Thus a strongly homogeneous form is a homogeneous form that is also closed.

In conclusion, there is a one-to-one correspondence between (weakly positive) strongly homogeneous forms on $\R^{n+1}_+$ and (weakly positive) closed forms on $\R^n=H_0$ that are `strong', i.e. lie in the kernel of $T'$. 
\begin{df} A $(p,p)$-form $\Omega'$ on $\R^n=H_0$ is {\it regular} if its homogeneous extension, $\Omega$,  from Theorem \ref{extending} extends smoothly to the closure of $\R^{n+1}_+$ outside the origin.
\end{df}
Notice that when $p=0$ so that we are dealing with a function $\phi$ on $H_0$, the restriction of its homogeneous extension $\Phi$ to $\{(0,x); x\neq 0\}$ is just
$$
\lim_{x_0\to 0} x_0\phi(x/x_0)=\phi^\circ(0),
$$
its {\it indicator function} defined in section \ref{Bedford-Taylor}.

In the same way, one sees that, when $p>0$, the restriction of $\Omega$ to $\R^n\setminus\{0\}= \{(0,x); x\neq 0\}$ is 
$$
\Omega'^\circ:=\lim_{x_0\to 0}G^*_{x_0}(\Omega'),
$$
where $G_{x_0}(x)=x/x_0$, provided $\Omega'$ is regular so that we know the limit exists. 


\section{ A Poincaré type inequality}

Let $\mu$ be a gauge function on $\R^{n+1}$ and let
$$
S_\mu^+:=\{(x_0, x); x_0>0, \mu(x_0,x)=1\},
$$
be its associated positive half-sphere.  If $D$ is a domain in $H_0$, we let 
$$
S(D)=\{(x_0,x)\in \R^n_+; \mu(x_0,x)=1, (1,x/x_0)\in D\}
$$
be the corresponding domain in $S_\mu^+$. Note that if $D=\{\mu(1,x)<t\}$, then 
$$
S(D)=S_\mu^+\cap\{x_0>1/t\}
$$
 is a `spherical cap'. If $D=H_0$, $S(D)=S_\mu^+$. 

\begin{prop}\label{bdry} Let $\Omega$ be strongly homogeneous in  $\R^{n+1}_+$ of bidegree $(n,n)$ and let $D$ be any domain in $H_0$. Then, 
$$
\int_{D} f\Omega\wedge \dhash x_0=\int_{S(D)} f\Omega\wedge \dhash\mu,
$$ 
if $f$ is homogeneous of order 1,  and 
 $$
\int_{D} g\Omega\wedge \dhash x_0=\int_{S(D)} g\Omega\wedge \dhash x_0,
$$
if $g$ is homogeneous of order zero. 
\end{prop}
\begin{proof} We follow the proof of Proposition \ref{independent}: 
If $f$ is smooth we have by Lemma \ref{fundamentallemma} and Stokes' theorem
$$
\int_{H_0} f\Omega\wedge \dhash x_0=\int_{H_0} \Omega\wedge \dhash f=
\int_{S_\mu^+} \Omega\wedge \dhash f= \int_{S_\mu^+}f \Omega\wedge \dhash \mu,
$$
since $\Omega\wedge \dhash f$ is closed. By approximation, the same formula holds if $f$ is just, say, bounded and Borel measurable. If $g$ is homogeneous of order zero, we can apply this to $f=g x_0$ and get 
$$
\int_{H_0} g\Omega\wedge \dhash x_0=\int_{S_\mu^+} g x_0\Omega\wedge \dhash\mu = \int_{S_\mu^+} g \Omega\wedge \dhash x_0,
$$
since
$$
x_0\Omega\wedge\dhash\mu/=\mu\Omega\wedge \dhash x_0
$$
by Lemma \ref{fundamentallemma}.

We then  replace $f$ by $f\chi$ and  $g$ by $g\chi$ respectively,
where $\chi$ is the characteristic function of the cone $\{(x_0,x); x/x_0\in D\}$.

\end{proof}
We shall now see that this proposition gives a family of  Poincaré-type inequalities on sets like $S(D)$, where $D$ is a domain  in $H_0=\R^n$. ( Recall that this class of domains includes spherical caps.) We start by considering the case when $D$ is relatively compact in $H_0$. The inequality (\ref{curious}) is stated for 1-homogeneous functions, but since both sides only depend on the restriction of the function $v$ to $S(D)$, it holds for any function on $S(D)$. 

\begin{thm}\label{Wirtinger}
 Let $v$ be a smooth 1-homogeneous function on $\R^{n+1}_+$, and let $\Omega$ be strongly homogeneous on $\R^{n+1}_+$ of bidegree $(n-1,n-1)$. Assume $v$ vanishes on the boundary of $D$ where $D$ is a relatively compact domain in  $H_0$. Then
$$
\int_{D} dv\wedge \dhash v\wedge\Omega\wedge \dhash x_0=
\int_{S(D)} dv\wedge \dhash v\wedge \Omega\wedge \dhash\mu-
\int_{S(D)}v^2\Omega\wedge\ddhash\mu\wedge\dhash\mu.
$$
In particular, if $\Omega$ is positive on $H_0$,
\be\label{curious}
\int_{S(D)}v^2\Omega\wedge\ddhash\mu\wedge\dhash\mu\leq
\int_{S(D)}dv\wedge \dhash v\wedge \Omega\wedge \dhash\mu,
\ee
with equality only if $v=0$ if $\Omega$ is strictly positive on $H_0$. (Recall that, on $H_0=\R^n$ all notions of positivity agree for forms of bidegree  $(n-1,n-1)$. Strict positivity means that the coefficient matrix of $\Omega$ is positive definite.)
\end{thm}
\begin{proof}By Stokes' theorem
$$
\int_{D} dv\wedge \dhash v\wedge\Omega\wedge \dhash x_0=
-\int_{D} v \ddhash v\wedge\Omega\wedge \dhash x_0,
$$ 
when $v$ vanishes on the boundary of $D$. By the previous theorem, the RHS equals
$$
-\int_{S(D)}v\ddhash v\wedge\Omega\wedge\dhash\mu
$$
(since $\ddhash v\wedge\Omega$ is strongly homogeneous) and, applying Stokes' again, this equals
$$
\int_{S(D)} dv\wedge\dhash v\wedge \Omega\wedge \dhash\mu -
\int_{S(D)}v\dhash v\wedge\Omega\wedge \ddhash\mu.
$$
By the Lemma \ref{fundamentallemma} again, when $\mu=1$, 
$$
\dhash v\wedge \Omega\wedge \ddhash\mu=v\dhash\mu\wedge\Omega\wedge\ddhash\mu
$$
which completes the proof of (\ref{curious}). If equality holds, we must have
$$
\int_{D} dv\wedge \dhash v\wedge\Omega\wedge \dhash x_0=0,
$$
which implies that $f$ is constant if $\Omega$  is strictly positive on $H_0$. Since $f$ vanishes on the boundary of $D$, $f=0$.
\end{proof}
We next consider the case when $D=H_0$ is the entire space. 

\begin{thm}\label{2Wirtinger} 
 Let $v$ be a smooth 1-homogeneous function on $\R^{n+1}_+$, and let $\Omega$ be strongly homogeneous on $\R^{n+1}_+$ of bidegree $(n-1,n-1)$.  Assume that $v$ and $\Omega$ extend smoothly to the closed half space 
 $$
 \{(x_0,x); x_0\geq 0\},
 $$
 outside the origin, and that $v(0,x)$ vanishes. Then
$$
\int_{H_0} dv\wedge \dhash v\wedge\Omega\wedge \dhash x_0=
\int_{S_\mu^+} dv\wedge \dhash v\wedge \Omega\wedge \dhash\mu-
\int_{S_\mu^+}v^2\Omega\wedge\ddhash\mu\wedge\dhash\mu.
$$
In particular, if $\Omega$ is positive on $H_0$,
\be\label{2curious}
\int_{S_\mu^+}v^2\Omega\wedge\ddhash\mu\wedge\dhash\mu\leq
\int_{S_\mu^+}dv\wedge \dhash v\wedge \Omega\wedge \dhash\mu,
\ee
with equality only if $v=c x_0$, for some constant $c$,  if $\Omega$ is strictly positive on $H_0$.
\end{thm}
\begin{proof}
We follow the same proof as for the previous theorem, exhausting $H_0$ by  relatively compact domains $D_t:=tD$,  $t>0$, where $D$ is a fixed, relatively compact subdomain of $H_0$ that contains the origin. The only difference is that $v$ no longer vanishes on the boundary so there appear two extra terms 
 $$
 \int_{\partial D_t} v \dhash v\wedge \Omega\wedge \dhash x_0=: b_1(t),
 $$
 and
 $$
  \int_{ \partial S(D_t)} v \dhash v\wedge \Omega\wedge \dhash \mu=: b_2(t).
  $$
  We claim that both these terms vanish when $t\to\infty$. This is evident for $b_2$ since $v$ tends smoothly to zero when $x_0\to 0$. As for $b_1$, we use the scaling
 of $H_0$, $x\to F_t(x):=tx$ (or rather $F_t(x,\xi) =(tx, \xi)$ on the complexification of $H_0$). Then
 $$
 b_1(t)=\int_{\partial D} F_t^*(v\dhash v\wedge \Omega)\wedge \dhash x_0.
 $$
 Since $v$ is smooth and 1-homogeneous on $\{(x_0,x); x_0\geq 0\}$ (outside the origin) and vanishes for  $x_0=0$, we have, when $x\in \partial D$,  $|v(x_0, x)|\leq Cx_0$, and $|v(1, x/x_0)|\leq C$.  Hence $F_t^*(v)$ is bounded when $t$ is large. Moreover, for $j=1, ...n$, the partial derivatives $v_j$ are homogeneous of order zero and vanish for $x_0=0$. This implies that
 $$
 v_j(1, x/x_0)=v_j(x_0,x)= O(x_0),
 $$
 so $F_t^*(\dhash v)=O(1/t)$ when $t\to \infty$. Finally, since $\Omega$ is regular, $F_t^*(\Omega)$ stays bounded as $t\to\infty$. Altogether this shows that $b_1(t)$ goes to zero and completes the proof.

\end{proof}

The last part of Theorem \ref{Wirtinger} says that the quadratic form
$$
\int_{S(D)}dv\wedge \dhash v\wedge \Omega\wedge \dhash\mu
$$
with Dirichlet boundary conditions has all eigenvalues with respect to the $L^2$-norm
$$
\int_{S(D)}v^2\Omega\wedge\ddhash\mu\wedge\dhash\mu
$$
strictly greater than 1, if $D$ is relatively compact in $H_0$. Theorem \ref{2Wirtinger} shows that we pick up a new eigenvalue, $\lambda=1$, when we take $D=H_0$, and that this eigenvalue is simple with eigenfunction $v=x_0$. The final theorem of this section reformulates this in terms of an associated `Laplace operator'. This aspect will be developed more in section \ref{Asdiffoperator}.

\begin{thm}\label{Dirichlet} Let $v$ be a smooth 1-homogeneous function on $\R^{n+1}_+$ that solves the differential equation
$$
\ddhash v\wedge\Omega=0,
$$
where $\Omega$ is regular, strongly homogeneous  of bidegree $(n-1,n-1)$, and that $\Omega$ is strictly positive on $H_0$. Assume $v$ extends smoothly to the boundary outside the origin, and that $v(0,x)=0$. Then $v=cx_0$ for some constant $c$.
\end{thm}
\begin{proof} As in the proof of the previous theorem,
$$
\int_{H_0} dv\wedge\dhash v\wedge\Omega\wedge\dhash x_0=
\int_{H_0} v \ddhash v\wedge\Omega\wedge \dhash x_0=0.
$$
Hence $dv=0$, so $v=c$ is constant on $H_0$. By homogeneity, $v=cx_0$.
\end{proof}

\newpage
\section{Operations on strongly homogeneous forms}\label{operations}
It is clear that if $\Omega_1$ and $\Omega_2$ are strongly homogeneous then their wedge product
$$
\Omega=\Omega_1\wedge\Omega_2
$$
is also strongly homogeneous, simply because $d$ and $\delta^\#$ are antiderivations. This means that in particular, all forms (or currents) of type
$$
\ddhash\phi_1\wedge ...\ddhash\phi_p,
$$
where $\phi_j$ are 1-homogeneous are strongly homogeneous. In this section we shall list several other operations that also preserve strongly homogeneous forms. 

\subsection{Restriction.}

\begin{prop}\label{restriction} The restriction of a strongly homogeneous form to a linear subspace, $V$,  is strongly homogeneous on $V$.
\end{prop}
\begin{proof} It is enough to prove this when $V$ is a hyperplane, that we may take to be $\{x_1=0\}$. Clearly the restriction is $d$-closed, so all we need to prove is that it lies in the kernel of $(\delta')^\#$, contraction with $\sum_2^n x_j\partial_{\xi_j}$. Write
$$
\Omega=\Omega' + d\xi_1\wedge \Omega'',
$$
where $\Omega'$ and $\Omega''$ do not contain $d\xi_1$. Then,
$$
0=\delta^\#\Omega = (\delta')^\#\Omega' + x_1\Omega''+ d\xi_1\wedge\delta^\#\Omega''.
$$
Restricting to (the complexification of)  $\{x_1=0\}$, we see that $(\delta')^\#\Omega'=0$ as claimed.
\end{proof}
\subsection{Extension}

\begin{prop}\label{extension}
Let $V$ be a hyperplane in $\R^n$, and let $[V]_s$ be its associated supercurrent of integration. Let $\Omega$ be a weakly positive strongly homogeneous $(p,p)$-current on $V$. Then
$$
\tilde\Omega:=\Omega\w[V]_s
$$
is a weakly positive and strongly homogeneous current of bidegree $(p+1,p+1)$ in $\R^n$.
\end{prop}
\begin{proof}
We may assume $V=\{x_1=0\}$. Then
$$
[V]_s=[V]\w \dhash x_1=\ddhash\max(x_1,0),
$$
where $[V]$ its the current of integration on $V$ in $\R^n$.
The existence of $\tilde\Omega$ follows directly from Proposition \ref{hyper} (or Theorem \ref{crucial}). Clearly, $d\tilde\Omega=0$ and that
$$
\delta^\#\tilde\Omega=0
$$
follows from $\delta^\#[V]_s=-x_1 [V]=0$.
\end{proof}

\subsection{Push-forward.}

Let $A$ be a linear map from $\R^n$ to $\R^m$. We denote by the same letter $A$ its extension to a linear map from $\C^n$ to $\C^m$
$$
A(x+iy)=Ax+iAy.
$$
If $\Omega$ is a (super)current on $\R^n$ we define its push-forward under $A$ by
$$
\int_{\R^m} A_*(\Omega)\wedge \alpha=\int_{\R^n} \Omega\wedge A^*(\alpha),
$$
if $\alpha$ is a compactly supported (super)form on $\R^m$. 

The push-forward is in general not well defined, since it is not guaranteed that the integral in the right hand side is convergent. We will show that the push-forward {\it is} well defined when $\Omega$ is strongly homogeneous, and that the push-forward of $\Omega$ is then also strongly homogeneous. We start with  the following lemma. 
\begin{lma}\label{finite} Let $\Omega$ be a strongly homogeneous current on $\R^n$ of bidegree $(p,p)$. Let $W$ be an affine subspace of $\R^n$ of dimension $p$ and let $f$ be a function of at most linear growth. Then 
$$
\int_Wf\Omega
$$
is absolutely convergent.
\end{lma}
\begin{proof} Assume first that $W$ does not contain the origin, so that $W$ is not a linear subspace. Let $V$ be the smallest linear subspace containing $W$; its dimension equals the dimension of $W$ plus 1. Since, by the previous proposition, the restriction of $\Omega$ to $V$ is again strongly homogeneous, we can work on $V$ instead of $\R^n$.  In other words, we may assume that the codimension of $W$ equals 1. 

Choose orthogonal coordinates so that $V=\R^m$, and $W=\{x_1=a\}$, $a>0$. We need to estimate
$$
\int_{\{x_1=a\}} f\Omega\wedge\dhash x_1
$$
(cf. Theorem \ref{currentformula}).
We will use Proposition \ref{bdry} and we extend $f$ from $W$ to $V=\R^m$ as a 1-homogeneous function. The content of the proposition is that integrals of the form
$$
\int_{\mu=1, x_1>0} f\Omega\wedge\dhash\mu
$$
do not depend on the gauge function $\mu$. Take e.g. $\mu=|x|$. Then
\be\label{what}
\int_{\{x_1=a\}} f\Omega\wedge\dhash x_1/a=
\int_{S^{n-1}_+} f\Omega\wedge\dhash\mu.
\ee
If $f$ is of at most linear growth its 1-homogeneous extension is bounded on 
$\{|x|=1\}$, so the RHS is finite, since $\Omega\wedge\dhash\mu$ is a finite measure on $S^{n-1}$ by Proposition \ref{sumup}. To see that the integral in the left hand side is absolutely convergent we may replace $f$ by $|f| \text{sign}(\Omega\wedge\dhash x_1)$. 
Finally, when $0\in W$, $W$ is already a linear subspace and $\Omega$ is a strongly homogeneous form of maximal degree on $W$. $\Omega$ then defines a measure on $W$ which vanishes outside the origin, so it is a point mass at the origin and the integral is trivially convergent. 

\end{proof}

Notice that formula (\ref{what}) implies that 
$$
\int_{x_1=a} \Omega\wedge\dhash x_1
$$
does not depend on $a$ (take $f=x_1$ in (\ref{what})) when $a\neq 0$. The continuity result in Proposition \ref{sumup} shows that $a=0$ also gives the same value.
\begin{thm}
Let $\Omega$ be a strongly homogeneous form of bidegree $(p,p)$ on $\R^n$. Let $A$ be a surjective linear map from $\R^n$ to $\R^m$, where $p\geq n-m$. Then the push-forward $A_*(\Omega)$ is a well defined strongly homogeneous current on $\R^m$ of bidegree $(q,q)$, where $q=p+m-n$.

\end{thm}
\begin{proof}
It is enough to prove this for the map
$$
(x_1, ...x_n)\to (x_1, ...x_m)=:x'.
$$
To show that $A_*(\Omega)$ is well defined, it suffices to show  that the integral
$$
\int_{\R^n} \Omega\wedge  A^*(\alpha)
$$
is absolutely convergent, where $\alpha$ is a smooth compactly supported form of bidegree $(m-q, m-q)$. We consider first the case $p+m=n$, so that $\alpha$ is a form of bidegree $(m,m)$ and $A_*(\Omega)$ is a function. Then $A_*(\Omega)$ is the  integral of $\Omega$ over the fibres $x'=a$, and it follows from Lemma \ref{finite} that  the integral is absolutely convergent. 

When $p+m>n$ we use that $\Omega$ is homogeneous so that its coefficients are homogeneous of order $-p$. Since $p$ is larger than the dimension of the fibres of $A$ the integral is absolutely convergent. Hence, in both cases, the push-forward is well defined. 

To show that $\delta^\# A_*(\Omega)=0$ we use the following fact:

\noindent {\it If $\alpha$ and $\gamma$ are superforms on $\R^N$ whose degrees in $d\xi$ sum up to $N+1$, then 
$$
\delta^\#(\alpha)\wedge\gamma=\pm \alpha\wedge\delta^\#\gamma.
$$}
(This follows from $0=\delta^\#(\alpha\wedge\gamma)$.)
To prove that $\delta^\# A_*(\Omega)=0$ it is therefore enough to
prove that
$$
\int_{\R^m} A_*(\Omega)\wedge\delta^\#\alpha=\int_{\R^n} \Omega\wedge \delta^\# A^*(\alpha)=0,
$$
when $\alpha$ is of bidegree $(m-q,m-q+1)$.
But this follows from the same fact since $\delta^\#\Omega=0$. 

All that remains is to prove that $dA_*(\Omega)=0$, or equivalently
$$
0=\int_{\R^m} A_*(\Omega)\wedge d\alpha=\int_{\R^n} \Omega\wedge d\alpha,
$$
when $\alpha$ is of bidegree $(m-q-1,m-q)$. But since the integral in the right hand side is absolutely convergent it equals the limit as $R\to\infty$ of 
$$
\int_{\R^n}\chi_R \Omega\wedge d\alpha,
$$
where $\chi_R(x)=\chi(|x|/R)$, and $\chi$  is a suitable cut-off function. That this limit is zero follows directly from Stokes' theorem since $\Omega$ is $d$-closed.

\end{proof}

\subsection{Convolution}

We are now ready to define the convolution of two strongly homogeneous forms 
$\Omega_1$ and $\Omega_2$ of bidegree $(p_1,p_1)$ and $(p_2,p_2)$, respectively. 

\begin{df}\label{convolution} Let 
$\Omega_1$ and $\Omega_2$ be strongly homogeneous on $\R^n$, of bidegree $(p_1,p_1)$ and $(p_2,p_2)$, respectively. Assume $p_1+p_2\geq n$. Then the convolution of $\Omega_1$ and $\Omega_2$ is defined as 
$$
\Omega_1\star\Omega_2:= \pi_*(\Omega_1(x)\wedge\Omega_2(y)),
$$
where $\pi:\R^n_x\times \R^n_y\to \R^n$ is the linear map $\pi(x,y)=x+y$. 
\end{df}
By the previous subsection $\Omega_1\star\Omega_2$ is a strongly homogeneous current of bidegree $(q,q)$, $q=p_1+p_2-n$. 
\subsection{Fourier transform.}
The Fourier transform of differential forms on $\R^n$ was introduced by Scarfiello, \cite{Scarfiello}, see also \cite{Schwartz}. The definition is that, if 
$f=\sum_{|I|=p} f_I dx_I$ is a form of degree $p$,
$$
\hat f=*\sum_I \hat f_I dx_I,
$$
where $\hat f_I$ is the classical Fourier transform of the function $f_I$, and $*$ is the Hodge $*$-operator. Thus, Scarfiello's Fourier transform maps $p$-forms to forms of degree $n-p$. (In particular, Scarfiello's Fourier transform maps functions to $n$-forms and conversely, so we are abusing language by using the same letter for the two Fourier transforms.)

Scarfiello's transform was rewritten more elegantly by Andersson, \cite{Andersson}. To describe this, let
$$
\beta(x,y):= \sum_i dx_i\wedge dy_i,
$$
be the standard  symplectic form on $\R^n_x\times \R^n_y$, viewed as the cotangent bundle of $\R^n_x$. Then Andersson's formula is
\be\label{Andersson}
\hat f(y):=c_{n,p}\int_{\R^n} f(x)\wedge e^{-i x\cdot y+ \beta(x,y)},
\ee
where $c_{n,p}$ equals plus or minus 1. To see that the two definitions coincide one uses the formula
$$
e^{\beta(x,y)}\wedge dx_I= c_{n,p}dx_1\wedge ...dx_n\wedge *dx_I,
$$
where $*dx_I$ is written in the coordinates $y$, so that e.g. $*dx_1\wedge ...dx_p=dy_{p+1}\wedge ...dy_n$. 
The exponential in (\ref{Andersson}) should be read as 
$$
e^{-i x\cdot y+ \beta(x,y)}= e^{-ix\cdot y} \sum_0^n\beta^k/k!,
$$
with the understanding that only the term $k=n-p$ gives a form of full degree in $dx$, so the other terms do not contribute to the integral, and the integral is a form of degree $n-p$ in $dy$.  Alternatively, and perhaps more correctly, (\ref{Andersson}) can be read as the push-forward of $f(x)\wedge e^{-i x\cdot y+ \beta(x,y)}$ on 
$\R^n_x\times \R^n_y$  under the map $(x,y)\to y$.  Andersson's version of the definition has the advantage that it uses no extra structure on $\R^n$, like Lebesgue measure or scalar product, but only the canonical symplectic form on a cotangent bundle. 

We copy almost verbatim (\ref{Andersson}) to the setting of superforms, and put, for a superform $\Omega=\Omega(x,\xi)$
\be\label{superandersson}
\hat \Omega(y,\eta)=C_{n,p}\int \Omega(x,\xi)\wedge 
e^{-i x\cdot y+ \beta(x,y)+ \beta(\xi, \eta)},
\ee
where again the constant is either plus or minus 1; the sign chosen so that
\be\label{concretedefinition}
\hat\Omega=\sum \widehat{\Omega_{I J}}*dx_I\wedge *d\xi_J.
\ee
The integral is here a superintegral in $(x,\xi)$, and the integral returns a superform in $(y,\eta)$ (of bidegree $(n-p,n-q)$ if $\Omega$ is of bidegree $(p,q)$. The integral is of course  not always convergent, but has a meaning as soon as the coefficients of $\Omega$ are tempered distributions in the sense of Schwartz. This is the case when $\Omega$ is smooth and strongly homogeneous. The next proposition shows that the Fourier transform of a convolution   behaves like the classical Fourier transform on functions. 
\begin{prop} If $\Omega_1$ and $\Omega_2$ are  smooth and  strongly homogeneous, 
$$
\widehat{\Omega_1\star\Omega_2} =\hat\Omega_1\wedge\hat\Omega_2.
$$
\end{prop}
\begin{proof} Formally, this is a direct consequence of the definitions. To see this, let for fixed $y$ and $\eta$, 
$$
\Phi(x,\xi)=e^{i x\cdot y+ \beta(x,y)+ \beta(\xi, \eta)}.
$$
Then, if $\pi(u,\mu, v, \nu):= (x,\xi)=(u+v, \mu+\nu)$, $\pi^*\Phi(x,\xi)=\Phi(u+v,\mu+\nu))=\Phi(u,\mu)\wedge \Phi(v,\nu)$, from which the claim follows. 

We can, however, not apply this argument directly since the coefficients of $\Omega_i$ are not integrable if $\Omega_i$ are strongly homogeneous. We therefore replace $\Omega_1(x)$ by $\Omega_1^j:=\chi_j(x)\Omega_1(x)$ and $\Omega_2(y)$ by $\Omega_2^j:=\chi_j(y)\Omega_2(y)$., where $\chi_j$ is a sequence of cut-off functions tending to 1. Then the previous, `formal', argument applies and gives that
$$
\widehat{\Omega_1^j\star\Omega_2^j}=\widehat{\Omega_1^j}\wedge\widehat{
\Omega_2^j}.
$$
We may then let $j\to \infty$ to conclude, using Schwartz' theory of the Fourier transform of temperate distributions.

\end{proof}
\begin{thm} If $\Omega$ is a strongly homogeneous current of bidegree $(p,p)$, $\hat\Omega$ is strongly homogeneous  of bidegree $(n-p,n-p)$
\end{thm}
\begin{proof}
We use the standard formulas
$$
i\partial\hat f/\partial y_j=\widehat{x_j f}, \quad iy_j\hat f(y)=\widehat{\partial f/\partial x_j},
$$
valid when $f$ is a temperate distribution on $\R^n$, together with 
$$
*(dx_j\wedge dx_I)= (-1)^p  \partial x_{j}\rfloor*dx_I,
$$
and
$$
*(\partial x_j\rfloor dx_I)=(-1)^{p-1} dx_j\wedge *dx_I,
$$
if $|I|=p$.
Combining these formulas with the definition (\ref{concretedefinition})  we get
\be\label{d-delta}
d\hat\Omega=A_{n,p,q}\widehat{\delta\Omega}, \quad \delta\hat\Omega=
B_{n,p,q}\widehat{d\Omega},
\ee
if $\Omega$ is of bidegree $(p,q)$. Now let $\Omega$ be strongly homogeneous of bidegree $(p,p)$. By Proposition \ref{symmetric}, $\Omega$ is symmetric, so $\delta^\#\Omega=0$ implies that $\delta \Omega=0$. By (\ref{d-delta}),  $d\hat\Omega=0$. In the same way, $d\Omega=0$ implies that $\delta\hat\Omega=0$, which gives $\delta^\#\hat\Omega=0$, since $\hat\Omega$ is symmetric too. This completes the proof.

\end{proof}

\newpage

\section{Alexandrov's differential operator.}\label{Asdiffoperator}

Alexandrov's second proof of the Alexandrov-Fenchel inequalities (see \cite{2Alexandrov})  is a generalization of Hilbert's proof of the Brunn-Minkowski inequality in three dimensions, \cite{Hilbert}. It is based on a study of the eigenvalues of a certain elliptic second order differential operator on the sphere, $u\to A(u)$, that we shall now discuss, using superforms.

Our first incarnation of Alexandrov's operator maps 1-homogeneous functions on $\R^n$ to (strongly homogeneous) $(n-1,n-1)$-currents. In the sequel we will sometimes write that a 1-homogeneous function is strictly convex. This is of course never literally true since the function is linear on any line through the origin. What we mean is that the function is strictly convex along any other line, which in case of smooth functions means that the Hessian has $(n-1)$ positive eigenvalues. 

Let $\phi_3, ...\phi_n$ be smooth strictly convex 1-homogeneous functions and put
$$
\Omega_{n-2}=\ddhash\phi_3\w ...\ddhash\phi_n;
$$
it is a strongly homogeneous and positive $(n-2,n-2)$-form ( if $n=2$ we put $\Omega=1$). We could also define it for non-smooth $\phi_i$ but for the proof of the Alexandrov-Fenchel inequalities, it is enough to consider smooth functions. (The question of when equality holds in the inequalities is a different matter; there it might be useful to consider general convex functions (cf. \cite{Shenfeld-vanHandel} and the references there for recent results on this problem).) 

We now define, for $v$ a smooth 1-homogeneous function
\be\label{incarnation1}
\A(v):=\ddhash v\w\Omega_{n-2}.
\ee
We thus get one operator for each choice of  $\phi_3, ...\phi_n$ and we will fix one such choice from now.

\begin{remark}\label{scalar} Notice that the forms $\A(v)$ are annihilated be $\delta$ and $\delta^\#$. Any such form of bidegree $(n-1,n-1)$ is a (positive)  multiple of one fixed form
$$
\delta^\#\delta dV
$$
(recall that $dV=dx_1\w d\xi_1\w ...dx_n\w d\xi_n$). Indeed, $d|x|\w \dhash|x|\w\A(v)=\lambda(v) dV/|x|^2$, and applying $\delta^\#\delta $ we get
$$
\A(v)=\lambda(v) \delta^\#\delta  dV/|x|^2.
$$
Hence we could identify the operator $\A$ with the scalar-valued operator $v\to \lambda(v)$, but it will be more convenient to work with $\A$.
\end{remark}

 We next define a bilinear form on the space of smooth 1-homogeneous functions through the pairing described in Section \ref{sectionhomform},  by
\be\label{pairing}
Q(u,v):=\langle u, \A(v)\rangle =\int_{\mu=1} u \A(v)\w\dhash\mu,
\ee
where $\mu$ is any gauge function, cf. Proposition \ref{independent}. The standard choice, corresponding to the original setting of Alexandrov,  is $\mu(x)=|x|$, but it will be  useful to allow general gauges. Here $u$ and $v$ are 1-homogeneous functions on $\R^n$, but we could also consider them as functions on $S_\mu=\{\mu=1\}$ since any function there has a unique homogeneous extension. 

We now proceed to relate our definition of the Alexandrov operator $\A(v)$ to the more standard one as a differential operator on the `sphere' $S_\mu$.
The differential form $\A(v)\wedge\dhash\mu$ is of degree $(n-1)$ in $dx$ and of full degree $n$ in $d\xi$. If $v$ is smooth as a 1-homogeneous function on $\R^n$ it defines a measure on $S_\mu$ which is absolutely continuous with respect to surface measure $dS$. Note that $S_\mu$ does not need to be smooth, so there is no a priori definition of being smooth `on $S_\mu$'. As an example, the function 1 on $S_\mu$ is here interpreted as $\mu$, so it is not smooth in our sense , if $\mu$ is not smooth.

Let $dm$ be any choice of measure on $S_\mu$ satsifying 
$$
cdS\leq dm\leq C dS.
$$
We will call such measures non-degenerate.
Then we can write
\be\label{incarnation}
\A(v)\wedge\dhash\mu|_{\mu=1}= A_m(v)dm,
\ee
where $A_m(v)$ is a function on $S_\mu$. This is our second incarnation of the Alexander operator, and we have
$$
Q(u,v)=\int_{\mu=1} uA_m(v) dm.
$$
When $\mu=|x|$ and $dm$ is surface measure on the sphere $A_m$ is a rewrite of Alexandrov's original operator and we will denote it simply by $A$. The next proposition says that any $A_m$ uniquely determines $\A$, so we can view all these operators as different incarnations of the same object.
\begin{prop}\label{sameobject}
If $u$ is 1-homogeneous and $A_m(u)=fdm$ for some choice of gauge $\mu$ and non-degenerate  reference measure $dm$ on $S_\mu$, then $\A(u)=\Omega^0$, where $\Omega^0$ is the uniquely determined strongly homogeneous $(n-1,n-1)$-current such that $\Omega^0\wedge d^\#\mu= fdm$ on $S_mu$ (cf. Theorem   \ref{secondpart} ). In particular, $A_m(u)=0$ if and only if $\A(u)=0$. 
\end{prop}
\begin{proof} That $A_m(u)=f$ means that 
$$
\A(u)\w\dhash\mu=fdm=\Omega^0\w \dhash\mu
$$
on $S_\mu$.  Contracting with $\delta^\#$ we get (since $\delta^\#\A(u)=0$) that $\A(u)=\Omega^0$ when $\mu=1$, and therefore everywhere by homogeneity. 
\end{proof}
Notice that
\be\label{QisV}
Q(u,v)=\int_{\mu=1} u \A(v)\w\dhash\mu=V(u,v, \phi_3, ...\phi_n)n!,
\ee
so it is symmetric in $u$ and $v$.

We are now ready to state a  first version of the Alexandrov-Fenchel theorem:
\begin{thm}\label{Lorentzian} The quadratic form $Q(u,u)$ has Lorentzian signature in the sense that it is strictly positive somewhere and negative semidefinite on some subspace of $C^2$ of codimension 1. More precisely, if $Q(\phi,\phi)>0$, $Q$ is negative semidefinite on the space of all functions $u$ such that $Q(u,\phi)=0$. 
\end{thm}
The first part of the statement is of course evident since $Q(\phi,\phi)>0$ is $\phi$ if strictly convex. The next section will be devoted to two different proofs of the last part. Another, more common, formulation of the Alexandrov-Fenchel theorem is given in the next corollary.

\begin{cor}\label{AF} If $\phi$ and $\psi$ are convex, then
$$
Q(\phi,\psi)^2\geq Q(\phi,\phi)Q(\psi,\psi).
$$
In other words
$$
V(\phi,\psi, \phi_3, ...\phi_n)^2\geq V(\phi,\phi, \phi_3, ...\phi_n) V(\psi,\psi, \phi_3, ...\phi_n).
$$
\end{cor}
This follows from the theorem by the Cauchy inequality for time-like vectors in Lorentz space. We include the statement and its proof for completeness, and also to check the precise assumptions needed.
\begin{thm}\label{Cauchy} Let $Q(u,u)$ be a quadratic form on a vector space such that for some vector $u_0$, $Q(u_0,u_0)>0$ and $Q$ is negative semidefinite on some subspace of codimension 1. Then, if $Q(v,v)\geq 0$
\be
Q(u,v)^2\geq Q(u,u)Q(v,v).
\ee
Moreover, if $Q(v,v)>0$,  $Q$ is negative semidefinite on the subspace $\{u; Q(u,v)=0\}$.
\end{thm} 
\begin{proof} Consider the polynomial in $t$
$$
p(t)=Q(u+tv,u+tv).
$$
Assume $Q(v,v)$ is positive (if $Q(v,v)=0$ there is nothing to prove). Then $p$ is a second degree polynomial with positive leading coefficient. We may  assume  that $u$ and $v$ are linearily independent since the inequality is trivial if they are proportional. Then the minimum of $p$ must be less than or equal to zero, since if it were positive, $Q$ would be positive definite on a subspace of dimension 2. (This subspace would have to intersect the codimension 1 subspace where $Q$ is negative semidefinite, leading to a contradiction.) Computing derivatives we see that the minimum is attained where 
$$
t=-Q(u,v)/Q(v,v),
$$
so the minimum value is $Q(u,u)-Q(u,v)^2/Q(v,v)$. Hence the Cauchy inequality follows. This in turn implies  that if $Q(u,v)=0$, then $Q(u,u)\leq 0$, which proves the last claim. 
\end{proof}

The corollary follows since $Q(\phi,\phi)\geq 0$ if $\phi$ is convex; the last part follows from (\ref{QisV}). Note also that $Q(\phi,\psi)\geq 0$ if both $\phi$ and $\psi$ are convex so we also have that
$$
Q(\phi,\psi)\geq (Q(\phi,\phi)Q(\psi,\psi))^{1/2}.
$$

We next record an alternative description of the bilinear form $Q$, that gives another interpretation of the Alexandrov-Fenchel theorem. 
\begin{prop}\label{Poincareineq} If $u$ is a  sufficiently regular 1-homogeneous function, and $\mu$ is any gauge function, then
$$
Q(u,u)=\int_{\mu=1}u^2\dhash\mu\wedge\ddhash\mu\wedge\Omega_{n-2}-\int_{\mu=1} du\wedge\dhash u\wedge \dhash\mu\wedge\Omega_{n-2}.
$$
\end{prop}
\begin{proof} 
By Stokes' theorem on the closed manifold $\mu=1$ we have
$$
 \int_{\mu=1} du\wedge\dhash u\wedge \dhash\mu\wedge\Omega_{n-2}=-\int_{\mu=1}u\ddhash u\wedge\Omega_{n-2}\wedge\dhash \mu+\int_{\mu=1}u\dhash u\wedge\ddhash\mu\wedge\Omega_{n-2}.
 $$
 The first term on the right hand side is $-Q(u,u)$. By Lemma \ref{fundamentallemma}, the second equals
 $$
 \int_{\mu=1}u^2\dhash\mu\wedge\ddhash\mu\wedge\Omega_{n-2},
 $$
 and the proposition follows.
 \end{proof}
 
In the light of the proposition we see that Theorem \ref{Lorentzian} can be seen as a Poincar\'e inequality:
\be\label{one}
\int_{\mu=1}u^2\dhash\mu\wedge\ddhash\mu\wedge\Omega_{n-2}\leq \int_{\mu=1} du\wedge\dhash u\wedge \dhash\mu\wedge\Omega_{n-2},
\ee
if $u$ satisfies any orthogonality condition
$$
Q(u,\phi)=0,
$$
with $Q(\phi,\phi)>0$. 

The first term on the right hand side here is the squared $L^2$-norm of the function $u$ with respect to a certain measure on $\{\mu=1\}$ and the second term  is an $L^2$-norm of the gradient of $u$. Explicitly, the orthogonality condition means that
\be\label{two}
Q(u,\phi)=\int_{\mu=1}u\ddhash\phi\wedge\Omega_{n-2}\wedge\dhash\mu=0.
\ee
We want to stress  the interesting fact that all such Poincar\'e inequalities -- corresponding to different choices of $\mu$ and $\phi$ -- are equivalent. In other words, if (\ref{two}) implies (\ref{one}) for a certain pair $\mu$ and $\phi$, the quadratic form $Q$ (whch does not depend on $\mu$ or $\phi$) has Lorentz signature, so the same thing must hold for any other choice of $\mu$ and $\phi$.

Now assume that $\mu$ is smooth.  Clearly $A_m$ is a second order operator and it follows from Proposition \ref{Poincareineq} that it is elliptic, since the leading order symbol can be read off from the integral
$$
\int_{\mu=1} du\wedge\dhash u\wedge \dhash\mu\wedge\Omega_{n-2}.
$$
We also have that
$$
\int u A_m(v) dm=\langle u,\A(v)\rangle =\int v A_m(u) dm,
$$
so $A_m$ is a symmetric operator. It follows that $A_m$ extends to a closed and densely defined  self adjoint operator on $L^2(dm)$ and we get the orthogonal decomposition
\be\label{decomposition}
L^2(dm)= R(A)\oplus N(A),
\ee
where $R(A)$ is the range of $A_m$ and $N(A)$ is its null space (see e.g. \cite{Warner}, chapter 6). The centerpiece of the proof of the Alexandrov-Fenchel theorem is the fact that 
$$
N(A)=\{u=x\cdot a; a\in \R^n\}.,
$$
which thus implies that the equation $A_m(u)=f$ can be solved on the sphere $S_\mu$  if and only if
$$
\int_{\mu=1} x_i fdm=0, \,\, i=1,...n.
$$

We will also have use for the corresponding results for the Dirichlet problem for $A_m$ on  the half-spheres
$$
S_\mu^+=\{x\in S_\mu; x_1>0\}.
$$
We first claim that $A_m$ is formally self adjoint on the space of smooth functions on the closure of $S_\mu^+$ that vanish on the boundary. This follows from
$$
\int_{S_\mu^+} u A_m(v)=\int_{S_\mu^+} u \ddhash v\w\Omega \w\dhash\mu=-\int_{S_\mu^+} du\w\dhash v\w\Omega\w\dhash\mu+\int_{S_\mu^+} u\dhash v\w\Omega\w\ddhash\mu.
$$
The first term on the right hand side is symmetric since $\Omega$ is symmetric. The second term equals, by Lemma \ref{fundamentallemma}, 
$$
\int uv\dhash \mu\w\ddhash\mu\w\Omega,
$$
so it is also symmetric. 

It follows from the theory of second order elliptic equations (see e.g. \cite{Evans}, Chapter 6, Theorem 4) that the Dirichlet problem
$$
A_mu=f, \quad u(x)=0\,\, \text{on}\,\, \partial S_\mu^+
$$
can be solved if and only if $f$ is orthogonal to the null space of $A_m$, $N(A)=\{u; A_mu=0, u(x)=0\,\, \text{on}\,\, \partial S_\mu^+\}$.
But,  by Theorem \ref{Dirichlet}, $N(A)$ is spanned by the function $u=x_1$. 
It follows that the Dirichlet problem on $S_\mu^+$ can be solved precisely when the right hand side satisfies
$$
\int_{S_\mu^+} x_1f dm=0.
$$


\newpage

\section{Alexandrov's (second) proof  of the Alexandrov-Fenchel inequalities.}

We start by stating Alexandrov's linear algebra lemma, which can be seen as the pointwise counterpart of the Alexandrov-Fenchel theorem.
\begin{thm}\label{Aslemma} Let $\omega_2, ...\omega_n$ be positive $(1,1)$ (super)forms on $\R^n$.
$$
\omega_k=\sum a^k_{i j}dx_i\wedge d\xi_j,
$$
where the matrices $M^k=(a^k_{i j})$ are symmetric, positive definite and constant. Let 
$$
B=\sum b_{i j}dx_i\wedge d\xi_j
$$ 
be an arbitrary $(1,1)$ (super)form with symmetric constant coefficients. Assume that
\be\label{trace2}
B\wedge\Omega_{n-1}=0,
\ee
where $\Omega_{n-1}=\omega_2\wedge ...\omega_n$.
Then 
\be\label{determinant}
B\wedge B\wedge\Omega_{n-2}\leq 0
\ee
where $\Omega_{n-2}=\omega_3\wedge ...\omega_n$. Equality holds only when $B=0$.
\end{thm}
Before we give the proof we look at the case $n=2$ for motivation. We then only have one positive form, $\omega_2$, and we may assume it equals $\sum dx_i\wedge d\xi_i$ after a linear change of coordinates. The assumption (\ref{trace2}) says that the trace of the matrix $(b_{i j})$ is zero, so its eigenvalues are $\lambda$ and $-\lambda$. The conclusion is that its determinant is negative, which is clear since it equals $-\lambda^2$. Moreover, the determinant can only be zero if $\lambda$ and hence $B$ is zero.

\begin{proof} We argue by induction, assuming $n>2$ and that the theorem holds in dimension $n-1$. The first, and most important, step is to prove that the map
\be\label{map}
B\to B\wedge\Omega_{n-2}
\ee
is injective. Assume $B\wedge\Omega_{n-2}=0$. Then 
$$
B\wedge\Omega_{n-2}\wedge dx_1\wedge d\xi_1=0.
$$
This means that the restriction of all the forms $B$, $\omega_3, ...\omega_n$ to the complex hyperplane $x_1=0=\xi_1$ satisfy the assumption (\ref{trace2}) in $n-1$ variables. By the inductive assumption
$$
B\wedge B\wedge\omega_4\wedge ...\omega_n\leq 0
$$
on this hyperplane, or in other words 
$$
B\wedge B\wedge\omega_4\wedge ...\omega_n\wedge dx_1\wedge d\xi_1\leq 0.
$$
(If $n=3$ the wedge product $\omega_4\wedge ...\omega_n$ should be interpreted as 1.)
We may assume that $\omega_3=\sum  dx_i\wedge d\xi_i$ and then argue in the same way for the other coordinates.
Summing up we get from the inductive assumption that
$$
B\wedge B\wedge \Omega_{n-2}\leq 0,
$$
 but in fact equality must hold since $B\wedge\Omega_{n-2}=0$. Hence
 $$
 B\wedge B\wedge\omega_3\wedge ...\omega_n\wedge dx_i\wedge d\xi_i= 0,
$$
for all $i$, since the sum of these forms is zero and each term is non positive. Hence, the inductive assumption implies that the restriction of $B$ to any complex coordinate hyperplane vanishes. But then $B$ must be zero, so the map in (\ref{map}) is indeed injective. 

The rest of the proof is a deformation argument. Just like in the previous section the theorem says that
the quadratic form
$$
Q(B,B)=B\wedge B\wedge\Omega_{n-2}
$$
has Lorentz signature and what we have just proved says that zero is not an eigenvalue. It is easy to check that the theorem does hold when $\omega_2= ...\omega_n=\sum dx_i\wedge d\xi_i$. Any other choice of $\omega_i$:s can be continuously deformed to this case and the eigenvalues change continuously under the deformation. Since zero is never an eigenvalue, we must always have one positive eigenvalue and the others strictly negative. Hence $Q$ has Lorentz signature, which implies the conclusion of the theorem.
\end{proof}

We are now ready to describe Alexandrov's proof of his theorem. We assume that $\phi_2, ...\phi_n$ are 1-homogeneous and positive outside the origin. In other words, they are support functions of convex bodies $K_2, ... K_n$ that contain the origin as an interior point, and we  also assume that these bodies are smoothly bounded  and strictly convex. This means that $\phi_j$ are smooth and their Hessians are strictly positive, apart from the zero eigenvalue in the radial direction. Here we also take the gauge function $\mu$ to be smooth. Notice that this implies that the `sphere' $S_\mu$ is smooth, since the gradient of $\mu$ cannot vanish by Euler's formulae (see the Appendix, Lemma \ref{smoothgauge}).  The following simple consequence of Alexandrov's lemma will be crucial in the proof. As before we write
$$
\Omega_{n-k}=\ddhash\phi_{k+1}\w...\ddhash\phi_n.
$$
\begin{cor}\label{2Aslemma}
Let $v$ be a smooth 1-homogeneous function such that
$$
\ddhash v\w\Omega_{n-2}=0.
$$
Then 
$$
(\ddhash v)^2\w\Omega_{n-3}\leq 0
$$
as an $(n-1,n-1)$-form on $\R^\setminus\{0\}$, i.e. for any $\alpha$ of bidegree $(1,0)$ 
\be\label{eq}
(\ddhash v)^2\w\Omega_{n-3}\w\alpha\w\alpha^\#\leq 0.
\ee
If equality holds at some point $x$ and some $\alpha$ such that $\sum a_ix_i\neq 0$, then $\ddhash v=0$ at $x$. 

\end{cor}
\begin{proof}
Let $\alpha=\sum a_i dx_i$, and let 
$$
V_\alpha=\{\sum a_i x_i=1\}.
$$
Then the forms $\ddhash\phi_i$ are strictly positive on $V_\alpha$, so Alexandrov's lemma implies that $(\ddhash v)^2\w\Omega_{n-3}\leq 0$ there. This means that
\be\label{phew}
(\ddhash v)^2\w\Omega_{n-3}\w\alpha\w\alpha^\#\leq 0
\ee
on $V_\alpha$.  By remark \ref{scalar}, $(\ddhash v)^2\w\Omega_{n-3}$ is a multiple of the fixed reference form $\delta^\#\delta dV$. Inequality (\ref{phew}) says that this multiple must be negative, which completes the proof of the first part. For the statement about equality we note first that by Alexandrov's lemma, 
$\ddhash v\w\alpha\w\alpha^\#=0$ if equality holds in (\ref{eq}). Applying $\delta$ and $\delta^\#$ we find that $\ddhash v=0$.

\end{proof}

The first, and main, step in Alexandrov's proof is the following result.
\begin{thm}\label{Alexandrov}{\it (Alexandrov)}Let $\mu$ be a gauge function and let $dm$ be a non-degenerate reference measure on $S_\mu$.  The only (1-homogeneous) functions $u$ such that $\A(u)=0$ , or equivalently 
$$
A_{m}(u)=0
$$
on $S_\mu$, 
 are the linear functions. Consequently, if $f\in L^2(dm)$, the equation 
 $$
 A_m(u)=f
 $$
 on $S_\mu$ can be solved if and only if  $f$ is orthogonal to all linear functions. More generally, if $d\nu$ is a measure on the sphere, the equation 
 $$
 A(u)=d\nu
 $$
 is solvable if and only if $d\nu$ has barycenter zero.
\end{thm}
By Theorem \ref{secondpart} and Proposition \ref{sameobject},  Theorem \ref{Alexandrov}  can be reformulated in terms of the operator $\A$;
$$
\A(u)=\ddhash u\w \Omega_{n-2}.
$$
\begin{cor}\label{cor1} Let $\Omega^0$ be a strongly homogeneous current of bidegree $(n-1,n-1)$. Then there is a 1-homogeneous function $u$ such that 
$$
\ddhash u\w\Omega_{n-2}=\Omega^0
$$
if and only if 
$$
\langle x_j,\Omega^0\rangle =0, j=1, ...n,
$$
or, equivalently, the trivial extension of $\Omega^0$ across the origin is closed.
\end{cor}

\begin{proof}(of Theorem \ref{Alexandrov}) We first note that by Proposition \ref{sameobject}, the equation $A_m(u)=0$ is equivalent to saying that the 1-homogeneous extension of $u$ satisfies 
\be\label{A(u)=0}
\ddhash u\wedge\Omega_{n-2}=0.
\ee
By our second version of Alexandrov's lemma (Corollary \ref{2Aslemma}) this gives 

\be\label{negative}
(\ddhash u)^2\wedge\Omega_{n-3}\wedge\dhash\mu\leq 0
\ee
as a measure on the sphere (cf. Theorem \ref{currentformula} in the appendix).

We now integrate (\ref{negative}) over $S_\mu$ and apply Stokes' theorem. This gives 
\be\label{integrated}
 \int_{S_\mu} (\ddhash u)^2\wedge \Omega_{n-3}\w\dhash\mu= 
 -\int_{S_\mu}\dhash u\w\ddhash u\w\Omega_{n-3}\w\ddhash\mu\leq 0.
 \ee
 Here we are free to choose $\mu$, since (\ref{A(u)=0}) does not depend on $\mu$. Taking $\mu=\phi_3$, we get
 \be
  \int_{S_\mu}\dhash u\w\ddhash u\w\Omega_{n-3}\w\ddhash\mu=\int_{S_\mu}\dhash u\w\ddhash u\w\Omega_{n-2},
  \ee
  which vanishes by (\ref{A(u)=0}). Hence, equality holds in (\ref{integrated}), and therefore in (\ref{negative}). By the condition for equality in Corollary \ref{2Aslemma}  $\ddhash u=0$, so $u$ is linear. 

The last claim on solvability of the inhomogeneous equation $A_m(u)=f$ follows from (\ref{decomposition}). For the very last part, one can approximate $d\nu$ in the weak* topology by measures $f_j dS$ where $f_j$ are smooth and orthogonal to $N(A)$. Then solve $A(u_j)=f_j$ with $u_j \in N(A)^\perp$. If $f$ is any continuous  function in $N(A)^\perp$ we solve $A(u)=f$ and get, since $A$ is formally selfadjoint,  that
$$
\int_{|x|=1} u_j f dS= \int_{|x|=1} u_j A(u) dS=\int_{|x|=1} f_j u dS.
$$
Hence $u_j$ converge weak* to  a weak solution of $A(u)=d\nu$.

\end{proof}

The remaining part of the proof is, just like in the proof of Alexandrov's lemma, a deformation argument, that, again, we only sketch. The theorem we have just proved says that the only eigenfunctions of the elliptic operator $A_m$ with eigenvalue zero are the linear functions. One verifies that when $\phi_3=...\phi_n=|x|$ the operator $A_m$ has only one positive eigenvalue. The case of general $\phi_j$ can be deformed continuously to this case, and the eigenvalues change in a continuous way. Finally, Theorem \ref{Alexandrov} says that the negative eigenvalues cannot pass the barrier $\lambda=0$, so $A_m$ has only one positive eigenvalue in the general case as well. 


\newpage

\section{A real variable variant of Gromov's proof of the Alexandrov-Fenchel theorem.}\label{Gromov}
In \cite{Gromov}, Gromov gave an alternative proof of the Alexandrov-Fenchel theorem, by giving a simplified proof of the Khovanskii-Teissier theorem (\cite{Khovanskii}, \cite{Teissier}), and applying that result to toric varieties. He then sketched how the Khovanski-Teissier theorem for toric varieties implies the Alexandrov-Fenchel theorem. Here we will avoid the use of toric varieties, and complex manifolds in general, and instead give a direct proof of the  Alexandrov-Fenchel theorem, using a real variable version of Gromov's main idea.

In this proof we assume again that the convex bodies $K_j$ are strictly convex and smoothly bounded, and  write 
$$
V(K_1, ...K_n)=V(\phi_1, ...\phi_n)=\int_{\R^n} \ddhash\phi_1\w...\ddhash\phi_n/n!,
$$
where $\phi_j$ are smooth, strictly convex functions of linear growth, such that their indicator functions 
$$
\phi_j^\circ=h_{K_j},
$$
equal the support functions of the  convex bodies (cf. Theorem \ref{MAestimate}).
The next proposition shows that we may take $\phi_j$ regular.
\begin{prop}
For any smoothly bounded and strictly convex body, $K$, there is a smooth strictly convex function, $\phi$, with $\phi^\circ=h_K$, such that $\phi$'s 1-homogeneous extension to $\R^{n+1}_+$, $\Phi(x_0, x)=x_0\phi(x/x_0)$, extends smoothly to the closure of $\R^{n+1}_+$ outside the origin.
\end{prop}
\begin{proof}
Let $\tilde K$ be a smoothly bounded and strictly convex body in $\R^{n+1}$, such that its projection $\pi(\tilde K)=K$, where $\pi(y_0,y)=y$. Then
$$
h_{\tilde K}(0, x)=\sup_{(y_0,y)\in\tilde K}  y_0\cdot 0+y\cdot x=h_{\pi(\tilde K)}(x)= h_K(x),
$$
and $h_{\tilde K}$ is smooth, so we can take $\phi(x)=h_{\tilde K}(1,x)$. Then the 1-homogeneous extension of $\phi$ is $h_{\tilde K}$, which is smooth on all of 
$\R^{n+1}\setminus\{0\}$, and
$$
\lim_{t\to \infty} \phi(tx)/t=\lim_{t\to\infty }h_{\tilde K}(1/t,x)=h_K(x).
$$
\end{proof}

Assume now that $u$ is a regular function (of linear growth, in practice the difference of two regular convex functions) such that
\be\label{hypothesis}
V(u,\phi_2, ...\phi_n)=0.
\ee
We need to prove that 
$$V(u,u, \phi_3, ...\phi_n)\leq 0.
$$
For this, we identify $\R^n$ with $H_0=\{(x_0,x)\in\R^{n+1}; x_0=1\}$ and write
$$
0=V(u,\phi_2, ...\phi_n)=\int_{H_0} \ddhash u\w\ddhash\phi_2\w...\ddhash \phi_n\w\dhash x_0/n!.
$$
Let $\mu$ be any smooth gauge function. Applying  Proposition \ref{bdry} with $f=x_0$ we get
$$
\int_{S_\mu^+} x_0 \ddhash u\w\ddhash\phi_2\w...\ddhash \phi_n\w\dhash \mu=0.
$$
This means that the measure, $fdm$,  defined by the form 
$$
\A(u)\w\dhash\mu=\ddhash u\w\ddhash\phi_2\w...\ddhash \phi_n\w\dhash\mu 
$$
 on $S_\mu^+$ annihilates 
the function $x_0$. By the discussion at the end of Section \ref{Asdiffoperator}, this means that we can solve the equation $A_m(v)= f$ on $S_\mu^+$ with boundary values zero. By elliptic regularity, $v$ is smooth up to the boundary, so the 1-homogeneous extension of $v$ is regular, vanishes for $x_0=0$,  and solves
$$
\A(v)=\ddhash u\w\ddhash\phi_2\w...\ddhash \phi_n.
$$
In other words,
$$
\ddhash(u-v)\w \ddhash\phi_2\w...\ddhash \phi_n=0.
$$
By Alexandrov's lemma,
$$
(\ddhash(u-v))^2\w \ddhash\phi_3\w...\ddhash \phi_n\wedge dx_0\w\dhash x_0\leq 0,
$$
on $H_0$, since the functions $\phi_j$ are strictly convex there. Hence
$$
V(u-v,u-v, \phi_3, ...\phi_n)\leq 0.
$$
The proof is then completed by the following proposition, which implies that 
$$
V(u-v,u-v, \phi_3, ...\phi_n)=V(u,u, \phi_3, ...\phi_n).
$$ 
\begin{prop}
Assume $v$ and $w$ are regular and that the 1-homogeneous extension of $v$ vanishes for $x_0=0$. Then 
$$
V(v,w, \phi_3, ...\phi_n)=0.
$$
\end{prop}
\begin{proof}
This is in principle a particular case of Theorem \ref{MAestimate}, since the condition that the homogeneous extension of $v$ vanishes on $x_0=0$ means that its indicator function vanishes. The only difficulty is that $v$ and $w$ are not convex. This can, however, easily be overcome by adding a sufficiently convex function to $v$ and $w$. One then gets
$$
V(v+\phi,w+\phi, \phi_3, ...\phi_n)=V(\phi, w+\phi, \phi_3, ...\phi_n),
$$
which implies the claim.

An alternative proof follows from writing
$$
n!V(v,w, \phi_3, ...\phi_n)=
\int_{H_0} \ddhash v\w\ddhash w\w \ddhash\phi_3\w...\ddhash \phi_n\w\dhash x_0=
$$
$$
\int_{S_\mu^+}\ddhash v\w\ddhash w\w \ddhash\phi_3\w...\ddhash \phi_n\w\dhash x_0,
$$
by Proposition \ref{bdry}, and applying Stokes' theorem. 
\end{proof}

 \begin{remark}\label{comparison}
 Let us  compare our proof to Gromov's proof of the Khovanski-Teissier theorem. 
 The Khovanski-Teissier theorem concerns K\"ahler forms, $\omega_i, i=1, ...n$ (or K\"ahler cohomology classes) on a compact $n$-dimensional  complex manifold $X$. This is the analog of our convex bodies, or rather of our superforms $\ddhash\phi$ where $\phi$ is convex on $\R^n$. 
 In the Khovanski-Teissier theorem one  studies bilinear forms defined by 
 $$
 \Q(\omega, \omega') =\int_X\omega\w\omega'\w \omega_3, ...\omega_n
 $$
 where $\omega, \omega'$ are closed $(1,1)$-forms;  the theorem says that $\Q$ has Lorentz signature. This is the analog of our $Q(u,v)= V(u,v, \phi_3, ...\phi_n)$, where the $\phi_i$ are assumed convex but $u$ and $v$ are not.
It is evident that $\Q$ depends only on the cohomology classes of $\omega, \omega', \omega_i$ in $H^{1,1}(X)$, since the manifold is compact without boundary. This is the analogue of the fact that $Q$ depends only on the indicator of the functions in our setting. Gromov defines an elliptic operator, $\L$ by
$$
\L(u) \omega_1\w...\omega_n=dd^c u\w\omega_2\w...\omega_n
$$
which corresponds to our operators $A_m$ . He then uses a standard solvability result for compact manifolds, that corresponds to our condition for solvability of the equation $A_m(u)=f$ with boundary values zero.

Thus, in our setting the closure of $S_\mu^+$ appears as a compactification of $\R^{n}$, and cohomology classes on $X$ correspond to
$$
[\psi]:= \{ \ddhash\phi; \phi^\circ=\psi^\circ\}.
$$
 This is  what replaces the use of toric compactifications in Gromov's proof.

\end{remark}

\newpage

\section{Valuations on convex bodies.}\label{valuations}
In this section we will apply the formalism  developed  in the previous sections to the theory of valuations on convex bodies. The main result is a representation formula for a large class of valuations  in terms of strongly homogeneous currents; see Theorems \ref{smoothvaluations} and \ref{monchar}.  Our basic background reference is the concise \cite{1Schneider} , see also \cite{2Schneider} and \cite{2Alesker}. In Section \ref{normalcycle}, we will compare our approach to representations of valuations as integrals over  the so called normal cycle (see e.g. \cite{Wannereretal}).

Let $\Gamma$ be an additive commutative semigroup. We will be mainly interested in the cases $\Gamma=\R$ or $\Gamma=\R_+$, but it will also be convenient to let $\Gamma$ be the space of 1-homogeneous convex functions or positive currents occasionally. Let $\mathcal{C}$ be the space of convex bodies in $\R^n$.  A {\it valuation} on $\mathcal{C}$ is map, $\theta$,  from $\mathcal{C}$ to $\Gamma$ such that  
$$
\theta(K\cup L)+\theta(K\cap L)= \theta(K)+\theta(L),
$$
{\it provided that $K\cup L$ is convex}. In case $\Gamma$ has a topology, we say that $\theta$ is continuous if it is continuous for the Hausdorff metric on $\mathcal{C}$. We will focus on valuations that are translation invariant, i.e. satisfy $\theta(K+x)=\theta(K)$ for any $x$ in $\R^n$. The following theorem of Salee, \cite{Salee} is simple but fundamental for the rest of this section.
\begin{thm}
Assume that $K$ and $L$ are convex bodies such that $K\cup L$ is convex. Then (+ denotes Minkowski addition)
$$
K\cup L+ K\cap L=K+L.
$$
\end{thm}
\begin{proof} Assume $K\cup L$ is convex and take $k\in K$, $l\in L$. Then 
$$
x_t:=tk+(1-t)l
$$
lies in $K\cup L$ for any $t\in [0,1]$. Let
$$
I_0=\{ t\in[0,1]; x_t\in K\}, \quad I_1=\{ t\in[0,1]; x_t\in L\}.
$$
$I_0$ and $I_1$ are closed and non empty with union equal to the whole interval $[0,1]$. Therefore they have at least one point $t_0$ in common. Then
$$
k+l=x_{1-t_0} +x_{t_0}\in K\cup L+K\cap L.
$$
Since the opposite inclusion is trivial, this proves the theorem.
\end{proof}
\begin{cor} Assume $K\cup L$ is convex. Then 
$$
h_{K\cup L}+h_{K\cap L}= h_K+h_L.
$$
Thus the map $K\to h_K$ defines a valuation with values in the space of convex functions.

As a consequence, if we set $\omega_K=\ddhash h_K$, then
\be\label{currentvaluation}
\omega_{K\cup L}+\omega_{K\cap L}=\omega_K +\omega_L,
\ee
so the  map $K\to \omega_K$ is a translation invariant  valuation, taking values in the space of $(1,1)$-currents.
\end{cor}
\begin{proof}
The first part of the corollary is an immediate consequence of Salee's theorem, and  (\ref{currentvaluation}) follows by applying the $\ddhash$-operator. Since 
$$
h_{K+a}(x)=h_K(x) + x\cdot a,
$$
$\omega_{K+a}=\omega_K$, which gives  the last statement.
\end{proof}
By the comments immediately after Proposition \ref{modtranslation} any translation invariant valuation must be representable as 
$$
\theta(K)=F_\theta(\omega_K)
$$
for some function $F_\theta$. 
Our main concern in the rest of this section is to write down this function $F_\theta$  explicitly, at least for sufficiently nice valuations. By the corollary, any linear function will give rise to a valuation. We shall now see that linear functions of $\omega_K^p$ also work, and that will turn out later to give almost all valuations.

The next theorem, which is a consequence of a more general result of Alesker (\cite{3Alesker}), generalizes (\ref{currentvaluation}).
\begin{thm}\label{powers}
Assume $K\cup L$ is convex. Then
$$
\omega_{K\cup L}^p+\omega_{K\cap L}^p=\omega_K^p +\omega_L^p
$$
for $p=1, 2 ...n$. Hence all powers
$$
K\to \omega_K^p
$$
define translation invariant current-valued valuations.
\end{thm}
This is a bit more subtle than the case $p=1$; it is certainly not a direct consequence of the corollary. When all the currents involved are smooth (outside the origin) it is not too hard to verify using that $h_{K\cup L}=\max(h_K, h_L)$ and, as it turns out, that $h_{K\cap L}=\min(h_K, h_L)$, but it  is not immediate   to reduce the general case to this.  Alesker proves a more general result on mixed Monge-Amp\`ere measures of plurisubharmonic functions, based on earlier work of Blocki, \cite{Blocki}. Here we will give a proof in our case which uses the results in the previous sections.

Partly as a preparation for this we will discuss briefly the translate of a valuation by a convex body, $A$, defined as
$$
\theta_A(K)=\theta(K+A).
$$
We first verify that the translate is a valuation. The main step is the following lemma.
\begin{lma} Let $K$ and $L$ be convex bodies such that $K\cup L$ is convex, and let $A$ be an arbitrary convex body. Then
$$
(K+A)\cup (L+A)=K\cup L+ A,
$$
and
$$
(K+A)\cap (L+A)=K\cap L+ A.
$$
\end{lma}
\begin{proof} The first claim is just an exercise in logic. It implies that $(K+A)\cup (L+A)$ is convex and so
$$
h_{(K+A)\cup (L+A)}+ h_{(K+A)\cap (L+A)}= h_{(K+A)}+h_{(L+A)}.
$$
Hence, invoking the first claim, 
$$
h_{K\cup L}+h_A+ h_{(K+A)\cap (L+A)}=h_K+h_L+2h_A.
$$
Thus
$$
h_{(K+A)\cap (L+A)}= h_K+h_L -h_{K\cup L}+h_A=h_{K\cap L}+h_A,
$$
which proves the second claim, since the right hand side is the support function of $K\cap L +A$ and the support function determines the convex body.
\end{proof}
It follows from the lemma that
$$
\theta_A(K\cup L)+\theta_A(K\cap L)=\theta(K\cup L+A)+\theta(K\cap L+A)=
$$
$$
=\theta((K+A)\cup(L+A)) +\theta((K+A)\cap(L+A)) =
$$
$$
\theta(K+A)+\theta(L+A)=\theta_A(K) +\theta_A(L).
$$
Hence $\theta_A$ is indeed a valuation. Notice that $\theta_A$ is continuous if $\theta$ is continuous and translation invariant if $\theta$ is translation invariant.

\bigskip

With this, we are ready to prove Theorem \ref{powers}. The starting point is that we know the theorem holds when $p=n$, simply because 
$$
\int \omega_K^n/n!=|K|,
$$
and Lebesgue measure is a valuation.  Since the integrand is a Dirac mass at the origin, 
$$
K\to \omega_K^n/n!
$$
is a current- (or measure-) valued valuations. This implies that its translate by any convex body $A$,
$$
K\to (\omega_K+\omega_A)^n/n!
$$
is a current valued valuation. Expanding and identifying terms we conclude that
$$
K\to \omega_K^p\w\omega_A^{n-p}
$$ 
is a current valued valuation for any $p$ and any convex body $A$.   This means that, with
$$
\Omega =\omega_{K\cup L}^p+\omega_{K\cap L}^p-\omega_K^p -\omega_L^p,
$$
$$
\Omega\w\omega_A^{n-p}=0
$$
for any convex  body $A$. 
Therefore, by Stokes' theorem and Proposition \ref{altdef}
$$
 \int_{|x|=1}h_A \Omega\w(\ddhash  h_A)^{n-p-1}\w\dhash |x|=0,
$$
which implies that $\Omega=0$ by Proposition \ref{mainpoint}. 
\bigskip

A translation invariant valuation is said to be (homogeneous ) of degree $p$ if it scales like
$$
\theta(tK)=t^p\theta(K)
$$
for $t>0$. A fundamental theorem of McMullen (for which we refer to \cite{1McMullen}) says that any translation invariant continuous valuation can be decomposed
$$
\theta=\sum_0^n \theta_p
$$
where $\theta_p$ is of degree $p$. Intuitively, this follows from Taylor expansion of $\theta(tK)$, together with the fact that there are no valuations of homogeneity larger than $n$. 

\begin{prop} Any continuous and translation invariant valuation that is homogeneous of order 1 is Minkowski additive, i.e., if $L+K$ denotes the Minkowski sum of two convex bodies, 
\be\label{orderonevaluation}
\theta(L+K) =\theta(L)+\theta(K).
\ee
Conversely, if $\theta$ is defined on the space of convex bodies and satisfies (\ref{orderonevaluation}), then $\theta$ is a valuation. If $\theta$ is continuous it is homogeneous of order 1.
\end{prop}
\begin{proof} By what we have seen, the translate of $\theta$ by $L$,
$\theta_L(K)=\theta(L+K)$, 
is a valuation. McMullen's theorem implies that
$$
g(t):=\theta_L(tK)=\theta(L+tK)
$$
is a polynomial of degree at most $n$. If $\theta$ is homogeneous of order 1 and continuous, the limit
$$
\lim_{t\to\infty} g(t)/t=\lim_{t\to\infty}\theta(L/t+K)=\theta(K),
$$
so $g$ is of degree 1, and since $g(0)=\theta(L)$, 
$$
\theta(L+tK)=\theta(L)+t\theta(K).
$$
Hence $\theta$ is Minkowski additive. 

Conversely,  Salee's theorem implies that any $\theta$ that satisfies (\ref{orderonevaluation}) is a valuation. Moreover,  (\ref{orderonevaluation}) implies that 
$\theta((p/q)K)=(p/q)\theta(K)$ for integers $p$ and $q$. Hence $\theta $ is of order 1 if it is continuous.
\end{proof}

 Any continuous and translation invariant valuation of order zero, is a multiple of  the {\it trivial valuation}
$$
\theta_0(K)=1,
$$
since, by homogeneity, $\theta(K)=\theta(\{0\})$.
On the other side, all  translation invariant continuous valuations of degree $n$ are multiples of  Lebesgue measure. This result is due to Hadwiger, see \cite{1McMullen}. As for valuations of intermediate order, any mixed volume
$$
V(K[p], A_1, ...A_{n-p}):=V(K,K...K,A_1, ...A_{n-p})=\int \omega_K^p\w\omega_{A_1}\w...\omega_{A_{n-p-1}}/n!
$$
defines a translation invariant continuous valuation \cite{1Schneider}. This  also follows immediately from Theorem \ref{powers}, which moreover implies the following generalization of this fact:

\begin{thm}\label{somevaluations}
Let $\Omega$ be a strongly homogeneous current of bidegree $(n-p,n-p)$ which is weakly positive. Then
$$
\F(\Omega)(K):=\int \omega_K^p\w \Omega/p!,
$$
defines a continuous and translation invariant valuation of degree $p$ with values in $\R_+$. If $\Omega$ is not weakly positive but can be written as a difference of weakly positive currents, we get a $\R$-valued valuation.
\end{thm}
\noindent Note that mixed volumes corresponds to $\Omega=\omega_{A_1}\w...\omega_{A_{n-p}}$; the wedge product of $(n-p)$ currents of bidegree $(1,1)$.

By Proposition \ref{mainpoint}, the map $\F$ is injective, since
$$
\int \omega_K^p\w \Omega=\int_{|x|=1} h_K (\omega_K)^{p-1}\w\Omega\w \dhash |x|.
$$

\medskip

One should note that for the integrand in Theorem \ref{somevaluations} to be well defined for general convex bodies $K$ we need, in order to apply Theorem \ref{crucial},  that $\Omega$ be strong and weakly positive, or can be written as a difference  of such currents. This is the reason why we have assumed to $\Omega$ is (the difference of) strongly homogeneous currents. We shall see shortly that this corresponds to valuations that can be written as  (differences of) {\it monotone} valuations.

\bigskip

The definition of $\F$ can be rewritten in a more suggestive way. 
For convex bodies $K$ let
\be\label{eupphojdtillomega}
e^{\omega_K}=\sum_0^n \omega_K^k/k!
\ee
(higher power than $n$ vanish for degree reasons). With this notation we get
\be\label{valuationasfouriertransform}
 \F(\Omega)(K)=\int e^{\omega_K}\w\Omega,
\ee
since, if $\Omega$ is of bidegree $(n-p,n-p)$, only the term  with $k=p$ in (\ref{eupphojdtillomega}) gives a non-zero contribution to the integral. One advantage of this formulation is that we can treat valuations of different homogeneity simultaneously and allow $\Omega$ to be a sum of terms of bidegree $(p,p)$ for different $p$. 
Formally,  $\F$ behaves like a Fourier transform, taking currents to (translation invariant) valuations. As a first example,
$$
\F(1)(K)=\int \omega_K^n/n!=|K|,
$$
so $\F(1)$ is Lebesgue measure. Furthermore, the trivial valuation is obtained as $\F(\delta_0)$, the `Fourier transform' of the Dirac delta function, viewed as having full bidegree (so that it acts only on the constant term in the series expansion of $e^{\omega_K}$). Note also that
\be\label{character}
e^{\omega_{K+L}}=e^{\omega_K}\w e^{\omega_L},
\ee
so we can think of the exponential of $\omega_K$ as a sort of character on the space of convex bodies with Minkowski addition, taking values in a space of currents. This implies immediately that
$$
\F(e^{\omega_A}\w\Omega)=\F(\Omega)_A,
$$
the translate of $\F(\Omega)$ by $A$. Thus (wedge-)multiplication with a character corresponds to translates on the transform side. In particular, the translate of Lebesgue measure
\be\label{mu_A}
K\to |K+A| =:\mu_A(K)
\ee
by a convex body, $A$, corresponds to the current $\Omega=e^{\omega_A}$, i.e. $\F(e^{\omega_A})=\mu_A$. By a theorem of Alesker, \cite{1Alesker}, the linear span of such valuations is dense in the space of continuous translation invariant valuations.
Viewing $A\to e^{\omega_A}$ as a current valued character on the space of convex bodies with Minkowski addition, this can be viewed as an analogue of the fact that trigonometric polynomials are dense in the space of continuous functions.

We next discuss which valuations that lie in the image of $\F$, i. e. are defined by strongly homogeneous currents. We will use  a fundamental theorem by Alesker, see \cite{2Alesker}, and start with the following definition..

\begin{df}({\it Alesker}) A translation invariant continuous valuation, $\theta$ is {\it smooth} if the map
$$
g\to \theta(gK),
$$
defined for $g\in GL(n,\R)$, is smooth for each fixed convex body $K$. 
\end{df}
Regularizing by taking convolution (with respect to Haar measure) on $GL(n,\R)$ with an approximate identity one sees that smooth translation invariant valuations are dense among continuous translation invariant valuations. Alesker proved a conjecture of McMullen, that linear combinations of mixed volumes are dense in the translation invariant continuous valuations,  by showing that any smooth translation invariant valuation can be represented as a convergent superposition of mixed volumes. This result was later sharpened by Knoerr (\cite{Knoerr}) and van Handel (unpublished) who showed that such a  valuation is even a finite linear combination of mixed volumes. Moreover, these mixed volumes can be taken of the form
$$
V(K[p], A_1, ...A_{n-p})
$$
where $A_i$ are smoothly bounded and strictly convex (they can actually be specified even further). This, together with Theorem \ref{mainpoint},  gives directly a characterization of smooth valuations in terms of currents.
\begin{thm}\label{smoothvaluations} Let $\theta$ be a smooth and translation invariant valuation of order $p$. Then there is a unique smooth and  strongly homogeneous form $\Omega$ of bidegree $(p,p)$ such that
$$
\theta(K)=\F(\Omega)=\int e^{\omega_K}\w\Omega.
$$
Conversely, any such $\Omega$ defines a smooth translation invariant valuation of order $p$.
\end{thm}
The uniqueness of $\Omega$ follows again by Proposition \ref{mainpoint}. Notice also that $\Omega$ is automatically strongly homogeneous, since all forms $\omega_{A_1}\w...\omega_{A_{n-p}}$ defining mixed volumes are strongly homogeneous.
The main point of the theorem is that the class of strongly homogeneous forms is precisely wide enough to furnish a unique representation of any smooth valuation.

We shall now extend the correspondence in theorem \ref{smoothvaluations} to one class of non-smooth valuations. Recall the following definition:
\begin{df} A continuous valuation $\theta$ is monotone  if $K\subseteq L$ implies that $\theta(K)\leq \theta(L)$.
\end{df}
By a theorem of McMullen, a monotone translation invariant valuation is automatically continuous. 
We shall show that there is a one-to-one correspondence between monotone translation invariant  valuations and weakly positive, strongly homogeneous currents:
\begin{thm}\label{monchar}
Let $\theta$ be a monotone  translation invariant valuation of order $p$. Then there is a unique weakly positive, strongly homogeneous current $\Omega$ of bidegree $(n-p,n-p)$ such that
$$
\theta(K)=\F(\Omega)=\int e^{\omega_K}\w\Omega.
$$
Conversely, if $\Omega$ is weakly positive and strongly homogeneous, then $\F(\Omega)$ is a monotone  and translation invariant valuation.
\end{thm}
\begin{proof}
Let the valuation  $\theta$ be monotone and translation invariant. Regularizing by convolution with an approximate identity on $GL(n,\R)$ we get a sequence of smooth monotone valuations $\theta_j$ converging to $\theta$. By the previous theorem we get smooth strongly homogeneous $(n-p,n-p)$-forms $\Omega_j$ such that
$$
\theta_j(K)=\int \omega_K^p\w\Omega_j.
$$
Take $f\geq 0$ smooth and 1-homogeneous. Let  $K$ be smoothly bounded and strictly convex so that $h_K$ is smooth and strictly convex.  Then, for $t>0$ small enough $h_K+tf$ is still convex so there is a convex body, $K_t$ such that 
$$
h_{K_t}=h_K+tf.
$$
Since $h_K+tf\geq h_K$, $K\subseteq K_t$ for $t>0$. It follows that
$$
0\leq (d/dt)|_{t=0} \theta_j(K_t)=\int \ddhash f\w\omega_K^{p-1}\w\Omega_j.
$$
By Stokes' theorem and Lemma \ref{fundamentallemma} this equals
$$
\int_{|x|=1} f\omega_K^{p-1}\w\Omega_j\w\dhash\mu/(p-1)!.
$$
By Proposition \ref{weaklypositive} this implies that $\Omega_j$ is weakly positive. Taking $K$ equal to the unit ball one gets a bound on the trace measure of $\Omega_j$ independent of $j$. 
By Lemma \ref{traceestimate} this gives a bound for all the coefficients of $\Omega_j$. Hence we can take a weak* limit of a subsequence of $\Omega_j$  as $j$ tends to infinity. The limit of any subsequence gives the same valuation so, by uniqueness,  there is actually one limit of the whole sequence, $\Omega$. It is clear that $\Omega$ is weakly positive and strongly homogeneous, since both these properties are preserved under weak limits. 

For the converse direction we take $K\subseteq L$ and define
$$
K_t=tL +(1-t)K.
$$
As in the previous argument we see that $(d/dt)\theta(K_t)\geq 0$ if $\theta=\F(\Omega)$ where $\Omega$ is strictly positive (it is enough to consider $t=0$, since $K_{t+\Delta t}$ is a convex combination of $K_t$ and $L$). Hence $\theta$ is monotone.

\end{proof}

The smooth case of this theorem is closely related to Theorem 2.7 in \cite{Bernig-Fu}, where Bernig and Fu give a characterization of smooth monotone translation invariant valuations in terms of positivity of certain `curvature measures'. 
When $p=n-1$ in the  theorem we get a strongly homogeneous current $\Omega$ of bidegree $(1,1)$. By Proposition \ref{deltaomega},  $\Omega=\ddhash\psi$ for some 1-homogeneous and convex function $\psi$, which is necessarily the support function of some convex body, $L$ (see the beginning of section \ref{volumes}). Then Theorem \ref{monchar} says that
$$
\theta(K)/n=\int \omega_K^{n-1}\w\omega_L/n!= V( K[n-1], L),
$$
the mixed volume of $(n-1)$ copies of $K$ with $L$. This is a theorem of McMullen, \cite{2McMullen}.

At the other extreme, when $p=1$, $\Omega$ is of bidegree $(n-1,n-1)$ and 
$$
\theta(K)=\int \omega_K\w\Omega=\int_{|x|=1} h_K\Omega\w\dhash|x|,
$$
by Lemma \ref{fundamentallemma} and Stokes' theorem. The (super)integral in the right hand side is well defined by Proposition \ref{sumup}, which also shows that 
$$
\int_{|x|=1} h_K\Omega\w\dhash|x|=\int_{|x|=1} h_K d\nu,
$$
for a certain positive measure $d\nu$ on the sphere. If $K$ is a point, then $\theta(K)=0$ by homogeneity, so $d\nu$ has barycenter zero. It then follows from the solution to Minkowski's surface area problem (see Theorem \ref{MAonsphere}) that there is a homogeneous convex function $\phi$,  with 
$$
d\nu =(\ddhash \phi)^{n-1}\w \dhash|x|.
$$
The function $\phi$ is the support function of some convex body, say $L$. 
Applying Lemma \ref{fundamentallemma} and Stokes' theorem again we find that
$$
\theta(K)=\int \ddhash h_K\w (\ddhash h_L)^{n-1}/(n-1)!= n V(K, L[n-1])
$$
is again a mixed volume. This leads to a result of Firey, \cite{Firey}, that any monotone  $\theta$  that is homogeneous of order 1 can be represented as a mixed volume. In the survey \cite{1McMullen}, McMullen states the conjecture that a monotone translation invariant valuation, which is homogeneous of any order $p$, is given by a mixed volume (but adds that evidence in this direction is `meager'). This does indeed seem highly unlikely, since it would mean that the current $\Omega$ given by Theorem \ref{monchar} could be written
$$
\Omega= \ddhash h_{L_1}\w ...\ddhash h_{L_{n-p}},
$$
for some convex bodies $L_i$, but we will not pursue this question now. 

\subsection{An open problem}
Theorems \ref{somevaluations} and \ref{monchar} give representation formulas for valuations in terms of strongly homogeneous currents. But, it is clear from Theorem \ref{powers} that 
$$
K\to\int \omega_K^p\w \Omega,
$$
defines a valuation for other choices of $\Omega$ as well, provided the integrals are well defined. This is the case e.g. if $\Omega$ is strong, and regular, i.e.  its strongly homogeneous extension, provided by Theorem \ref{extending}, is smooth on$\{x_0=0\}$ outside the origin. Indeed, it follows from Proposition \ref{bdry} that
$$
\int_{\R^n} \omega_K^p\w\Omega =\int_{H_0} \omega_K^p\w\Omega\w\dhash x_0=
\int_{S^n_+}  \omega_K^p\w\Omega\w\dhash x_0,
$$
which is finite. (We are using the same letter to denote $\Omega$ and its strongly homogeneous extension. $\omega_K$ is already strongly homogeneous on $\R^n$, and therefore also on $\R^{n+1}$.) Therefore there may be other natural classes of currents -- apart from the strongly homogeneous ones --  that also give a unique representative of, say, smooth valuations. This leads to the following definition.

\begin{df} Let $\theta$ be a smooth, translation invariant, valuation of order $p$. Then
$$
[\theta] =\{\Omega; \text{current of bidegree}\, (n-p,n-p); d\Omega=0, T(\Omega)=0,  \, \text{and}\, \forall K, 
\theta(K)=\int\omega_K^p\w\Omega\}.
$$
\end{df}

By Theorem \ref{somevaluations} the class $[\theta]$ is not empty, since it contains a (unique) strongly homogeneous smooth form, $\Omega_0$. One may verify that the condition that $\Omega\in[\theta]$ means that 
$$
\lim_{t\to\infty} F_t^*(\Omega)=\Omega_0, 
$$
so we may write $[\theta]=[\Omega_0]$ instead, and think of the class as consisting of all $\Omega$ with a given boundary value at infinity. We think of $[\theta]=[\Omega_0]$ as a cohomology class (cf. Remark \ref{comparison} for the case $p=1$).

\medskip

Now let $\omega=\ddhash\phi=dd^c\phi$, where $\phi$ is strictly convex of linear growth on $\R^n$. Then $\omega$ defines a Kähler metric on $\C^n$, and therefore a notion of harmonic form. Assume also that $\phi$ is regular, so that the Kähler metric $\omega$ also extends smoothly `to infinity'. It seems natural to expect that in analogy with Hodge's theorem (cf. \cite{Griffiths-Harris}), in any `cohomology class' $[\theta]$ there should be one unique harmonic representative. If so, this would give an approach Hodge-Riemann bilinear relations for valuations (cf. \cite{AleskerLefschetz} and \cite{Wannereretal}). 

\medskip

This conjecture does hold in case $p=n-1$, by the arguments in section \ref{Gromov}, applied to the case when all $\phi_2, ...\phi_n=\phi$ are the same. Write
$$
\int\Omega\w\omega^{n-1}=c\int\omega^n,
$$
and decompose $\Omega=\ddhash u+ c\omega$. By the arguments in Section \ref{Gromov}, we can solve $\ddhash u=\ddhash v$, with $v$ regular and such that its 1-homogeneous extension vanishes for $x_0=0$. Then $u-v$ is harmonic for the Laplacian defined by the Kähler metric $\omega$, so $\ddhash u -\ddhash v$ is a harmonic form. Since $\omega$ is also harmonic, we see that $\Omega':=\Omega-\ddhash v$ is a harmonic form in the same class as $\Omega$.

\newpage 


\section{Operations on valuations}\label{operationsval}

In this section we will show how various important operations on the space of smooth valuations can be obtained from the corresponding operations on strongly homogeneous forms in section \ref{operations}. These results are due to Alesker, Bernig-Fu and Faifman-Wannerer; our basic reference is \cite{Bernig-Fu2}. 

First we recall that for any convex body, $A$, the translate of Lebesgue measure by $A$, 
$$
K\to |K+A|=: \mu_A(K)
$$
is a valuation, corresponding to the strongly homogeneous current $e^{\omega_A}$, see (\ref{mu_A}) and the remarks immediately after. By a fundamental theorem of Alesker, \cite{1Alesker}, the linear span of these $\mu_A$ is dense in the space of all  continuous translation invariant valuations.

Alesker defined a product structure on the space of  valuations, by first defining the product of $\mu_{A_1}$ and $\mu_{A_2}$, and then extending by continuity. Alesker's work covers general valuations, and even valuations on manifolds, but here, as in \cite{Bernig-Fu2},  we will restrict ourselves to translation invariant valuations. In that case, Alesker's definition is
$$
\mu_{A_1}\cdot \mu_{A_2}(K):= |\Delta(K) +A_1\times A_2|,
$$
where 
$$
\Delta(K)=\{(x,x)\in \R^{2n}; x\in K\}
$$
is the diagonal embedding of $K$ into $\R^{2n}$, and $|\cdot|$ is Lebesgue volume in
$\R^{2n}$.   

We claim that Alesker's product of two general smooth valuations, $\theta_1$ and 
$\theta_2$ is the valuation associated to the convolution of the corresponding two strongly homogeneous forms, $\Omega_1$ and $\Omega_2$, where 
$$
\theta_i=\F(\Omega_i).
$$
To verify this, it suffices to consider the case when $\theta_i=\mu_{A_i}$, so that 
$\Omega_i=e^{\omega_{A_i}}$. Then, by the definition of convolution (see Section \ref{operations}),
$$
\F(e^{\omega_{A_1}}\star e^{\omega_{A_2}})(K)=\int_{\R^{2n}} e^{\omega_{A_1}(x)}
\w e^{\omega_{A_2}(y)}\w \pi^*(e^{\omega_K}),
$$
where $\pi(x,y)=x+y$. 

Now, the support function of $\Delta:=\Delta(K)$ is 
$$
h_\Delta(x,y)=\sup_{u\in K} x\cdot u+y\cdot u=h_K(x+y),
$$
so $h_\Delta=\pi^*h_K$. Since $\pi$ is extended to a complex  linear map between the complexifications of $\R^{2n}$ and $\R^n$, this implies that
$$
\pi^*(\omega_K)=\omega_\Delta, \quad \pi^*(e^{\omega_K})=e^{\omega_\Delta}.
$$
Moreover,
$$
e^{\omega_{A_1}(x)}\w e^{\omega_{A_2}(y)}=e^{\omega_{A_1}(x)+\omega_{A_2}(y)}=e^{\omega_{A_1\times A_2}(x,y)},
$$
since
$$
h_{A_1\times A_2}(x,y)=\sup_{u_i\in A_i} x\cdot u_1+y\cdot u_2=h_{A_1}(x)+h_{A_2}(y).
$$
Hence,
$$
\F(e^{\omega_{A_1}}\star e^{\omega_{A_2}})(K)=
\int_{\R^{2n}} e^{\omega_{A_1}(x)}
\w e^{\omega_{A_2}(y)}\w \pi^*(e^{\omega_K})=\int_{\R^{2n}} e^{\omega_{\Delta +A_1\times A_2}}=|\Delta+A_1\times A_2|,
$$
which proves our claim. 

\medskip

We next turn to the convolution of two valuations, which was introduced by Bernig and Fu, following Alesker's work. This was defined in \cite{Bernig-Fu} by putting
$$
\mu_{A_1}\star\mu_{A_2}=\mu_{A_1+A_2},
$$
when $A_i$ are smoothly bounded convex bodies, and then showing that this extends to a translation invariant smooth valuation in a unique way. We claim that this corresponds to wedge product on the level of strongly homogeneous forms. Indeed, 
$$
\mu_{A_i}=\F(e^{\omega_{A_i}}),
$$
and 
$$
e^{\omega_{A_1}}\w e^{\omega_{A_2}}=e^{\omega_{A_1}+\omega_{A_2}}=e^{\omega_{A_1+A_2}}.
$$

Thus, the convolution of two translation invariant smooth valuations is obtained by taking the wedge product of the corresponding smooth strongly homogeneous forms, whereas the product corresponds to convolution of the two strongly homogeneous forms. The fact that convolution and multiplication are intertwined by the operator $\F$ gives further support to the analogy between $\F$ and Fourier transformation. 

Note finally, that, by the results in Section \ref{operations}, convolution of strongly homogeneous forms is converted to wedge multiplication by the Fourier transform of strongly homogeneous forms. In the same way, Bernig and Fu showed that the product and convolution of valuations are related via Alesker's Fourier transform of valuation, introduced in \cite{AleskerLefschetz} and \cite{AleskerFourier}. Later, it was shown by Faifman and Wannerer, \cite{Faifman-Wannerer}, that Alesker's Fourier transform on valuations can be obtained from Scarfiello's Fourier-Transform of differential forms (cf. Section \ref{operations}), that they rediscovered, via the representation of valuations in terms of the `normal cycle', see Section \ref{normalcycle}. Most likely, this means that the action on valuations defined by the Fourier transform on strongly homogeneous forms also coincides with Alesker's Fourier transform, but I have not checked that.

\newpage

\section{The normal cycle}\label{normalcycle}

A classical approach to  smooth  valuations is to write a valuation $K\to \theta(K)$ as an integral of a certain differential form -- determined by $\theta$-- over the `normal cycle' of $K$. In this section we will translate this formalism to our setting of supercurrents, and show that the two approaches are equivalent, as far as smooth valuations are concerned. Consider  a convex body, $K$,  which is smoothly bounded and strictly convex.

Let us  also assume that the origin lies in $K$ and write
$$
K=\{y; \mu(y)\leq 1\},
$$
where $\mu$ is a gauge function (the Minkowski gauge of $K$).
 Let
$$
S\R^n=\R^n_y\times S^{n-1}_x
$$
be the trivial unit circle bundle over $\R^n$  and put
$$
nc(K)=\{(y,x)\in S\R^n; y\in\partial K, \quad x=n(y)\},
$$
where $n(y)=\partial\mu(y)/|\partial\mu(y)|$ is the outward normal to $\partial K$ at $y$. This is the normal cycle of $K$; it is an $n-1$-dimensional submanifold of the circle bundle. This construction actually extends to general convex bodies, and the normal cycle is then a Lipschitz manifold, but we will limit ourselves to convex bodies that are smoothly bounded and strictly convex here.

Let  $\omega$ be   a smooth differential form on $S\R^n$ of bidegree $(p,k)$ in $(x,y)$;
$$
\omega=\sum_{I J}\omega_{I J}(x)dx_I\wedge dy_J
$$
whose coefficients depend only on $x$. Such forms are called translation invariant, since they are invariant under the natural translations $(y,x)\to (y+a,x)$ on the circle bundle.
By e.g. Theorem 2.2 in \cite{Wannereretal}  any translation invariant  smooth valuation, which has no component that is homogeneous of degree $n$,  can be written as
$$
\theta(K)=\int_{nc(K)} \omega
$$
for some translation invariant $(n-1)$-form $\omega$. Here we will restrict ourselves to such valuations, since the component of degree $n$ is just a multiple of Lebesgue measure and can be treated separately .
We will  now  translate this representation to super-language.

First, there is a map from the space of differential forms  $\omega=\sum \omega_{ I J} dx_I\w dy_J$ of bidegree $(p,k)$ in $(x,y)$ to superforms $\tilde\omega$ in $(x,\xi)$ of bidegree $(p, n-k)$, defined in the following way. 
Let $\beta(\xi, y)=\sum d\xi_i\wedge dy_i$. Then $\tilde\omega$ is defined by
\be\label{definitionofPhi}
\omega\wedge\beta(\xi,y)^{n-k}/(n-k)!=:\tilde\omega\wedge dy.
\ee
( This means that 
$$
\omega\to\tilde\omega=\sum \pm \omega_{I J}(x)dx_J\wedge d\xi_{I^c}.)
$$
Notice that the map $\omega\to \tilde\omega$ commutes with the exterior derivative, so $\widetilde{d\omega}=d\tilde\omega$ and closed forms are mapped to closed forms. Moreover, $\tilde\omega$ is of bidegree $(p,n-k)$ in $(x,\xi)$, so $\tilde\omega$ is of bidegree $(p,p)$ for some $p$ if and only if $\omega$ is of total degree $n$.

\begin{lma}\label{inverse} The inverse of this map is given by
$$
\tilde\omega\wedge \beta(y,\xi)^k/k!=\omega\wedge d\xi(-1)^{n(n+1)/2}
$$
\end{lma}
This lemma can of course be proved by brute computation, but it seems worth while to elaborate it somewhat. Recall the definition of {\it Berezin integration}. If $F$ is a differential form, depending on any number of variables, its Berezin integral with respect to one of them, say $y\in \R^n$, is obtained as follows: Write
$$
F=f\w dy_1\w...dy_n+ ...=f\w dy + ...
$$
where the dots stand for terms that are not of full degree in $y$. 
Then the Berenzin integral of $F$ with respect to $y$ is defined as
$$
BI_y(F):=f.
$$
With this definition, we see that $\tilde\omega$ is the Berezin integral with respect to $y$ of 
$$
\omega(x,y)\wedge\beta(\xi,y)^{n-k}/(n-k)!.
$$
It is convenient to rewrite this in a way that handles all bidegrees at the same time,  as
\be\label{Berezintransform}
\tilde\omega:= BI_y(\omega\w e^{\beta(\xi,y)}):=\Phi(\omega).
\ee
Let us now apply the operator $\Phi$ to $\tilde\omega$, taking the Berezin intgeral with respect to $\xi$. This gives
$$
\Phi^2(\omega)=BI_\xi(\tilde\omega(x,\xi)\w e^{\beta(z,\xi)})=BI_\xi BI_y (\omega(x,y)\w e^{\beta(\xi, y)}\w e^{\beta(z,\xi)})=
$$
$$
BI_\xi BI_y (\omega(x,y)\w e^{\beta(\xi, y-z)}) =BI_\xi BI_y (\omega(x,y)\w \beta(\xi, y-z)^n/n!) =
$$
$$
BI_\xi BI_y (\omega(x,y-z+z)\w \beta(\xi, y-z)^n/n!)=\omega(x,z)BI_\xi BI_y(\beta(\xi, y-z)^n/n!)=
$$
$$
\omega(x,z)(-1)^{n(n+1)/2}.
$$
Thus, up to a sign, $\Phi$ is its own inverse, which gives Lemma \ref{inverse}.

Next we will parametrize  the normal cycle by the unit sphere. For this we use Theorem \ref{polar} from the appendix. The map
$$
y\to \mu(y)=x
$$
is a bijection between the boundary of $K$ and the boundary of its polar, $K^\circ$. Its inverse is 
$$
x\to \partial h_K(x),
$$
where $h_K$ is the support function of $K$. The map
$$
\gamma(y):=\partial \mu/|\partial\mu|=n(y)
$$
is the Gauss map to the unit sphere. The inverse of the Gauss map is also given by 
$$
x\to \partial h_K(x),
$$
since $\partial h_K$ is homogeneous of order zero:
$$
\partial h_K(\partial \mu/|\partial\mu|)=\partial h_K\circ\partial\mu=id.
$$
Now we can parametrize the normal cycle by $S^{n-1}$ by
$$
x\in S^{n-1}\to F(x)=(\partial h_K(x),x)\in nc(K).
$$
In the next computations we will use that
\be\label{formulaone}
F^*(\omega)\wedge d\xi =(-1)^{n(n+1)/2}\tilde \omega\wedge (\ddhash h_K)^k/k!,
\ee
if $\omega$ is of bidegree $(n-k-1, k)$ in $(x,y)$, which follows from
\be
\omega\wedge d\xi=(-1)^{n(n+1)/2}\tilde\omega\wedge \beta^k/k!, \quad  F^*(\beta)=\ddhash h_K.
\ee
Putting this together we get first that 
$$
\theta(K)=\int_{nc(K)}\omega=\int_{|x|=1} F^*(\omega)=(-1)^{n(n-1)/2}\int_{|x|=1} F^*(\omega)\wedge d\xi,
$$
where the last equality is the definition of superintegrals. Finally, using (\ref{formulaone}) we get
$$
(-1)^n\theta(K)=\int_{|x|=1} \tilde\omega\wedge (\ddhash h_K)^k/k!.
$$

This representation of $\theta(K)$ is, however,   not unique. Following \cite{Wannereretal} we say that a  form $\omega$ on the circle bundle is {\it vertical} if $\alpha\w\omega=0$, where $\alpha$ is the contact form 
$$
\alpha=\sum x_idy_i
$$
on $S\R^n$. This translates as follows to the superformalism (recall that $\delta^\#$ is contraction with the vector field $E^\#=\sum x_i\partial/\partial \xi_i$) : 
\begin{lma}\label{vertical} A form $\omega=\sum\omega_{I J} dx_I\w dy_J$  is vertical if and only if $\delta^\#\tilde\omega=0$.
\end{lma}
\begin{proof}
Say $\omega$ is of bidegree $(p,k)$ in $(x,y)$.
Note that $\alpha=\delta^\#\beta$, where $\beta=\sum d\xi_i\w dy_i$.
Since $\alpha\w\omega=0$ if and only if $\alpha\w \omega\w d\xi=0$, Lemma \ref{inverse}, shows that $\omega$ is vertical if and only if  
$$
\delta^\#(\beta)\w\beta^k\w\tilde\omega=\delta^\#(\beta^{k+1})\w\tilde\omega/(k+1)=0.
$$
Since $\delta^\#$ is an antiderivation and $\beta^{k+1}\w\tilde\omega=0$ for degree reasons, this means that
$$
\delta^\#(\tilde\omega)\w\beta^{k+1}=0,
$$
which is equivalent to $\delta^\#\tilde\omega=0$ since $\delta^\#\tilde\omega$ is of degree $(n-k-1)$ in $d\xi_i$.
\end{proof}

If $\eta$ is  a vertical form of bidegree $(p,k)$ in $(x,y)$, then 
$$
f:=\tilde\eta\wedge (\ddhash h_K)^k=0,
$$
since $\delta^\# f=0$ and $f$ is of full degree in $d\xi$. Therefore we can replace $\omega$ by $\omega+\eta$, where $\eta$ is vertical in the computations. By \cite{Wannereretal}, page 9, there is a unique choice of $\eta$ such that $d(\omega+\eta)$ is vertical. Hence we can assume that we have already chosen $\omega$ so that $d\omega=:\tau$ is vertical. By Stokes' theorem,
$$
(-1)^n\theta(K)=\int_{|x|=1} \tilde\omega\w(\ddhash h_K)^k /k!        = \int_{|x|=1} \tilde\tau\wedge (\ddhash h_K)^{k-1}\wedge \dhash h_K/k!,
$$
since $d\tilde\omega=\tilde\tau$. 
Here $\tilde\tau$ is closed  of bidegree $(n-k,n-k)$ and it satisfies $\delta^\#\tilde\tau=0$ since $\tau$ is vertical.  Let $\Omega$ be the pullback of $\tilde\tau$ to $\R^n\setminus\{0\}$ under the map $x\to x/|x|$. Then
$$
d\Omega=0,  \quad \delta^\#\Omega =0,
$$
and, by Stokes', 
$$
\theta(K)= \int_{|x|=1} \tilde\tau\wedge (\ddhash h_K)^{k-1}\wedge \dhash h_K/k!=\int \Omega\wedge(\ddhash h_K)^k/k!
$$
The right hand side here is precisely what appears in Theorem \ref{smoothvaluations}.

\medskip

\newpage

\section{Characterization of strong forms}\label{charstrongform}

Here we will give the proof of Theorem \ref{strongforms1} from section \ref{positivity}. Let 
$$
\Omega=\sum\Omega_{I J} dx_I\w d\xi_J
$$
be a $(p,p)$-form with constant coefficients. Recall that $\Omega$ is `strong' if it can be written as a linear combinations of elementary forms
$$
\alpha_1\w...\alpha_p\w\alpha_1^\#\w...\alpha_p^\#.
$$
Remember also that the operator 
$$
\chi\to T(\chi)=\sum dx_i\w\delta_{\xi_i}(\chi)
$$
annihilates all elementary, and therefore all strong, forms. The next theorem is the converse of this fact.
\begin{thm}\label{theoremstrongform} Let $\Omega$ be a $(p,p)$-form with constant coefficients satisfying $T(\Omega)=0$. Then $\Omega$ can be written as a  linear combination of elementary forms.
\end{thm}
We will deduce this theorem from a  more general result on forms 
$$
\tau=\sum_{|I|+|K|=N} \tau_{I K}dx_I\w dy_K
$$
on $\R^n_x\times\R^n_y=\C^n_{x+iy}$, using the translation between the two pictures from section \ref{normalcycle}. As it turns out, $(p,p)$-forms in the kernel of $T$ correspond precisely to {\it primitive} forms in the sense of Lefschetz (cf. \cite{Griffiths-Harris}), of total degree $n$ on the K\"ahler manifold $\C^n_{x+iy}$ and we can then use a  characterization of primitive forms that we give in Theorem \ref{theoremprimitiveform}. 

Let
$$
L\tau=\beta\w\tau,
$$
where $\beta=\beta(x,y)=\sum dx_i\w dy_i$; this is  the $L$-operator from K\"ahler geometry on $\C^n$ with the K\"ahler form $\beta$. Obviously it increases the total degree by two units. Its adjoint is defined by 
$$
(\Lambda\tau,\tau')=(\tau, L\tau'),
$$
where $(\cdot,\cdot)$ is the standard Euclidean scalar product on forms on $\R^n_x\times\R^n_y$. Explicitly
$$
\Lambda =\sum dy_i\rfloor dx_i\rfloor\,\,\text{or}\,\,  \sum \delta_{x_i}\delta_{y_i}.
$$
Following the terminology of K\"ahler geometry we define the space of primitive forms of degree $N$ as 
$$
P(N)=\{\tau=\sum_{|I|+|K|=N} \tau_{I K}dx_I\w dy_K; \Lambda\tau=0\}.
$$
We also define a form of degree $N$ to be decomposable if it can be written as
$$
\alpha_1\w...\alpha_q\w\alpha_{q+1}^\#\w...\alpha_N^\#,
$$
where $\alpha_i$ are 1-forms on $\R^n_x$, $\alpha^\#=\sum \alpha^jdy_j$ if $\alpha=\sum \alpha^j dx_j$ and all $\alpha_i$ with $i\leq q$ are othogonal to all $\alpha_i$ with $i\geq q+1$. 
Finally we let $D(N)$ be the linear span of the space of all decomposable forms. Note that there are no non-zero decomposable forms in degree $N>n$. 
\begin{thm}\label{theoremprimitiveform} For any degree $N$, 
$$
P(N)=D(N).
$$
In particular, if $N>n$, there are no non-trivial primitive forms.
\end{thm}
\begin{proof}
It is easy to see that any decomposable form is primitive, so $D(N)\subseteq P(N)$. Indeed, we may assume that all the $\alpha_i=e_i$ form a part of an orthonormal basis of the space of 1-forms on $\R^n$ and write
$$
\Lambda=\sum e_i^\#\rfloor e_i\rfloor.
$$
For the converse we use induction and assume the theorem has been proved in all smaller dimensions. Let  $\tau$ be  of degree $N$ on $\R^n_x\times\R^n_y$ and assume $\Lambda\tau=0$. 
It is enough to prove that if $\tau\perp D(N)$ then $\tau=0$, and we may assume that $\tau$ is of pure bidegree $(p,q)$, with $p\geq q$. Let $a$ be a form of bidegree $(1,0)$ and put
$$
E:=a^{\perp_\C}:=\{b_0+ b_1^\#; b_i \, \, (1,0)-forms\, , (a,b_i)=0\},
$$  
the orthogonal complement of $a$ in the space of 1-forms on $\C^n$. We claim that the orthogonal projection of $a\rfloor\tau$ on $\bigwedge^{N-1}(E)$ is primitive on $E$ and orthogonal to all decomposable forms on $E$. It then follows from the inductive assumption that the orthogonal projection of $a\rfloor\tau$ vanishes, which is the main step of the argument.

For the first part, we may assume that $a=e_n$ is the last element in an orthonormal basis of the space of $(1,0)$-forms. Write $\beta=\sum_1^n e_1\w e_i^\#$ and $\beta'=\sum_1^{n-1}e_1\w e_i^\#$. That $e_n\rfloor\tau$ is primitive on $E$ means that
\be\label{primitiveonE}
(e_n\rfloor\tau,\beta'\w \chi)=0
\ee
for any $\chi\in \bigwedge(E)$, and we know that
$$
(e_n\rfloor\tau,\beta\w \chi)= (\tau, e_n\w \beta\w\chi)=0
$$
since $\tau$ is primitive. (\ref{primitiveonE}) therefore follows since 
$$
(e_n\rfloor\tau, e_n\w e_n^\#\w \chi)=(\tau, e_n\w e_n\w e_n^\#\w \chi)=0.
$$
Let next $\chi$ be a decomposable $(N-1)$-form on $E$ of bidegree $(p-1,q)$. This means that 
$$
\chi= \alpha_1\w...\alpha_{p-1}\w\alpha_{p+1}^\#\w...\alpha_N^\#,
$$
where all the $\alpha_i$ are orthogonal to $a$ and $\alpha_i\perp\alpha_j$ if $i<p$ and $j\geq p+1$. For the second part of the claim we need to verify that
$$
(a\rfloor\tau,\chi)=0.
$$
 But this is clear since
$$
(a\rfloor\tau,\chi)=(\tau, a\w\alpha_1\w...\alpha_{p-1}\w\alpha_{p+1}^\#\w...\alpha_N^\#)
$$
and $a\w\alpha_1\w...\alpha_{p-1}\w\alpha_{p+1}^\#\w...\alpha_N^\#$ lies in $D(N)$.

 By the induction hypothesis, this means that the orthogonal projection of  $a\rfloor\tau$ on $\bigwedge^{N-1}(E)$
vanishes, so  as soon as all $\alpha_i$ are orthogonal to $a$,  
\be\label{ortogonal}
(\tau, \alpha_1\w...\alpha_{p-1}\w a\w\alpha_{p+1}^\#\w...\alpha_N^\#)=0.
\ee

Choose any orthonormal basis, $e_i$,  for the space of $(0,1)$-forms, and consider the  basis for $(p,q)$-forms
$$
\alpha_{I J K}:= e_I\w e^\#_J\w dV_K,
$$
where the multiindices are disjoint and 
$$
dV_K=e_{k_1}\w e^\#_{k_1}\w...e_{k_s}\w e^\#_{k_s}, \,s=|K|.
$$
Write
$$
\tau=\sum c_{I J K}\alpha_{I J K}
$$
in this basis.  Then (\ref{ortogonal}) implies that  $\tau$ is orthogonal to all 
$\alpha_{I J K}$ with $I$ non-empty, since we can choose $a$ to be equal to $e_i$ if $i\in I$. Hence  the expansion of $\tau$ contains only terms with $I$ empty. Since $q\leq p$,  all $J$ are also empty, so 
$$
\tau=\sum\lambda_K dV_K.
$$
This must hold for any choice of orthonormal basis. 
Therefore, if $L$ is a multiindex not containing $j$ and $k$, we get that 
$\tau$ is orthogonal, not only to $e_j\w e^\#_k\w dV_L$, but also 
to $(e_j-e_k)\w(e^\#_j+e^\#_k)\w dV_L$. Therefore
$$
\lambda_{\{j\}\cup L}=\lambda_{\{k\}\cup L}.
$$
This implies that all $\lambda_K$ are equal, which immediately implies that $\tau=0$ if $\Lambda\tau=0$.

\end{proof}

The only thing that remains is now to show that Theorem \ref{theoremprimitiveform} implies Theorem \ref{theoremstrongform}. For this we will use the map $\Phi$ from the previous section that takes translation invariant forms $\omega$ on $\C^n_{x+iy}$ to (super)forms on $C^n_{x+i\xi}$, defined by (cf. (\ref{definitionofPhi})), 
$$
\Phi(\omega)\w dy=\omega(x,y)\w e^{\beta(\xi,y)}.
$$
It is immediate that $T(\beta(\xi,y))=\beta(x,y)$, from which follows that
$$
T(e^{\beta(\xi,y)})=\beta(x,y)\w e^{\beta(\xi,y)}
$$
since $T$ is a derivation.
Hence 
$$
T(\Phi(\omega))=\Phi(\beta\w\omega),
$$
so, under the map $\Phi$, elements in the kernel of $L=\beta\w$ on the $\C^n_{x+iy}$-side correspond to elements in the kernel of $T$. This holds in any bidegree. Now note that forms of total degree $n$ on the $\C^n_{x+iy}$-side correspond  to forms of bidegree $(p,p)$ on the super-current side ($\Phi$ maps $(p,k)$-forms to $(p, n-k)$-forms). Thus $(p,p)$-forms in the kernel of $T$ correspond to $n$-forms in the kernel of $L$. But, it is well known (cf. \cite{Griffiths-Harris}) that in degree $n$, the kernel of $L$ equals the kernel of $\Lambda$. The proof is then completed by the easily checked fact that in degree $n$, decomposable forms are mapped under $\Phi$ to elementary forms. 

\begin{remark}

It is a bit curious that K\"ahler theory on the `normal cycle side' of the map $\Phi$ plays a role in this argument, since the idea of superforms is to use K\"ahler methods on the other side, of `superforms'. On the formal level, the isomorphism $\Phi$ behaves somewhat like a mirror symmetry between the K\"ahler manifolds $\C^n_{x+i\xi}$ and $\C^n_{x+i y}$, mapping the vertical diagonal of the Hodge diamond ($(p,p)$-forms) in the first space to the horizontal diagonal ($n$-forms) in the second. 

\end{remark}
\begin{remark} 
The counterpart of Theorem \ref{theoremprimitiveform} also holds in the  setting of forms with the standard complex bigrading, with a similar proof. This can serve as a convenient starting point to prove the Kähler identities, e. g. relating the Hodge $*$-operator to the Lefschetz map (see \cite{Berndtssonnotes}).
\end{remark}


\newpage

\section{Minkowski's surface area measure.}\label{Minkowski}

Let $K$ be a smoothly bounded  strictly convex body, containing the origin in its interior, defined by
$$
K=\{y;\psi(y)<1\}
$$
where $\psi$ is a smooth 1-homogeneous function. Let $dS$ be surface area on the boundary of $K$. Minkowski's surface area measure, $dS^K$,  associated to $K$ is the push-forward of $dS$ under the Gauss map
$$
y\to \gamma(y),
$$
where $\gamma$ is the exterior unit normal  to $\partial K$. 
The surface area measure is thus the measure on the unit sphere satisfying
$$
\int_{|x|=1} f dS^K=\int_{\partial K} f(\partial\psi/|\partial\psi|) dS.
$$

There is a well known formula for $dS^K$ (see \cite{2Schneider}) that in our formalism becomes
\begin{prop}\label{surfaceareaprop}
The surface area measure is given by
\be\label{surfacearea}
\int_{|x|=1} f dS^K= \int_{|x|=1} f (\ddhash h_K)^{n-1}\w\dhash |x|/n!
\ee
where as before $h_K$ is the support function of $K$. 
\end{prop}
\begin{proof}
 By Theorem \ref{polar} in the appendix, the polar of $K$, $K^\circ$, is also smooth and the map
$$
y\to \partial\psi(y)=x
$$
is a diffeomorphism from the boundary of $K$ to the boundary of $K^\circ$, with inverse
$$
x\to \partial h_K(x)=y.
$$
Therefore the inverse of the Gauss map equals
$$
x\to \partial h_K,
$$
for $x$ on the unit sphere, since $\partial h_K$ is homogeneous of order zero. 
By Subsection \ref{subsec2}, surface measure on $\partial K$ is given by the differential form
$$
\sum \frac{\psi_i\w \widehat{ dy_i}}{|\partial\psi|},
$$
where $\widehat{ dy_i}$ is the wedge product of all $dy_j$ except $dy_i$, with a sign so that $ dy_i\w\widehat{ dy_i}= dy_1\w ...dy_n $ .  Hence
$$
\int_{\partial K} f(\partial\psi/|\partial\psi|) dS=\int_{\partial K} f(\partial\psi/|\partial\psi|) \sum \frac{\psi_i\w \widehat{ dy_i}}{|\partial\psi|}= \int_{|x|=1} f(x)\sum x_i\widehat {d(h_K)_i}.
$$
The formula (\ref{surfacearea}) then follows from 
$$
\sum x_i\widehat {d(h_K)_i}\w d\xi_1\w...d\xi_n (-1)^{n(n-1)/2}= (\ddhash h_K)^{n-1}\w\dhash|x|/(n-1)!,
$$
which in turn follows from $\ddhash h_K =\sum d(h_K)_i \w d\xi_i$.

\end{proof}
By the results of section \ref{Bedford-Taylor}, the right hand side of (\ref{surfacearea}) makes sense even if $K$ is not smoothly bounded and strictly convex and furnishes the unique definition of surface area measure of convex bodies with non-empty interior, which is continuous under uniform convergence in $\psi$ and equals (\ref{surfacearea}) when $K$ is smoothly bounded and strictly convex.

We will think of the right hand side of (\ref{surfacearea}) as the Monge- Amp\`ere measure of $h_K$ on the boundary of the ball, and in analogy we define for a  1-homogeneous convex function $\phi$ (which is necessarily the support function of some convex body)  the Monge- Amp\`ere measure of $\phi$ on $S^{n-1}$ as
\be\label{MAdef}
MA_{|x|}(\phi)=(\ddhash\phi)^{n-1}\w\dhash x|_{|x|=1}/n!.
\ee
As in the case of Alexandrov's differential operator on the sphere, we can view the 
Monge- Amp\`ere operator as defined on functions on the sphere, since any function on the sphere has a unique 1-homogeneous extension. At the end of this section we shall write the Monge- Amp\`ere operator explicitly, in terms of only the restriction of $\phi$ to the sphere.

Thus we see that, while Alexandrov's proof of the Alexandrov-Fenchel theorem used Alexandrov's differential operator -- a linear operator of second order --
$$
\A(\phi)=\ddhash\phi\w\Omega_{n-2}
$$
(see Section \ref{Asdiffoperator}), Minkowski's surface area problem concerns the fully non-linear Monge- Amp\`ere operator.

The following theorem has a long history, see \cite{Cheng-Yau}.
\begin{thm}\label{Cheng-Yau} Let $d\nu$ be a positive measure on the unit sphere. 
Then $d\nu$ is the surface area measure of some convex body, $K$, with non-empty interior,
if and only if $d\nu$ has barycenter zero and  is not supported on any hyperplane. If $d\nu$ is absolutely continuous with respect to surface measure on the sphere with a smooth density, then $K$ is smoothly bounded . $K$ is uniquely
determined up to translation.
\end{thm}

We shall now reformulate Theorem \ref{Cheng-Yau} in terms of solvability of the
Monge- Amp\`ere equation on the sphere.
\begin{thm}\label{MAonsphere}
Let $d\nu$ be a positive measure on the sphere. Then the equation
$$
MA_{|x|}(\phi)=d\nu
$$
has a 1-homogeneous convex solution if and only if the barycenter of $d\nu$ is zero,
$$
\int_{S^{n-1}} xd\nu=0.
$$
The solution is uniquely determined up to a linear function, and it is smooth if $d\nu$ has a smooth density.
\end{thm}
\begin{proof}
If the support of $d\nu$ is not contained in any proper linear subspace of $\R^n$, the theorem follows from Theorem \ref{Cheng-Yau} and Proposition \ref{surfaceareaprop}. 

If $d\nu$ is supported in a linear subspace, we may,  by induction, assume the equation can be solved there. Choose coordinates so that the subspace is $V=\{x_n=0\}$. Then there is a convex 1-homogeneous function $\phi(x_1, ...x_{n-1})$ such that
$$
(\ddhash\phi)^{n-1}/(n-1)!\w\dhash|x|=d\nu,
$$
where we consider $d\nu$ as a measure on $S^{n-2}$. Let
$$
\tilde\phi(x_1, ...x_n)=\phi(x_1, ...x_{n-1})+\max(x_n, 0).
$$
Then
$$
\ddhash\tilde\phi=\ddhash\phi+[V]\w\dhash x_n.
$$
Hence
$$
(\ddhash\tilde\phi)^n/n!\w\dhash |x|=(\ddhash\phi)^{n-1}\w [V]\w\dhash x_n\w \dhash|x|/(n-1)!=d\nu,
$$
which completes the proof. 
\end{proof}
\begin{remark}
Note that the condition  in Theorem \ref{Cheng-Yau}, that $d\nu$ is not supported in any proper linear subspace of $\R^n$ is absent in Theorem \ref{MAonsphere}. The reason for this is that Theorem \ref{Cheng-Yau} concerns Minkowski's surface area measure as the push-forward of surface measure on the boundary of $K$, which only exists if $K$ has interior points. Measures that are supported in a proper subspace, $V$,  correspond to convex bodies without interior. If $V$ is the minimal subspace that contains $K$, $K$ has non-empty interior relative to $V$, and $d\nu$ is Minkowski's surface area measure of a convex body in lower dimension.  

\end{remark}

Now recall that, by Theorem \ref{secondpart}, any positive measure, $d\nu$,  on the sphere can be represented by a uniquely determined positive and strongly homogeneous $(n-1,n-1)$-current, $\Omega$, 
\be\label{trace}
d\nu= \Omega\w \dhash|x|/n!,
\ee
and that,  conversely, any such $\Omega$ determines a positive measure.  Moreover, the barycenter of $d\nu$ is zero if and only if the trivial extension of $\Omega$ across the origin is closed. Thus, Theorem \ref{Cheng-Yau} can be reformulated in the following way (cf. to Corollary \ref{cor1}). 
\begin{thm}\label{MAcurrent} For a positive and strongly homogeneous current, $\Omega$,  of bidegree $(n-1,n-1)$ on $\R^n\setminus\{0\}$ there is a 1-homogeneous convex function $\phi$ on $\R^n$, such that
\be\label{MAB}
(\ddhash\phi)^{n-1}=\Omega
\ee
if and only if the trivial extension of $\Omega$ over the origin is $d$-closed. The solution $\phi$ is then uniquely determined up to addition of a linear function, and $\phi$ is smooth if $\Omega$ is smooth.
\end{thm}
\begin{proof}
If  
$$
(\ddhash\phi)^{n-1}\w\dhash |x|/n!=MA_{|x|}(\phi)=d\nu=\Omega\w\dhash|x|/n!,
$$
we get
$$
(\ddhash\phi)^{n-1}=\Omega
$$
by applying $\delta^\#$ to both sides. 
\end{proof}

From the comparison with Alexandrov's operator we also get an alternative way of defining Minkowski's surface area measure.
\begin{prop} Let $L$ be a convex body and $h_L$ its support function. Then, for any convex body, $K$, 
$$
\int_{S^{n-1}} h_L dS^K= V(L,K[n-1]),
$$
the mixed volume of $L$ with $(n-1)$ copies of $K$. 
\end{prop}
\begin{proof}
$$
\int_{S^{n-1}} h_L dS^K=\int_{S^{n-1}} h_L (\ddhash h_K)^{n-1}\w\dhash|x|/n!=
$$
$$
=V(h_L, h_K, ...h_K)=V(L,K[n-1])
$$
by Proposition \ref{altdef}.
\end{proof}

In analogy with definition (\ref{MAdef}) we can also define the Monge- Amp\`ere operator on the boundary of any convex body with non-empty interior. If $L$ is given as
$$
L=\{x; \mu(x)\leq 1\},
$$
where $\mu$ is a gauge function, we put
\be\label{MAdef2}
MA_\mu(\phi):=(\ddhash\phi)^{n-1}\wedge\dhash\mu|_{S_\mu}/n!.
\ee
Again, by Theorem \ref{secondpart}, any measure $d\nu$ on $S_\mu$ can be written
\be\label{trace}
d\nu= \Omega\w \dhash\mu/n!,
\ee
where $\Omega $ is a strongly homogeneous current of bidegree $(n-1,n-1)$, and, as in Theorem \ref{MAcurrent}, the equation
$$
MA_\mu(\phi)=d\nu
$$
is equivalent to
$$
(\ddhash\phi)^n/n!=\Omega.
$$
This means that all Monge- Amp\`ere equations on the `spheres' $S_\mu$ are equivalent. By Proposition \ref{pushforward} we have that if 
$$
MA_{|x|}(\phi)= d\nu
$$
on $S^{n-1}$, and 
$$
MA_\mu(\phi)=d\nu'
$$
on $S_\mu$, then 
$$
\pi_*(d\nu)=|x| d\nu',
$$
where $\pi$ is the radial projection, $\pi(x)=x/\mu(x)$ to $S_\mu$. This means that $|x|d\nu'$ is a sort of generalized surface area measure, defined as the push-forward of the surface measure on $\partial K$,$dS$, under a generalized Gauss map
$$
\gamma_\mu(y)= \partial\psi/\mu(\partial\psi),
$$
that maps $\partial K$ to $S_\mu$.

We will now  write down the Monge- Amp\`ere operator explicitly, without referring to  the 1-homogeneous extension of a function, with special emphasis on the case when $S_\mu$ is a polytope. 

 First we need to take a step back and discuss the real Monge- Amp\`ere measure on a (real) vector space, $V$, of dimension $d$.

 Let $t_1, ...t_d$ be linear coordinates on $V$ and let $\eta=dt_1\w...dt_d$ be the associated volume form. Then, if $u$ is a smooth function $V$,
$$
\det(\frac{\partial^2 u}{\partial t_i\partial t_k}) \eta^2,
$$
is a well defined function with values in $\Lambda^n(V^*)\otimes\Lambda^n(V^*)$, independent of the choice of coordinates. In order to define the Monge- Amp\`ere-operator we want instead an $n$-form on $V$. We therefore choose one fixed (reference) $n$-form, $\eta_0\in \Lambda^n(V^*)$, and define
$$
MA^{\eta_0}(u)=  \det(\frac{\partial^2 u}{\partial t_i\partial t_k}) \eta^2/\eta_0.
$$
We thus get one Monge- Amp\`ere-operator for each choice of $\eta_0$, and different Monge- Amp\`ere measures differ by a multiplicative constant. Equivalently, we can define
$$
MA^{\eta_0}(u)=  \det(\frac{\partial^2 u}{\partial t_i\partial t_k}) \eta,
$$
for any choice of coordinates on $V$ such that $\eta=\eta_0$.

In section \ref{Bedford-Taylor} we defined the Monge- Amp\`ere measure on $\R^n$ as the measure defined by the current
$$
(\ddhash u)^n/n!= \det(\frac{\partial^2 u}{\partial x_i\partial x_k}) dx_1\w ...dx_n\w d\xi_1\w...d\xi_n (-1)^{n(n-1)/2}.
$$
Thus, here, $\eta_0=dx_1\w...dx_n$. 

Let us now consider an affine  hyperplane, $V$, in $\R^n$, defined by an equation $\mu(x)=1$, where $\mu(x)$ is a linear function. We then define the Monge- Amp\`ere measure of a (smooth) function $u$ on $V$ as the measure, $MA_\mu(u)$, defined by the form
$$
(\ddhash u)^{n-1}\w\dhash\mu/(n-1)!|_V.
$$
Unwinding the definition, we see that this means that our reference form,  $\eta_0$, is an $(n-1)$-form on $V$ such that $\eta_0\w d\mu= dx_1\w ...dx_n$. 
In other words
\be\label{2MAdef}
MA_\mu(u)=\det(\frac{\partial^2 u}{\partial t_i\partial t_k})dt_2\w...dt_n,
\ee
where $t_i, i=2, ...n$ is any choice of coordinates on $V$ such that $d\mu\w dt_2\w...dt_n= dx_1\w ...dx_n$. Note that this definition depends on the defining equation.

Let now $P$ be a polytope in $\R^n$; the convex hull of its vertices, $v^1, ...v^N$. We assume $P$ contains the origin in its interior, so that its polar body
$$
P^\circ=\{y; h_P(y)\leq 1\}
$$
is also a compact convex body. Since
$$
h_P(y)=\sup_{x\in P} y\cdot x=\max_i y\cdot v^i,
$$
$P^\circ$ is also a polytope with vertices, say, $u^1, ...u^M$. Taking the polar again we get back $P$, so
$$
P=\{x; h_{P^\circ}(x)\leq 1\}=\{x; \max_j u^j\cdot x\leq 1\},
$$
so $P$ is defined by the gauge function $\mu(x)=\max_j u^j\cdot x =  h_{P^\circ}(x)$.

\bigskip

It follows that on the interior of each facet (face of maximal dimension), $\{x\in P;  u^j\cdot x=1\}$, of $P$, the Monge- Amp\`ere measure defined in (\ref{MAdef}) has the explicit form (\ref{2MAdef}) where $\mu(x)=u^j\cdot x$. We shall now see  that there is also another contribution to the Monge- Amp\`ere measure which is concentrated on the lower dimensional faces.

 \begin{prop}. Let $\mu$ be a gauge function and let $f$ be a smooth function defined in a neighbourhood of  $S_\mu$. Let $F(x)=\mu(x)f(x/\mu(x))$ be the 1-homogeneous extension of $u$ from $S_\mu$. . Then, for $k=1, ...n-1$, 
 \be\label{MAexplicit}
 (\ddhash F)^k\wedge [S_\mu]_s=(\ddhash f +(f-E(f))\ddhash\mu)^k\wedge [S_\mu]_s,
 \ee
 where $E(f)=\sum x_i f_i$, the Euler vector field acting on $f$. In particular, the  right hand side depends only on the restriction of $f$ to $S_\mu$. 
 \end{prop}
 \begin{proof} 
 When $\mu$ is smooth, this follows from a direct computation of $\ddhash F$, using that $F(x)=f(x/\mu(x))\mu(x)$. (Since we are wedging with the supercurrent of integration on $S_\mu$, we may discard all terms containing $d\mu$ or $\dhash\mu$.) We now argue that when $\mu$ is not smooth, we can approximate by smooth functions and pass to the limit. Write $F=F_\mu$ to emphasize the dependence on $\mu$. Note first that the case $k=1$ implies that for some positive constant
 $$
 \ddhash F_\mu + C(\ddhash\mu +\ddhash|x|) \geq 0.
 $$
 Indeed, this holds when wedged with $[S_\mu]_s$ (since $f$ is smooth) and we can then apply $\delta$ and $\delta^\#$. The inequality therefore holds when $\mu=1$, and hence everywhere by homogeneity.  Write $\mu=\lim \mu^i$ in the uniform norm, where $\mu^i$ are smooth. Then $F_{\mu^i}$ tend to $F_\mu$ uniformly. Hence, Proposition \ref{sumup} implies that
 $$
 (\ddhash F_{\mu^i})^k\wedge [S_\mu^i]_s \to(\ddhash F_\mu)^k\wedge [S_\mu]_s.
 $$
 (Since $F_{\mu^i}$ may not be convex, we apply the proposition to $\Omega^i=(\ddhash(F_{\mu^i} +C(\mu^i+|x|)))^k$.) 
 On the other hand, we also have 
 $$
 (\ddhash f +(f-E(f))\ddhash\mu^i)^k\wedge [S_{\mu^i}]_s \to (\ddhash f +(f-E(f))\ddhash\mu)^k\wedge [S_\mu]_s
 $$
 by the same proposition.
 
  \end{proof}

\begin{remark}
 From  the proposition, we see that, when $\mu$ defines a polytope,  the singular contribution to the Monge- Amp\`ere measure  comes from powers of $\ddhash\mu$, multiplied by $f-E(f)$. Thus the singular part vanishes precisely when $f$ is smooth and 1-homogeneous. Note also that the right hand side of (\ref{MAexplicit}) gives an explicit formula for the Monge- Amp\`ere operator applied to a smooth function, that does not involve its homogeneous extension. The case $k=1$ and $\mu=|x|$  also gives an explicit formula for Alexandrov's differential operator if we wedge with $\Omega=\ddhash\psi_3\w...\ddhash\psi_n$. 
 \end{remark}
 
 \newpage
 
 \section{Appendix}
 
 \subsection{Smoothly bounded convex bodies}
 
 Let $K$ be a convex body containing the origin in its interior. The Minkowski gauge of $K$ is defined as 
 \be
 \mu(x)=\mu_K(x):= \inf\{s\in\R_+; x/s\in K\}.
 \ee
 Clearly, $\mu$ is homogeneous of order 1, and  it is also subadditive, since 
 $$
 \frac{x+y}{\mu(x)+\mu(y)}= \lambda\frac{x}{\mu(x)}+(1-\lambda)\frac{y}{\mu(y)},
 $$
 if 
 $$
 \lambda=\mu(x)/(\mu(x)+\mu(y)).
 $$
 Hence $\mu$ is convex, so it is indeed a gauge function in our previous terminology, and $K=\{\mu\leq 1\}$. 
 
 \begin{lma}\label{smoothgauge}   $K$ is smoothly bounded if and only if $\mu$ is smooth outside the origin.
 \end{lma}
 \begin{proof}
 First assume that $\mu$ is smooth. By Euler's formula $\mu=\sum x_i\partial\mu/\partial x_i $, so the gradient of $\mu$ does not vanish outside the origin. Therefore 
 $$
 \partial K=\{\mu=1\}
 $$
 is smooth. 
 
 Conversely, assume that the boundary of $K$ is smooth. Then $\partial K$ can be defined by an equation $\rho(x)=0$, where $d\rho$ does not vanish where $\rho=0$. By the implicit function theorem, the equation 
 $$
 \rho(x/s)=0
 $$
 defines a smooth function $s=\mu(x)$, since 
 $$
 \frac{d}{ds}\rho(x/s) =-\sum x_i\rho_i(x/s)
 $$
 does not vanish when $s=\mu(x)$ (a ray through the origin cannot be tangent to the boundary).
 
\end{proof}
The polar of $K$, $K^\circ$,  is the set where $h_K\leq 1$. In other words, the Minkowski gauge of the polar  of $K$ is the support function of $K$. It follows from the next theorem (or the Hahn-Banach theorem) that  the support function of $K^\circ$ is $\mu$. 
\begin{thm} \label{polar}
If $K$ is smoothly bounded and strictly convex, $K^\circ$ is also smoothly bounded and strictly convex. The map
\be\label{1}
y\to \partial \mu(y)
\ee
is then a diffeomorphism from $\partial K$ to $\partial K^\circ$, with inverse
\be\label{2}
x\to \partial h_K(x).
\ee
\end{thm}
\begin{proof}
By Euler's formulas
$$
y\cdot\partial\mu(y)=1
$$
on $\partial K$. Since $a=\partial\mu(y)$ is a normal to $\partial K$ at $y$,
$$
a\cdot y=\sup_{y'\in K} a\cdot y'=h_K(a).
$$
Hence, $h_K(\partial\mu(y))=1$, so (\ref{1}) maps the boundary of $K$ to the boundary of its polar. It is clear that if $K$ is strictly convex, the Gauss map $y\to \vec{n}(y)$ is injective, so (\ref{1}) is also injective. Strict convexity means that this also holds on the infinitesimal level, so (\ref{1}) is in fact an embedding. 

Thus, if $y\in \partial K$ and $x\in\partial K^\circ$, then 
\be\label{equality}
y\cdot x\leq 1
\ee
with equality if  $x=\partial\mu(y)$. Conversely, for any $x\in\partial K^\circ$, the supremum over $y\in \partial K$ equals one and is attained somewhere. At that $y$, $x$ is a normal to the boundary of $K$, so $x$ is a multiple of $\partial\mu( y)$. In fact, 
$x=\partial\mu(y)$, since $h_K(x)=1=h_K(\partial\mu(y))$. Thus $y\to \partial\mu(y)$ is also surjective, so it is a diffeomorphism to the boundary of $K^\circ$, which is therefore smooth, and equality in (\ref{equality}) holds if and only if 
$x=\partial\mu(y)$.

Therefore we can apply the same argument to $K^\circ$ and conclude that (\ref{2}) maps into $\partial K$. Since we have equality when 
$y= \partial h_K(x)$, it follows that $\partial\mu\circ  \partial h_K(x)=x$, so the two maps are inverses of each other. Finally, since (\ref{2}) is a diffeomorphism, $\partial K^\circ$ is strictly convex.
\end{proof}
 
 \subsection{Currents defined by hypersurfaces}\label{subsec2}
We start with a simple but useful criterion.
 \begin{lma} Let $V=\{\rho=0\}$ be a hypersurface in $\R^n$, where $d\rho\neq 0$ on $V$. Then a differential form, $\alpha$,  of degree $n-1$ on $\R^n$ defines surface measure on $V$ if and only if 
 \be\label{thumb}
 \frac{d\rho}{|d\rho|}\w\alpha=dx:=dx_1\w...dx_n,
 \ee
 if $V$ is oriented so that $\alpha$ is positive on $V$.
 \end{lma}
 \begin{proof}
 Fix a point $p$ in $V$ and choose orthonormal coordinates so that the tangent plane of $V$ at $p$ is $\{x_1=a\}$. Then the unit conormal  of $V$ at $p$ is 
 $$
  \frac{d\rho}{|d\rho|}=dx_1.
  $$
 By definition, the surface element of $V$ at $p$ is the volume form on the tangent plane, i.e. $dx_2\w...dx_n$. 
 Write
 $$
 \alpha= g dx_2\w...dx_n+ dx_1\w\alpha'.
 $$
 Then the restriction of $\alpha$ to $V$ at $p$, $ g dx_2\w...dx_n$, equals the surface element if and only if $g=1$, which means that $dx_1\w\alpha=dx$, i.e. that (\ref{thumb}) holds.
 \end{proof}
 As an example,
 $$
 \alpha= \sum n_i \widehat{d x_i},
 $$
 represents surface measure on $V$,
 where $n_i$ is the $i$:th component of the unit normal, 
 $$
 n_i=\rho_i/|d\rho|,
 $$
  and 
 $\widehat{d x_i}$ is the wedge product of all $dx_j$ except $dx_i$, with sign so that 
 $dx_i\w \widehat{dx_i}=dx$.  When $V$ is the graph of a function $f$, we can also choose 
 $$
 \rho(x)= x_1 -f(x_2, ...x_{n}),
 $$
 and 
 $$
 \alpha ={\sqrt{1+|df|^2}}\,{dx_2\w... dx_n},
 $$
 which gives another classical formula. 
 
 A measure on $\R^n$, $d\nu$, acts on functions so it defines a current of degree $n$.  We denote by 
 $$
 *d\nu=d\nu/ dx,
 $$
 the measure regarded as a current of degree zero.  With this notation we can state the next theorem. 
 
 \begin{thm}
 Let $V$ be a hypersurface in $\R^n$ defined by an equation $V=\{\rho=0\}$, where $d\rho\neq 0$ on $V$. Then the current of integration on $V$, $[V]$, equals
 \be\label{currentformula}
 [V]=\frac{d\rho}{|d\rho|} *dS,
 \ee
 where $*dS$ is surface measure on $V$ regarded as a current of degree zero.
 
 If 
 $$
 V=\{x_1=a\}
 $$
 is an affine hyperplane, then the supercurrent of integration on $V$ is
 $$
 [V]_s=[V]\w \dhash x_1.
 $$
 The last statement means that, if $\Omega$ is a smooth superform of bidegree $(n-1,n-1)$, then 
 $$
 \int_Vr(\Omega)=\int_V \Omega\w \dhash x_1,
 $$
 where $r(\Omega)$ is the restriction of $\Omega$ to (the complexification of) $V$.
 \end{thm}
 \begin{proof} Let first $\alpha$ be a smooth  form of degree $n-1$. Then
 $$
  \frac{d\rho}{|d\rho|}\w\alpha=fdx,
  $$
  for some function $f$. Then, by the lemma, 
  $$
  \int_{V\cap\{f\neq 0\}}\alpha= \int_{V\cap\{f\neq 0\}}f(\alpha/f)= \int_{V\cap\{f\neq 0\}}fdS= \int_{V\cap\{f\neq 0\}}(fdx)*dS.
  $$
  
  The same formula clearly holds where $f=0$,
  so
  $$
  \int_V\alpha=\int fdx*dS=\int [V]\w\alpha,
  $$
  if $[V]$ is defined by (\ref{currentformula}). This proves the first part of our claim.  The last part is a direct consequence of the definitions. 
  
  \end{proof}

\end{document}